\documentclass[11pt,amsfonts, epsfig,reqno]{amsart}
\usepackage{yfonts}
\usepackage{amsmath, amscd, amssymb}
\usepackage[hyperindex,breaklinks]{hyperref}
\usepackage{float}
\usepackage{graphpap}
\usepackage{mathrsfs}
\usepackage{pstricks}
\usepackage{cancel}
\usepackage[mathscr]{eucal}
\usepackage{pstricks}
\usepackage{cancel}
\usepackage{verbatim}
\usepackage{latexsym}
\usepackage{tikz}
\usepackage[top=1in, bottom=1in, left=1.3in, right=1.3in]{geometry}

 \usepackage[arc,all]{xy}

\usepackage{array}

\def\Quintic{Q}
\def\Coeff{{\mathrm{Coef}}}
\def\n{{\mathrm{N}}}
\def\msp{^{\mathrm{msp}}}

\def\nmsp{\mathrm{NMSP}}
\def\ft{{t}}
\def\uev{^{\mathrm{ev}}}
\def\msp{^{M}}
\def\VV{\mathscr E}

\newcolumntype{C}[1]{>{\centering\arraybackslash$}m{#1}<{$}}
\newlength{\mycolwd}                                         
\newlength{\mycolwdm}                                         
\newlength{\mycolwda}                                         
\newlength{\mycolwdb}                                         
\newlength{\mycolwdc}                                         
\newlength{\mycolwdd}                                         
\newlength{\mycolwddd}                                         
\newlength{\mycolwddc}                                         
\def\tp{\varphi}
\def\bp{\psi}

\def\01{{[0,\!1]}}
\def\1i{{[1\!,\infty]}}
\def\locg{\Theta}
\def\bipg{\Lambda}

\def\alp{\alpha}

\def\ffp{\mathrm{pt}}

\usepackage{upgreek}
\usepackage{wasysym}
\usepackage{ esint }

\numberwithin{equation}{section}

\def\cG{{\mathcal G}}
\def\E{{\mathcal E}}
\newcommand{\id}{\mathrm{I}}

\newcommand{\bt}{\mathbf{t}}

\newcommand{\sbigotimes}{%
  \mathop{\mathchoice{\textstyle\bigotimes}{\bigotimes}{\bigotimes}{\bigotimes}}%
}

\def\BCOV{\mathrm{BCOV}}

\def\ba{{\mathbf a}}
\def\bb{{\mathbf b}}
\def\bA{{\text{\bf A}}}
\def\bB{{\text{\bf B}}}

\def\fS{{\mathfrak S}}

\def\bB{{\mathbf B}}
\def\diag{{\mathrm {diag}}}
\def\sR{{ {\mathscr R}}}

\def\ff{{p}}

\def\virt{^{\vir}}

\newcommand{\tw}{\mathrm{tw}}

\def\lsta{_{\ast}}

\def\sO{\mathscr{O}}
\newcommand{\CC}{\mathbb{C}}

\newcommand{\EE}{\mathbb{E}}

\newcommand{\PP}{\mathbb{P}}
\newcommand{\QQ}{\mathbb{Q}}
\newcommand{\RR}{\mathbb{R}}
\newcommand{\ZZ}{\mathbb{Z}}

\newcommand{\vir}{ {\mathrm{vir}} }
\newcommand{\ev}{ \mathrm{ev} }

\newcommand{\cal}{\mathcal}
\def\cC{{\cal C}}

\def\cW{{\cal W}}

\def\dual{^{\vee}}

\def\sta{^\ast}

\newcommand{\Ga}{\Gamma}
\def\sQ{\mathrm Q}

\DeclareMathOperator{\End}{End}
\DeclareMathOperator{\Hom}{Hom}
\DeclareMathOperator{\Aut}{Aut}

\def\tt{\tau}
\def\aA{\mathbb A}

\newtheorem{dummy}{}[section]

\newtheorem{lemma}[dummy]{Lemma}
\newtheorem{proposition}[dummy]{Proposition}
\newtheorem{theorem}[dummy]{Theorem}
\newtheorem{thm}{Theorem}

\newtheorem{corollary}[dummy]{Corollary}
\newtheorem{conjecture}[dummy]{Conjecture}

\newtheorem*{theorem1}{Theorem}

\theoremstyle{definition}
\newtheorem{convention}[dummy]{Convention}
\newtheorem{definition}[dummy]{Definition}
\newtheorem{example}[dummy]{Example}

\newtheorem{remark}[dummy]{Remark}

\def\bd{\bold d}
\def\lloc{{\mathrm{loc}}}

\def\loc{^{\mathrm{loc}}}
\def\npt{\n\mathrm{pt}}
\def\pt{\mathrm{pt}}

\def\beq{\begin{equation}}
\def\eeq{\end{equation}}

\def\Pf{{\PP^4}} \def\bone{{\mathbf 1}}
 
\def\ti{\tilde}
\def\Lam{{\Lambda}}

\def\sub{\subset}

\let\ga=\Ga

\DeclareMathOperator{\sspan}{span}
 \DeclareMathOperator{\Cont}{Cont}
  \DeclareMathOperator{\Contr}{Contri}

\def\sH{{\mathscr H}}

\newcommand{\M}{\overline{\mathcal{M}}}

\renewcommand{\ev}{\mathrm{ev}}

\newcommand{\Res}{\mathrm{Res}}

\def\lra{\longrightarrow}
\def\and{\quad\mathrm{and}\quad}
\def\bl{\bigl(} \def\br{\bigr)}
\def\upmo{^{-1}}

\title[BCOV via NMSP]{BCOV's Feynman rule of quintic $3$-folds}

\author{Huai-Liang Chang}
\address{Department of Mathematics, Hong Kong University of Science and Technology, Hong Kong} \email{mahlchang@ust.hk}
\thanks{${}^1$Partially supported by Hong Kong grant GRF 16301515 and  GRF 16301717}

\author[Shuai Guo]{Shuai Guo}
\address{School of Mathematical Sciences and Beijing International Center for Mathematical Research, Peking University
}
\email{guoshuai@math.pku.edu.cn}
\thanks{${}^2$Partially supported by NSFC grants 11431001 and 11501013}

\author{Jun Li}
\address{Department of Mathematics, Stanford University,
USA; \hfil\newline 
\indent Shanghai Center for Mathematical Sciences, Fudan University, China} \email{jli@stanford.edu}
\thanks{${}^3$Partially supported by   NSF grant DMS-1564500 and DMS-1601211.  }

\begin{document}

\begin{abstract}
We prove the Bershadsky-Cecotti-Ooguri-Vafa's conjecture for  all genus Gromov-Witten potentials of the quintic $3$-folds, by identifying the Feynman graph sum with the $\nmsp$ stable graph sum  via an $R$-matrix action. The Yamaguchi-Yau functional equations (HAE) 
are direct
consequences of the BCOV Feynman sum rule.
\end{abstract}
\maketitle
\vspace{-0.6cm}
{
\tableofcontents}
\vspace{-0.6cm}



\section{Introduction}
 
The landmark work of Witten \cite{Wit2} and Candelas-Ossa-Green-Parkes
\cite{Cand} have initiated a new era of enumerating curves in 
projective (symplectic) manifolds. The mathematical foundation of this theory, called the Gromov-Witten (GW) theory,
was laid by the work of Ruan-Tian \cite{RT} for semi-positive symplectic manifolds, and by Li-Tian and Behrend-Fantechi
\cite{LT, BF} for projective manifolds.

Since then, a central problem is to find the explicit formulae for all genus GW generating functions $F_g$ of the
distinguished CY threefold, the quintic threefold $Q$, among other CY threefolds. 
For genus zero case, $F_0$ is determined by the celebrated mirror
symmetry conjecture \cite{Cand}, which was mathematically proved by Givental \cite{Givental} and by Lian-Liu-Yau \cite{LLY}. 
For higher genus cases, Bershadsky-Cecotti-Ooguri-Vafa (BCOV)  
conjectured a Feynman rule for any CY threefold based on Super-Strings theories \cite{BCOV}. This rule gives an algorithm which effectively calculated the GW potential $F_g$ for all $g>0$, via the lower genus GW-potentials and finitely many 
(yet to be determined) initial conditions.  BCOV's Feynman rule is a cornerstone in the GW theory of CY threefolds. 
  The main result of this paper is (see \S 0.2 for a more explicit statement)

\smallskip
\noindent
{\bf Main Theorem}. {\sl The BCOV Conjecture for quintics 
holds for all genus.}

\medskip

\subsection{Earlier developments}

Using Mirror Symmetry Conjecture, Super-String theorists have computed the genus zero
GW-invariants $F_0$ for many CY threefolds, by effectively evaluating certain
variation of Hodge structures of the mirror CY at large complex structure limits,
following the lead by Candelas et.\,al.. As we will be focusing on
high genus GW-invariants, we will bypass listing any references along this line of development.

The theory developed in \cite{BCOV} 
is fundamental in the study of higher genus GW-invariants of CY threefolds.
For a CY threefold $M$ the authors used path integral to form a B-model topological 
partition function,
which is a non-holomorphic extension of the GW potential $F^{\hat M}_g(q)$ of the mirror 
CY threefold $\hat M$. They further showed that this B-model topological partition function 
satisfies the {\sl holomorphic anomaly equation}.  Solving the equations and  using mirror symmetry, they deduced their (BCOV) Feynman rule.

As will be demonstrated in the later part of the introduction, the BCOV's Feynman 
rule provides an effective algorithm to determine recursively all genus GW potentials of a CY threefold $M$, 
after the finite many ambiguity can be found at each $g$.


Huang-Klemm-Quackenbush in \cite{HKQ} has pushed the work of \cite{BCOV} further,
demonstrating how to effectively find all initial conditions necessary for determining
genus $g\le 51$ GW generating function $F_g$ for the quintic threefold $\Quintic$. 

\smallskip
The task of mathematically proving these formulas (algorithms) for $F_g$ has progressed as well.
In \cite{Kont}, Kontsevich showed how 
to use a hyperplane property of genus zero GW-invariants of $\Quintic$ to relate that of $\Quintic$ 
with that of $\PP^4$, and to evaluate them
using localization via the $\CC\sta$-action on $\PP^4$.
Based on this,
the genus zero formula of Candelas for $F_0$ was proved independently by 
Givental \cite{Givental} and Lian-Liu-Yau \cite{LLY}. 

For $F_1$, 
Li-Zinger developed a theory of reduced genus one GW-invariants of the quintics, which made
using $\CC\sta$ localization to evaluation $F_1$ possible \cite{LZ}.
Shortly after, by overcoming daunting obstacles, 
Zinger in \cite{Zi}, using the results proved by Zagier-Zinger \cite{ZZ}, proved the explicit formula of
$F_1$ 
obtained by BCOV.
It is also worth mentioning that Kim and Lho \cite{KLh}
gave an independent proof of 
BCOV's formula for $F_1$.

Another line of attacks on $F_g$ (for the quintic $\Quintic$) is via using 
the algebraic relative GW-invariants and the degeneration formula of GW-invariants \cite{Li3,Li4}. 
(For the symplectic version, see \cite{Li-R, IP}.)
 In \cite{MP} Maulik-Pandharipande developed
an algorithm, which in principle can evaluate all genus GW-invariants of the quintic $\Quintic$. 
They also proposed an alternative approach, 
which was simplified in \cite{Wu} for genus 2 and 3, after combined with that proposal in \cite{Gat}.
In \cite{FL} via applying localization to a degeneration of $\Pf$ to $Q$, Fan-Lee obtained a 
recursive algorithm for $F_g$, depending on some initial conditions.  
In \cite{GJR}, Guo-Janda-Ruan have proved that a conjectural localization formula via 
compactifying the moduli of stable maps with $p$-fields does give the $F_2$ of the quintic conjectured
in \cite{BCOV}. 


\subsection{BCOV's Feynman rule}
Let $N_{g,d}$ be the genus $g$ degree $d$ GW-invariants of quintics $\Quintic$.
The genus $g$ GW generating function (potential) $F_{g}$ of the $\Quintic$ takes the form:
\beq\label{Fg}
 F_g(\sQ) =  { \begin{cases}
 \quad \frac{5}{6} \log \sQ^3+ \sum_{d\geq 1} N_{g,d} \cdot  \sQ^d,  & g=0\,;
 \\  -\frac{25}{12} \log \sQ+
\sum_{d\geq 1} N_{g,d} \cdot  \sQ^d,  & g=1\,;
 \\
\qquad \sum_{d\geq 0} N_{g,d} \cdot  \sQ^d,  & g>1\,.
 \end{cases}
 } 
\eeq
Here the log term comes from
the degree zero ``unstable" contributions.

The genus zero $F_0$ can be computed by the genus zero mirror symmetry. 
Let 
\begin{align}
 I(q,z) :=\,&
z \,  \sum_{d=0}^\infty q^d
\frac{ \prod_{m=1}^{5d}(5 H +m z)}{\prod_{m=1}^d(H+ m z)^5  }
= \sum_{i=0}^3 I_i(q) H^i z^{1-i}; \nonumber 
\end{align}
be the $I$-function of the quintic threefold and let $J_i(q):=I_i(q)/I_0(q)$ for $i=0,\cdots,3$.

The mirror theorem \cite{Givental,LLY} is
\beq 
F_0(\sQ) = {\textstyle \frac{5}{6}} \big(\log \sQ^3-J_1(q)^3\big)+ {\textstyle \frac{5}{2} } \big({J_1(q)J_2(q)}- {J_3(q)\big)},\quad \text{ with } \sQ = q \exp  J_1(q).  \nonumber
\eeq

\medskip

We now state BCOV's conjecture. Let the three ``propatators" introduced in 
\cite{BCOV} be $T^{\tp\tp}, T^{\tp}$ and $T  \in \QQ[\![q]\!]$, which are essentially the genus zero invariants (the explicit formulae are given in \eqref{BCOVgauge}, see also Remark \ref{rmgauge}). 
We define 
\beq  \label{defnI11}
Y=(1-5^5q)^{-1} \and I_{11}=1+ q\frac{d}{dq} J_1.
\eeq
For $2g-2+m>0$, we introduce the ``normalized" GW potential following \cite{YY}
\begin{align}\label{Pgm}
P_{g,m}:= \frac{ (5Y)^{g-1}  (I_{11})^m}{ (I_0)^{2g-2}   }  \big( \sQ\frac{d}{d\sQ} \big)^m  F_g \Big|_{\sQ = q \exp  J_1 } \in \QQ[\![q]\!]. 
\end{align}

Let $H_\bB\!:=\!\sspan\{\bp,\tp\}$ be the state space, which is a linear span of the
formal variables $\tp$ and $\bp$. Let $G_{g}$ be the 
set of genus $g$ stable (dual) graphs. For each $\Gamma\in G_{g}$, we define a contribution $\Cont_\Gamma^\BCOV$ via the following construction:
\begin{itemize}{\sl
\item[(i)] at each edge, we place a bi-vector in $\QQ[\![q]\!]\otimes H_\bB^{\otimes  2}$:
$$T^{\tp\tp}\,  {\tp} \otimes {\tp}
 +T^{\tp} \,   ( {\tp}  \otimes \bp +   \bp \otimes {\tp} )
+ T  \,  \bp \otimes   \bp ;
$$
\item[(ii)]  at each vertex, we place a multi-linear map $H_\bB^{\otimes (m+n)} \longrightarrow \QQ[\![q]\!]$:
\begin{align}\label{Pgmn}\qquad
{\tp}^{\otimes m}\otimes \bp^{\otimes n} \longmapsto   P_{g,m,n}: =\begin{cases}
(2g\!+\!m\!+\!n\!-\!3)_n \!\cdot  P_{g,m}  & \text{ if  }\  2g-2+m>0\\
\quad (n-1)!  \big( \frac{\chi}{24}-1\big) &  \text{ if  }\  (g,m) = (1,0)
\end{cases}
\end{align}
where $\chi = -200$ and $(a)_k:=a(a-1)\cdots(a-k+1)$;
\item[(iii.)]  we  apply the map (ii) at each vertex to the placements (i) at the edges incident to that vertex; we define $\Cont_\Gamma^{\bB}$ to be the product over all vertices and edges.
}
\end{itemize}
Later, we will simply call (iii) the composition rule.

\medskip

The {BCOV's} Feynman rule Conjecture, in the case without insertions,   is:
\begin{theorem}\label{BCOV0}  
For $g>1$,   the Feynman graph sum
\beq  \label{fgsum}
f_g^\BCOV :=    \sum_{\Gamma \in G_{g}}  \frac{1}{|\Aut (\ga)|}\Cont_\Gamma^\BCOV  ,
\eeq
which a priori is a power
series in the Novikov variable $q$, is a polynomial in $X:= \frac{-5^5q}{1-5^5 q}$ of degree at most $3g-3$.
\end{theorem}
 This polynomial is called the ambiguity in the physics literature. Once it is known, the $F_g$ is
determined entirely by the lower genus $F_{h<g}$. In Section~\ref{quantization}, we will represent it via the quantization of a symplectic transformation on the ``small" phase space $H_\bB$.

\begin{remark} \label{rmgauge}
{In \cite{BCOV}, there are also freedoms in choosing the propagators, which were called ``gauge". 
They conjectured that,  the Feynman rule will hold with a suitable choice of gauge. In \S \ref{BCOVrule}, we give the most general freedoms for the gauges \eqref{propgauge} and their explicit roles in propagators \eqref{bcovgens}. 
 For this reason we regard  Theorem \ref{thm1} (given in \S \ref{BCOVrule}) as the most general form of  BCOV's conjectures, with insertions, and with gauges \eqref{propgauge}.  
 } 
 \end{remark}

\subsection{The algorithm} \label{introalgorithm}

The BCOV's Feynman rule provides a recursive algorithm for determining $F_g$, up to finite ambiguity. 
The set $G_g$ contains a distinguished ``leading" graph $\Ga_g$ 
which has only a single genus $g$ vertex with contribution  $P_g$. Others $\ga\in G_g\backslash\{\Ga_g\}$
contribute to products of 
$F_{g'<g}$ and propagators $\{T^{\tp\tp}, T^{\tp} , T\}$, which are explicitly computable assuming all
$F_{g'<g}$ are known. Then \eqref{fgsum} implies that 
$$P_g=-\sum_{\ga\ne \ga_0\in G_g}  \frac{1}{|\Aut (\ga)|}\Cont_\Gamma^\BCOV
 +f^\BCOV_{g},\and \deg_X f^\BCOV_{g}\le 3g-3.
$$ 
This way, $F_g$ is determined explicitly once we have found 
$f^\BCOV_{g}$, which has $3g-3$ unknown coefficients,
as the constant term is given by the (known) degree zero GW-invariants. 

\smallskip
To illustrate this, we apply the algorithm to find the genus two potential $F_2$ (A more detailed computation can be found in  Appendix \ref{examples}). There are $6$ stable graphs in $G_2\backslash\{\Ga_2\}$:
{\small
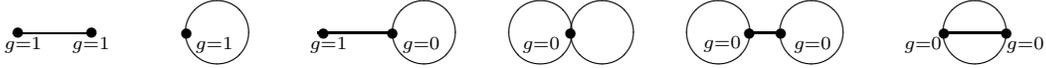
\begin{align*}
  \xy
 (9.5,0.8); (20,0.8), **@{-}; (20,0.8)*+{\bullet};
(11,-0.8)*+{\displaystyle{{}_{g=1}}}; (20,-0.8)*+{\displaystyle{{}_{g=1}}};
(10.2,0.8)*+{\bullet};
\endxy
 \qquad 
 \xy
  (20.5,-0.2)*+{\bullet};
(24.3,-1.6)*+{\displaystyle{{}_{g=1}}};
(24.6,0)*++++[o][F-]{}
\endxy 	
  \qquad 
 \xy
 (9.5,0); (19,0), **@{-}; (19.4,-0.2)*+{\bullet};
(11,-1.6)*+{\displaystyle{{}_{g=1}}}; (23.3,-1.6)*+{\displaystyle{{}_{g=0}}};
(23.6,0)*++++[o][F-]{}  ; (10.2,-0.2)*+{\bullet};
\endxy   \qquad
  \xy
  (19.8,-0.2)*+{\bullet};
(16,-1.6)*+{\displaystyle{{}_{g=0}}};
(24,0)*++++[o][F-]{}  ;(15.8,0)*++++[o][F-]{};
\endxy    \qquad \xy
 (9.5,0); (14,0), **@{-}; (13.7,-0.2)*+{\bullet};
(18,-1.6)*+{\displaystyle{{}_{g=0}}};
(6,-1.5)*+{\displaystyle{{}_{g=0}}};
(17.8,0)*++++[o][F-]{}  ;(5.4,0)*++++[o][F-]{};(9.5,-0.2)*+{\bullet};
\endxy  
 \qquad  \xy
 (9.5,0); (18,0), **@{-}; (18.2,-0.2)*+{\bullet};
(20.8,-1.6)*+{\displaystyle{{}_{g=0}}};(7.2,-1.6)*+{\displaystyle{{}_{g=0}}};
(14,0)*++++[o][F-]{}  ; (9.9,-0.2)*+{\bullet};
\endxy 
\end{align*}} 
The BCOV's Feynman rule for $g=2$ gives us
\begin{align*}
 &- P_2  \ =\  {\small \text{$ \frac{1}{2} \Big( T^{\tp\tp} P_{1,1}^2 +2\, T^{\tp} P_{1,0,1}P_{1,1} + T P_{1,0,1}^2 \Big)+\frac{1}{2}\Big( T^{\tp\tp} P_{1,2} +T^{\tp} P_{1,1}+T P_{1,0,2} \Big) $}} \nonumber \\
 &\quad {\small \text{$ + \frac{1}{2} \Big(  (T^{\tp\tp} )^2 P_{1,1}+  T^{\tp\tp}  T^{\tp} P_{1,0,1}  \Big)+\frac{1}{8}\Big( (T^{\tp\tp} )^2 P_{0,4}+ 4\, T^{\tp} \Big)+\frac{1}{8}\, (T^{\tp\tp} )^3+\frac{1}{12}\, (T^{\tp\tp} )^3+f^\BCOV_{2}$}}. \quad
\end{align*} 
As $N_{2,1}$, $N_{2,2}$ and $N_{2,3}$ can be calculated classically (see  Appendix \ref{lowdegreeGW}), by using the definition of $P_{1,0,n}$ in \eqref{Pgmn} and the genus one mirror formula \cite{Zi,KLh,CGLZ}  \footnote{ See also Example~\ref{g1n1} for a short proof of the genus $1$ mirror formula via BCOV's rule.}
$$ {\small \text{$
 P_{1,1} = - \frac{28}{3}\cdot q\frac{d}{dq} (\log I_{0})-\frac{1}{2}  T^{\tp\tp} -\frac{1}{12}X-\frac{107}{60} $}}
$$
we prove the genus two mirror formula conjectured in \cite{BCOV}:

\begin{theorem} \label{genus2formula}
Let $B:=q\frac{d}{dq} (\log I_{0})$. The genus two  GW potential $F_2$ of quintics is
{\small
\begin{align*} 
F_2 = & \  \frac{-I_0^2}{5(1-X)} \!\left[ {\frac {350\,{T}}{9}}+ \!\Big(
{\frac {25\,X+535}{36}}\!+\!{\frac {700\,B}{9}}\!+\!{\frac {25\,
T^{\tp\tp}}{6}} \Big) T^{\tp}\!+{\frac {5}{24}{(T^{\tp\tp})}^{3}} +{\frac {25\,
B+X+4}{6}}  {(T^{\tp\tp})}^{2} \right. \\
& \left.+ \Big({\frac 
{65\, X^2+46\,X+2129}{1440}}\!+  \!
{\frac {25\,X+535}{36}}\, B\!+\!{\frac {350}{9
}\,{B}^{2}} \Big) T^{\tp\tp}\!+\!\Big({\frac {{X}^{3}}{
240}}\!-\!{\frac {113\,{X}^{2}}{7200}}
\!-\!{\frac {487\,X}{300}}\!+\!{\frac{11771}{7200}}\Big) \right].
\end{align*}}
\end{theorem}
\medskip

\subsection{The strategy of the proof}

Our proof of {BCOV's Feynman} rule is via applying $\nmsp$ theory, which was introduced in \cite{NMSP1}. In its sequel 
\cite{NMSP2}, the property of $\nmsp$ theory was further studied, and the conjecture on the Yamaguchi-Yau ring was proved. 
In this paper, we will continue to use the results proved in \cite{NMSP1,NMSP2}.

We begin our paper with stating the generalized {BCOV's} Feynman rule (Theorem 1). 
We then introduce a 
parallel Feynman rule, derived from the $\nmsp$ theory, which we call the $\nmsp$ Feynman rule (Theorem 2). 
We then state our Theorem 3, which says that the generalized {BCOV's} Feynman rule is equivalent to the $\nmsp$ Feynman rule.

In the first half of the paper, we {will build} the mentioned $\nmsp$ Feynman rule and prove 
Theorem 2. To build the $\nmsp$ Feynman rule, we use
the $\nmsp$ theory and its $\CC\sta$ localization. As is shown in \cite{NMSP1}, the organization of the $\CC\sta$ localization of
the $\nmsp$ theory is governed by a class of graphs, whose vertices are categorized into level $0$, $1$ and $\infty$; 
and among these three vertices, level $0$ vertices are \textit{GW-invariants of the quintic $Q$}. The key 
is that \textit{the edges connecting level 0 vertices contribute (in $\nmsp$ theory) exactly the BCOV propagators}. 
This leads us to introduce the ``$\nmsp$-$[0]$ theory", given by {summing the contributions from}
graphs in $\nmsp$ theory whose vertices are of level 0.

In \cite{NMSP2}, we have identified the $\nmsp$-$[0,1]$ theory (constructed in \cite{NMSP2}) 
with the $R$ matrix action on the CohFT of the union of 
the quintic $Q$ with $\n$ points. We have proved the polynomiality of  the $\nmsp$-$[0,1]$ theory there.
Based on these results, we prove the polynomiality of ``$\nmsp$-$[0]$ theory"  in Proposition \ref{Sgl}
 via Lemma \ref{bigstar} (proved in \S\ref{polyof0}).
We also identify (via the factorization \eqref{0-X-A}) the $\nmsp$ Feynman rule  with the polynomiality of ``$\nmsp$-$[0]$ theory", with the same controlled degree bound $3g-3$. 
So the $\nmsp$ Feynman rule is proved simultaneously.

In the second half of the paper, 
we write the generalized BCOV's Feynman rule in the form of the operator  quantization of the symplectic 
transformation $R^\bB$ on the $B$-model state space $H_\bB$. Here the $R^\bB$-matrix is exactly the restriction of the $R^\bA$-matrix that appears in the $\nmsp$ Feynman rule(\S \ref{RARB}).  We then introduce the ``modified" Feynman rule via the factorization of the quantization action(\S \ref{BmodelE}). Compared with the $\nmsp$'s modified rule (\S \ref{AmodelE}),  we prove that the generalized BCOV's Feynman rule is equivalent to the $\nmsp$ Feynman rule,   hence proving Theorem 3. Theorem 2 and 3 imply Theorem 1 directly, and provide a mathematical proof of the BCOV's Feynman rule.

 {As a further remark (in \S7), we will show that 
Yamaguchi-Yau's {functional} equations \eqref{YYHAE1} and \eqref{YYHAE2} for quintic Calabi-Yau threefold \footnote{ They are called Holomorphic  Anomaly Equations in \cite{LhoP, GJR18Dec}. }, can be derived from the operator formalism of the BCOV Feynman sum rule (Theorem \ref{YYeqn12}).
Indeed, J. Zhou and the authors of this paper will give a geometric proof that for a general Calabi-Yau threefold its BCOV Feynman rule 
implies Yamaguchi-Yau functional equations \eqref{YYHAE1} and \eqref{YYHAE2} (cf. \cite{YYFE}). For the quintic Calabi-Yau threefold, we include here a direct proof. 
}

\smallskip

The paper is organized as follows. In \S 1, we make precise the statements of Theorem 1, 2 and 3. In \S2, 
we recall the notion of CohFTs and $R$ matrix actions. In \S3 and \S4, we prove the $\nmsp$ 
Feynman rule, the Theorem 2. In \S 5 and \S 6, we prove the equivalence of two 
Feynman rules, which is Theorem 3. 
 In \S7, we verify  the Yamaguchi-Yau equations, and apply our main theorems to give lower  genus $F_{g\leq 3}$.

\smallskip

We believe that this approach should provide Feynman rules for complete intersection 
CY threefolds in products of weighted projective spaces. This is our work in progress.

\vspace{1cm}

 \section{The Main Theorems}
 
 In this section we give the statement of three theorems, respectively (i) generalized BCOV's Feynman rule, (ii) the $\nmsp$ Feynman rule, and (iii) their equivalence. We will prove (i) by showing (ii) and (iii), in next sections.
 

\medskip

Following \cite{YY}, let $D: =  q\frac{d}{dq}$ and we introduce the following generators 
\footnote{ Recall $I_{11}$ was defined in \eqref{defnI11}. Here our choice of generators are slightly different from that in \cite{YY} and \cite{HKQ}, which comes out
 naturally from our approach through $A$-model theory.}
\begin{equation*}
  A_p:=\frac{D^p   I_{11}(q)}{  I_{11}(q)},\quad
 B_p:= \frac{D^p I_0(q)}{I_0(q)} ,\quad
 X:= \frac{-5^5q}{1-5^5 q} .
\end{equation*}
It is proved in \cite{YY} that the generators $A_{k\geq 4}$ and  $B_{k\geq 2}$ all lie in the {\sl ring of five generators}
$$\sR=\QQ[A_1,B_1,B_2,B_3,X].
$$
Namely, this ring is closed under the differential operator $D$.  Indeed, it is proved \footnote{ Their proof relies on a ``non-holomorphic completion" of the generators. For an algebraic proof of the first equation see \cite[Lemma 3]{ZZ}. The second equation follows directly from the Picard-Fuch equation.}
\beq \label{YYrelation}
\text{\small $A_2 = 2B^2-2AB-4B_2- X\!\cdot \!\Big(A+2 B+\frac{2}{5}\Big),\quad B_4 = -X \!\cdot \! \Big(2\,{B_3}+ \frac{7}{5}\,{B_2}+\frac{2}{5}\,B+ {\frac{24}{625}}\Big) $} .
\eeq

In \cite{NMSP2},  the {\sl finite generation property} raised in \cite{YY} is proved. We state it now.

\begin{theorem1} [Polynomial structure]
For $2g-2+m>0$, 
$P_{g,m}$ lies in the ring $\sR$.
\end{theorem1}

\subsection{BCOV's Feynman rule with insertions in general gauge} \label{BCOVrule}

We now introduce a Feynman rule generalizing that in \cite{BCOV}\footnote{ See Appendix \ref{origBCOV} for a statement of this Feynman rule in the original language, and the relations with our version.  See also \S \ref{quantization} for the Feynman grasph sum as a geometric quantization.}. First we introduce the propagators
\begin{multline} \label{bcovgens}
 E^\cG_{\bp}:= B_1+c_{1a},\quad
     E^\cG_{\tp\tp}:=A+2\,B_1 +c_{1b} ,\quad
 E^\cG_{\tp\bp}:=-B_2-c_{1b} \, B_1 -c_2, \quad \\
\     E^\cG_{\bp\bp}:= -  B_3+(B-X)\cdot B_2 -\frac{2}{5}\,B_1X  + c_{1b} \,B_1^2-2 c_2 \, B_1+ c_3,   \  \qquad
\end{multline}
which depend on the ``gauge" $\cG:=(c_{1a},c_{1b},c_2,c_3)$, where 
\beq \label{propgauge}
c_{1a}, \ c_{1b} \in \mathbb Q[X]_1,\quad c_2 \in \mathbb Q[X]_2,\and c_3 \in \mathbb Q[X]_3.
\eeq
 Here we denote by $\QQ[X]_d$ the set of polynomials of degree $\le d$. In the papers \cite{BCOV,YY}, the propagators were chosen with the following special ``gauge"   
 \beq \label{BCOVgauge} 
 (T^{\tp\tp}, T^{\tp}, T):=  (E^\cG_{\tp\tp}, E^\cG_{\bp\tp}, E^\cG_{\bp\bp})|_{
 (c_{1b},c_2,c_3)= (\frac{3}{5}, -\frac{2}{25},  - \frac{4}{125})}.
 \eeq

\medskip

Let $G_{g,n}$ be the set of stable graphs of genus $g$ and $n$ legs. Let $H_\bB \!:=\! \sspan\{{\tp},  \bp\}$ be the $B$-model state space. 
We define the $\bB$-master potential via the graph sum formula
$$
f^{\bB,\cG}_{g,m,n} =  \langle  {\tp}^{\otimes m},  \bp^{\otimes n} \rangle^{\bB,\cG}_{g,m+n}:= {\small \text{$ \sum_{\Gamma \in G_{g,m+n}}$}}  \frac{1}{|\Aut (\ga)|}\Cont_\Gamma^{\bB,\cG}( {\tp}^{\otimes m},  \bp^{\otimes n}).
$$
Here for each $\Gamma\in G_{g,n}$, the contribution $\Cont_\Gamma^{\bB,\cG}$
is defined via taking the product  through all vertices by
the composition rule by the following placements:
\begin{itemize}
\item at each of the first $m$ or last $n$ legs, we place a vector 
\begin{equation*}
\qquad\qquad
\tp- E_{\bp}^\cG\! \cdot  \bp \ \text{ or  } \  \bp 
 \,\, \quad \text{respectively}
 ; \qquad
\end{equation*}
\item at each edge, we place a bi-vector
$$
\qquad  \quad E^\cG_{\tp\tp}\,  {\tp} \otimes {\tp}
 + E^\cG_{\tp\bp} \,   ( {\tp}  \otimes \bp +   \bp \otimes {\tp} )
+ E^\cG_{\bp\bp}   \,  \bp \otimes   \bp   \,; \quad
$$
\item  at each vertex, we place a multi-linear map :$H_\bB^{\otimes (m+n)} \longrightarrow \QQ[\![q]\!]$:
\begin{align*}
\qquad {\tp}^{\otimes m}\otimes \bp^{\otimes n} & \mapsto \big\langle  {\tp}^{\otimes m}, \bp^{\otimes n}\big\rangle^{Q,\bB}_{g,m+n} \!:= P_{g,m,n},
\end{align*}
where we recall $ P_{g,m,n}$ is  defined in \eqref{Pgmn}.
\end{itemize}
\smallskip


\begin{thm}[BCOV's Feynman rule]\label{thm1}
For any gauge satisfying  \eqref{propgauge},
we have the following polynomial structure statement
$$
f^{\bB,\cG}_{g,m,n} \in \mathbb Q[X]_{3g-3+m}.
$$
\end{thm}
By taking $g>1$, $m-n=0$ and picking the special gauge \eqref{BCOVgauge} in Theorem \ref{thm1}, one recovers Theorem~\ref{BCOV0} in the introduction.


 \begin{convention}
In this paper  $\bp$ is the psi class of $\M_{g,n}$, namely, the ancestor class.
\end{convention}

\begin{remark} \label{genus1correction}
After identification ${\tp} = I_0I_{11}H$, the correlation function $P_{g,m,n}$ matches   the normalized GW correlator of    quintic CY threefolds.
Namely  let $Y:=1-X$, then
\begin{align*} 
P_{g,m,n}= \frac{(5Y)^{g-1}}{I_0^{2g-2+m}} \!\left<\tp^{\otimes m}, \psi^{\otimes n} \right>_{g,m+n}^Q
\end{align*}
except for the ``exceptional" cases when $(g,m)=(1,0)$. Here
$$
 \!\left< \tau_1 \bp_1^{k_1},\cdots,\tau_n \bp_1^{k_n} \right>_{g,n}^Q:=  \!\sum_d  \! \sQ^d \!\!  \int_{[\M_{g,n}(Q,d)]\virt} \!\!\!\!\!\ev_1^*(\tau_1) \bp_1^{k_1} \!\cup\!\cdots\!\cup\!\ev_n^*(\tau_n) \bp_n^{k_n}.
$$
For the exceptional cases, the BCOV's  correlators
$$\textstyle  P_{1,0,n}=
(n-1)!(\frac{\chi}{24}-1)$$
  differ  from the corresponding GW correlators  $\big<\psi^{\otimes n}\big>^Q_{1,n} =(n-1)! \frac{\chi}{24}$ by a ``correction term"
$-(n-1)!$.  This term is mysterious from the $A$-model side. In the proof of Theorem \ref{thm3}, we
will see how this term comes into play.
\end{remark}

\subsection{The $\nmsp$ Feynman rule} \label{NMSPrule}
Let $H_\bA$ be the $A$-model state space:
$$
H_\bA:=\sspan\{\tp_0,\cdots,\tp_3\}[\bp] ,\quad \tp_i:=I_0 I_{11}\cdots I_{ii} H^i \quad \text{ for } i =0,\cdots,3 .
$$ 
We introduce the propagator matrix with  {gauge} $\cG$  by
\beq\label{RAG}
R^{\bA,\cG}(\psi)\upmo=
 \id-{\small \left(\!\!\!\!\!\! \!\!  \arraycolsep=1.4pt\def\arraystretch{1.1}\begin{array} {*{15}{@{}C{\mycolwddc}}}
0 &  \psi\cdot E^\cG_{1\tp_2} & \psi^2\cdot E^\cG_{\tp\bp} & \psi^3\cdot E^\cG_{1\bp^2}\\
& 0&  \psi\cdot E^\cG_{\tp\tp} & \psi^2\cdot E^\cG_{1,\tp\bp} \\
& & 0 &  \psi\cdot E^\cG_{1\tp_2}\\
& & & 0\\
\end{array} \right) }\in \End H_\bA.
\eeq  
Here besides the BCOV's propagators \eqref{bcovgens}, we introduce extra propagators
 \begin{align*}
E^\cG_{1\tp_2} := E^\cG_{\bp},\quad E^\cG_{1,\tp\bp}:=- E^\cG_{\bp}\cdot E^\cG_{\tp\tp} -  E^\cG_{\tp\bp} ,
 \quad E^\cG_{1\bp^2}:=-E^\cG_{\bp}\cdot E^\cG_{\tp\bp} -  E^\cG_{\bp\bp}.
 \end{align*}
 \smallskip
 
We define the A-model master potential via the following graph sum formula
$$
f^{\bA,\cG}_{g;\ba,\bb} =\langle \tp_{a_1} \bp^{b_1},\cdots, \tp_{a_n} \bp^{b_n}
\rangle^{\bA,\cG}_{{g,n}}:=\sum_{\Gamma \in G_{g,n}} \frac{1}{|\Aut (\ga)|} \Cont_\Gamma^{\bA,\cG} (\tp_{a_l}\bp^{b_l})
\vspace{-0.2cm}
$$
where for each stable graph $\Gamma$, the contribution $\Cont_\Gamma^{\bA,\cG}$ is defined via taking
the composition rule along the following placements\footnote{ Indeed, the graph sum defined here is the $R^{\bA,\cG}$-matrix action, see \S \ref{Raction} for more details.}:
\begin{itemize}
\item at each leg $l$ with insertion $\tp_{a_l} \bp^{b_l}$, we place the vector
\begin{align*} \textstyle
 R^{\bA,\cG}(\bp)^{-1} \tp_{a_l} \bp^{b_l} \   \in  \ H_\bA \ ;
\end{align*}
\item at each edge, we place the bi-vector\footnote{ A direct computation shows
\begin{align*}
& \qquad  V^{\bA,\cG}( \psi, \psi') =  E^\cG_{\tp\tp}\, ( \tp_1 \otimes \tp_1  )
 + E^\cG_{\tp\bp} \,   ( \tp_1  \otimes \tp_{0}\bp'  \!+   \tp_{0}\bp \otimes \tp_1 ) + E^\cG_{\bp\bp}   \, (  \tp_{0}\bp \otimes   \tp_{0}\bp')\\
&  \quad \ +E^\cG_{1,\tp\bp}\, ( \tp_0 \otimes \tp_1 \bp' \!+\! \tp_1 \bp \otimes  \tp_0 )
 \!+\! E^\cG_{1\bp^2} \,   ( \tp_0  \otimes  \tp_{0}(\bp')^2 \!+\!   \tp_{0}\bp^2 \otimes \tp_0 )
  + E^\cG_{1\tp_2}   \, (\tp_0\otimes \tp_2 \!+\!\tp_2\otimes \tp_0).
\end{align*}} in $H_\bA  \otimes H_\bA$
\begin{align*}
\qquad \qquad \! V^{\bA,\cG}( \psi, \psi'):=  {\footnotesize \text{$\frac{1 }{ \psi +\psi'}$}}{\small \text{$\sum_i \Big(\tp_i \otimes \tp_{3-i} - \!R^{\bA,\cG}\!(\psi)^{-1}\tp_i \otimes R^{\bA,\cG}\!(\psi)^{-1}\tp_{3-i}  \Big)$}}  \ ;
\end{align*}   
\item at each vertex, we place the map
\begin{align} \label{normalizedquintic}
\qquad \tau_1(\bp)\otimes \cdots \otimes \tau_n(\bp)  \mapsto  {\footnotesize \text{$\frac{(5Y)^{g-1}}{I_0^{2g-2+n}}$}}\left<  \tau_1(\bp_1),\cdots, \tau_n(\bp_n) \right>^{Q}_{g,n}.
\end{align}
\end{itemize}
In particular, when $\ba=1^m0^n$ and $\bb=0^m1^n$, we define
\beq  \label{AmodelMpotential}
f^{\bA,\cG}_{g,m,n} =\langle \tp_{1}^{\otimes m}, (\tp_{0}\bp)^{\otimes n}  \rangle^{\bA,\cG}_{g,m+n}.
\eeq

\begin{thm}[A-model $\nmsp$ Feynman rule]\label{thm2}
For  any
\beq  \label{conditionforcG}
c_{1a}, \ c_{1b} \in \mathbb Q[X]_1,\quad c_2 \in \mathbb Q[X]_2,\and c_3 \in \mathbb Q[X]_3,
\eeq
we have the following polynomial structure statement
$$
f^{\bA,\cG}_{g;\ba,\bb}\in  \mathbb Q[X]_{3g-3+n-\sum_i b_i } .
$$
\end{thm}

\begin{remark}
Comparing with BCOV's Feynman rule, we see that in the A-model case
\begin{itemize}
\item the state space is is of higher dimension; and we have $6$ (instead of $3$) types of edge contributions (which we call
{extra propagators});
\item there is no ``correction term" in the $g=1$ vertex (see Remark \ref{genus1correction} for more details);
\item the  master potential $f^{\bA,\cG}$  is indeed the generating function of a CohFT $R^{\bA,\cG}.\bar\Omega^Q$ (c.f. \S \ref{cohft}), where $\bar\Omega^Q$ is the normalized CohFT of quintics (c.f. \S \ref{cohftNQ}).
\end{itemize}

\end{remark}

{
\begin{remark}
The $\nmsp$ Feynman rule essentially says that, 
in the orbit of the $R$-matrix group action  on the quintic CohFT, there exists a ``special" subset $\{\Omega^{\bA,\cG}: \cG\}$ which is invariant under BCOV's ``gauge" group, such that their genus $g$ potential functions are simply  degree $3g-3$ polynomials in $X$.
\end{remark}
}

\subsection{BCOV's Feynman rule versus $\nmsp$ Feynman rule }\label{NMSPeqBCOV}

We now state our final result. 
Theorem \ref{thm1} is a direct consequence of Theorem \ref{thm2} and this result.

\begin{thm}  \label{thm3}
For   $\star = A$ or $B$ we introduce the master potential function
$$
f^{\star,\cG} (\hbar,x,y) := \sum_{g,m,n} \hbar^{g-1}  f^{\star,\cG}_{g,m,n} x^m y^n.
$$
Then we have the identity 
$$  f^{\bA,\cG} (\hbar,x,y)  =  f^{\bB,\cG}  (\hbar,x,y) -\ln  (1-y).
$$
Namely, the two types of graph sums are related by
$$
f^{\bA,\cG}_{g,m,n} = f^{\bB,\cG}_{g,m,n}+ \delta_{g,1}\delta_{m,0} (n-1)!.
$$
\end{thm}

\begin{remark}
{ 
Indeed, we will see that the graph sum definition of $f^{\star,\cG}_{g,m,n}$ (for $\star  =\bA$ or $\bB$) is equivalent to the following quantization of $R^\star$-matrix action:
$$
 f^{\star,\cG}(h,x,y) := \log \Big(  \widehat R^{\star.\cG}  P^\star(\hbar;x,y) \Big)
$$
where the generating function $P^\star(\hbar;x,y)$ are defined via
\begin{align}\label{PBA}
P^\bB(\hbar,x,y) = \,  P^\bA(\hbar,x,y)   + \ln(1-y) :=\,&   \sum_{g,m,n} \hbar^{g-1} \frac{x^my^n}{m!n!} \cdot   P_{g,m,n}.
\end{align}
See \S \ref{quantization} for more details about the quantization of symplectic transformations.
}
\end{remark}

 \vspace{1cm}
\section{Cohomological field theory and $R$-matrix action} \label{cohft}

\def\pr{\mathrm{pr}}

In this section, we investigate the CohFTs and the $R$-matrix actions. We will follow closely the treatment developed by Pandharipande-Pixton-Zvionkone in \cite{PPZ}.

We first fix notations. Let $Q\sub\Pf$ be a smooth quintic CY threefold; let
$(\pi,\ev_{n+1}):\cC\to \M_{g,n}(Q,d)\times Q$ 
be the universal family of the moduli of stable maps to $Q$,  
and let
$$\ff_{g,n,d}: \M_{g,n}(Q,d)\to\M_{g,n} \and \pr_k:\M_{g,n+k}\to \M_{g,n}
$$
be the obvious the forgetful morphisms.
\subsection{Definition of cohomological field theory }
We recall the definition of a CohFT introduced by Kontsevich-Manin \cite{KM}.
\begin{definition}
A CohFT consists of
a triple $(V,\eta,\mathbf 1)$, where $V$ is an $F$-linear
space\footnote{ By ``a space over a domain $F$" we mean a locally free $F$ module.} for an integral domain $F$,
$\eta$ is a 
non-degenerate {(super) symmetric} bilinear form $\eta: V \times V \rightarrow F$, $\bone\in V$ is called a unit,
and $S_n$-equivariant maps
$$
\Omega_{g,n}:  V^{\otimes n} \rightarrow  H^*( \M_{g,n})\otimes \aA, \quad g\ge 0, \ 2g-2+n>0,
$$
where $\aA$ is an $F$-algebra, called the coefficient ring,
such that,  for any basis $\{e_k\}$ of $V$ and its dual basis $\{e^k\} $\footnote{ $\{e_k\}$ and $\{e^k\}$
satisfying
$\eta(e_k, e^\ell) = (-1)^{\deg e_k\deg e^\ell}\eta( e^\ell,e_k) = \delta_{k \ell}$.
},
the maps $\Omega_{g,n}$ satisfy the following properties (axioms):
\begin{itemize}
\item[(a1)] Fundamental Class Axiom:
\begin{align*}
\Omega_{0,3}(\mathbf 1, \tau_1,\tau_2)= & \ \eta( \tau_1, \tau_2), \\
\Omega_{g,n+1}( \tau_{\mathbf n},\mathbf 1) =(\pr_1)^*\Omega_{g,n}(\tau_{\mathbf n}), &\qquad \tau_{\mathbf n}:=(\tau_1,\cdots,\tau_n);
\end{align*}
\item[(a2)] Splitting Axiom and Genus reduction axiom
\begin{align*}
s^*\Omega_{g_1+g_2,n_1+n_2}(\tau_{\mathbf n_1}, \tau_{\mathbf n_2} ) =&\, {\textstyle  \sum_k} \ \Omega_{g_1,n_1+1}(\tau_{\mathbf n_1},e^k)  \cdot\Omega_{g_2,1+n_2} (e_k,\tau_{\mathbf n_2}),\\
 r^*\Omega_{g+1,n}(\tau_{\mathbf n})=&\, {\textstyle  \sum_k} \ \Omega_{g,n+2} (\tau_{\mathbf n},e^k,e_k ).
\end{align*}
\end{itemize}
\end{definition}
Here $s$ and $r$ are the obvious gluing maps.



\begin{example}
[CohFT  of the  GW theory of $X$] \label{quinticGW}
 For a projective variety $X$, and a coefficient field $F$, we introduce the triple and the maps by
\begin{align*}
V=& \ H\sta(X,F);   \quad  (x,y)=\int_X x\cup y; \quad 1\in H^0(X,F);
\\
\Omega^{X}_{g,n}(\tau_{\mathbf n}):= & \sum_{d=0}^\infty q^d
\ff_{g,n,d\ast}\Bigl(  \prod_{i=1}^n\ev_i\sta \tau_i   \cap [\M_{g,n}(X,d)]\virt \Bigr)  \in H\sta(\M_{g,n},F)[\![q]\!].
\end{align*}
\end{example}

\subsection{Shift and direct sum of CohFTs}

\begin{definition}[The shifted CohFT $\Omega^{\tau}$ of a given CohFT $\Omega$.]
 For  $\tau\in V\otimes_F \aA $, 
$$
\Omega^\tau_{g,n} (\tau_{\mathbf n}):= \sum_{k \geq 0} \frac{1}{k!} \pr_{k\ast}\Omega_{g,n+k} (\tau_{\mathbf n},\tau^k) \in H\sta(\M_{g,n},\aA),
$$
with the same triple $(V,\eta,1)$ of $\Omega$.
Here we assume that the infinite sum is well defined.
\end{definition}

\begin{definition}[The direct sum  of  CohFTs]
Let $\Omega^a$ and $\Omega^b$ be two CohFTs with identical coefficient ring $\aA$. We define their direct sum to
be the CohFT with triple
$(V^a\oplus  V^b,\eta^a\oplus  \eta^b,\mathbf 1^a \oplus  \mathbf 1^b)$, and with maps
\begin{align*}
&(\Omega^a \oplus \Omega^b )_{g,n}(\tau_{\mathbf n}) =\Omega^a_{g,n} (\tau_{\mathbf n}^a) +\Omega^b_{g,n} (\tau_{\mathbf n}^b) \in H^*(\M_{g,n},\aA),
\end{align*}
where $\tau_i=(\tau_i^a,\tau_i^b) \in V^a\oplus  V^b$.
By iterating, we get a direct sum of finite copies of CohFTs.
It is easy to check that the direct sum of CohFTs so defined satisfies all the CohFT axioms, and hence is a CohFT.
\end{definition}

\begin{example} Let $\Omega^X$ be as in Example \ref{quinticGW}.
For two smooth projective varieties $X_1$ and $X_2$, we have
$\Omega^{X_1 \sqcup X_2}=
\Omega^{X_1}\oplus \Omega^{X_2}$.
\end{example}

\subsection{$R$-matrix action on CohFT}\label{Raction}
The $R$-matrix was first introduced in \cite{Giv2, Giv3} { to compute   higher genus} equivariant GW  invariants. Its lifting to CohFTs was studied in \cite{Sh09,Te12}.
In this section, we will mostly follow 
\cite{PPZ}\footnote{ In their paper, the authors give a careful proof that $R$-matrix actions preserve CohFTs.},
with a slight generalization.

Let $\Omega$ be a CohFT with the triple $(V,\eta,\mathbf 1)$. We consider another triple $(V',\eta',\mathbf 1')$, and a formal power series
\begin{align*}
R (z)= R_0 +R_1  z+R_2  z ^2 +\cdots \in  \End(V,V')\otimes \aA[\![z]\!],
\end{align*}
which satisfies
the ``symplectic condition": \footnote{ The symplectic condition is equivalent to :
$\eta(v_1,v_2)   =  \eta'\big( R(z) v_1, R(-z)v_2 \big)$, for $v_1,v_2 \in V$.
It could not deduce
$\eta(R^*(z) v_1,R^*(-z)  v_2)   =  \eta'\big(  v_1,  v_2 \big)$
when  $\dim V \neq \dim V'$.}
\begin{align}\label{symplectic}
R^*(-z) R(z) = \id \in \End(V) .
\end{align}
  Notice that \eqref{symplectic} implies that $R(z)$ is injective and $\dim_F V\leq \dim_F V'$. 

We define the $R$-matrix action following \cite{PPZ}. Let $\ga\in G_{g,n}$ be a  genus $g$ stable graph with $n$ legs.  
For each vertex $v$ of $\ga$, we denote its genus by $g_v$ and its valence by $n_v$. For each $\Gamma$ we associate it the space
$\M_\Gamma:= \prod_{v\in V(\Gamma)}\M_{g_v,n_v}$,
and define the contribution
$$
\Cont_\Gamma :  V^{' \otimes n} \lra H^*(\M_{\Gamma}, \aA)
$$
by the following construction
\begin{enumerate}
\item at each leg $l$ of $\Gamma$, we place a map
\begin{align*}
 R^*(-\bp_l) :   V' \lra  V[{\bp_l}]  ;
 \end{align*}
\item at each edge $e=(v_1,v_2)$ of $\Gamma$, we place 
\begin{align*}
\qquad
\frac{\sum_\beta e_\beta \otimes e^\beta-\sum_\alp R^*(-\bp_{(e,v_1)}){e'_\alpha} \otimes  R^*(-\bp_{(e,v_2)}) e'^{\alpha} }{\bp_{(e,v_1)}+\bp_{(e,v_2)}}
 \in {V[\bp_{(e,v_1)}]\otimes V[\bp_{(e,v_2)}]};
\end{align*}
where $\{e_\beta\}, \{ e'_\alpha\}$ are bases for $V,V'$ respectively, and 
$\{e^\beta\}\sub V, \{ e'^\alpha\}\sub V'$ are their dual bases under $\eta,\eta'$ respectively;
\item at each vertex $v$ of $\Gamma$, we place
$$\Omega_{g_v,n_v}:  V ^{\otimes n_v}\lra  H^*(\M_{g_v,n_v},\aA).
$$
\end{enumerate}

Let $\xi_\Gamma: \M_\Gamma \rightarrow \M_{g,n}$ be the tautological morphism by gluing. We define 
$$
(R\Omega)_{g,n}=\sum_{\Gamma\in G_{g,n}} \frac{1}{|\Aut \Gamma|} {\xi_\Gamma}_*\bl\Cont_\Gamma\br.
$$


Let $\bp_i$ be the ancestor psi classes of $\M_{g,n+k}$.
For the given $R$-matrix, we associate 
\begin{align*}
T_R(z)= z  \mathbf 1 -z R(-z)\sta  \mathbf 1'   \in   z\aA[\![z]\!]   \otimes V; \end{align*}
we define its associated translation action by
\begin{align}\label{Tail-action}
T_R \Omega_{g,n} (-)= \sum_{k\geq 0}\frac{1}{k!}  (\pr_k)_* \Omega_{g,n+k}(-, T_R(\bp_{n+1}),\cdots, T_R(\bp_{n+k})),
\end{align}
assuming that the infinite sum makes sense in $ H^*( \M_{g,n})\otimes \aA$.
For example, if $\aA=F[\![q]\!]$ is  endowed with   $q$-adic topology and 
\begin{align}\label{qT}
T_R(z)\in z^2\aA[\![z]\!]\otimes V+q\, z\, \aA[\![z]\!]\otimes V, \end{align}
then \eqref{Tail-action} automatically converges. We call   \eqref{qT} the $q$-adic conditions for $T_R$.

\begin{definition}
The $R$-matrix action on a CohFT $\Omega$ is defined by
$$
R.\Omega :=  R T_R \Omega .
$$
\end{definition}

\subsection{Properties of CohFTs under $R$-matrix actions}

Following \cite{PPZ}, we have
\begin{theorem}\label{PPZThm}
Let $\Omega$  be   a CohFT with unit for the triple $(V,\eta,\mathbf 1)$. Let $\aA=F[\![q]\!]$ be endowed with $q$-adic topology. 
We have the followings.
 \begin{enumerate}
\item Let $(V',\eta',\mathbf 1') $ be another triple. Suppose  $R (z)\in  \End(V,V')[\![z]\!] $ is symplectic and 
 $T_R$ satisfies the $q$-adic condition.  
Then $T_R\Omega$ is well-defined and is a CohFT, and $R. \Omega$ is also a CohFT. 
Furthermore,  if $R^*(z)\in  \End(V',V)[\![z]\!]$ is symplectic\footnote{ This is equivalent to
$\dim_F V=\dim_F V'$.},   $R. \Omega$ is a CohFT with unit $\mathbf 1'\in V'$.
\item Suppose $(V'',\eta'' ,\mathbf 1'')$ is  another vector space with pairing and unit. Suppose
$$R_a (z)\in  \End(V,V')[\![z]\!] \and R_b (z)\in  \End(V',V'')[\![z]\!]
$$
 are symplectic, with $T_{R_a},T_{R_b}$  satisfying the $q$-adic conditions. 
Then as CohFTs on $(V'',\eta'',\mathbf 1'')$
$$ (R_aR_b). \Omega = R_a. (R_b. \Omega).
$$
\end{enumerate}
\end{theorem}
\begin{proof} All statements can be proved by exactly the same arguments as in \cite[Prop 2.12 and 2.14]{PPZ}, except that for
the axioms on unit  in (1), the identity $(RT\Omega)_{0,3}(\mathbf 1, \tau_1,\tau_2)=\eta'( \tau_1, \tau_2)$
is shown in Lemma \ref{unit-lem}.  We leave other identities  to readers.
\end{proof}


\begin{remark}
We remark that in \cite{PPZ} the authors used $V=V'$ and $R_0=\id$.  In the next section  we will use $R$ actions in the case  $\dim_F V < \dim_F V'$. For more relations with \cite{PPZ}, see Example \ref{examplePPZ}.
\end{remark}

\subsection{Examples of CohFTs}  In this subsection,
we list some CohFTs  used in this paper.

\smallskip

We consider the following CohFTs that  arise   in the localization of
$\nmsp$ theory. As in \cite{NMSP2}, we pick a sufficiently large integer $\n$; let $G=(\CC\sta)^\n$,
and take $H\sta(BG)= \QQ[\ft_1,\cdots,\ft_\n]$
where $\ft_\alp$ is the $\alp$-th equivariant generator.
In this paper, after equivariant integration we will always specialize $\ft_\alp$ to $-\zeta_\n^{\alpha} \ft $,   where
$\zeta_\n=e^{\frac{2\pi i}{\n}}$ is the primitive $\n$-th root  of unity.  In this paper we always take $F=\QQ(\zeta_\n)(\ft)$ and $\aA=F[\![q]\!]$.

\subsubsection{CohFT $\Omega^{\ffp_\alpha,\tw}\!\!$ of twisted GW theory of a point}
The triple is
$$\sH_{\ffp_\alp}= H^0(\ffp_\alpha),\quad (\cdot,\cdot)^{\ffp_\alpha,\tw}, \quad 1_\alp:=1\in H^0(\ffp_\alpha)
$$
with the inner product
$$
\qquad (x,y)^{\ffp_\alpha,\tw}:=\frac{5}{\ft_\alpha^4 \prod_{\beta: \beta \neq \alpha} (\ft_\beta-\ft_\alpha)}xy
= \frac{-5}{\n \ft_\alp^{3} \ft^\n }xy. 
$$
Let $\EE_{g,n}$ be the  Hodge  bundle over $\M_{g,n}$; the maps are
\begin{equation*}
\quad\Omega^{\ffp_\alp,\tw}_{g,n}(\tau_{\mathbf n})=   (-1)^{1-g} \frac{e_T(\EE_{g,n}\dual\otimes (-\ft_\alp))^5}{(-\ft_\alp)^5} \frac{5\ft_\alp}{e_T(\EE_{g,n}\otimes5\ft_\alp)}
 \prod_{\substack{\beta: \beta\neq \alp}}\frac{e_T\big(\EE_{g,n}\dual\otimes (\ft_\beta-\ft_\alp) \big)}{ (\ft_\beta -\ft_\alp)}    \prod_i \tau_i .  \qquad
\end{equation*}
This gives us $\Omega^{\ffp_\alp,\tw}_{g,n}$.

We introduce a CohFT $\omega^{\ffp_\alp,\tw}$,  which is  the topological part of $\Omega^{\ffp_\alp,\tw}_{g,n}$:
the triple is $\sH_{\ffp_\alp}$ with the same inner product and unit;
the maps are defined by $\omega^{\ffp_\alp,\tw}_{g,n}(1^{\otimes n}) =  (\frac{-5}{\n \ft_\alp^{3 }\ft^\n })^{1-g}.$
By \cite{Mu83, FP, Giv2}, we have
\begin{proposition}\label{GRRpt} The following identity   between CohFTs holds
\begin{equation}\label{ptGRR}
\Omega^{\ffp_\alp,\tw} = \Delta^{\ffp_\alp}.  \omega^{\ffp_\alp,\tw} ,
\end{equation}
where the $R$-matrix $ \Delta^{\ffp_\alp}$ is given by 
\begin{equation*} \label{deltapt}
  \Delta^{\ffp_\alp}(z)= \exp \Bigl( {\sum_{k> 0 }     \frac{B_{2k}}{2k(2k-1)}   \Big(\frac{5}{(-\ft_\alpha)^{2k-1} }+\frac{1}{(5\ft_\alpha)^{2k-1}}+\sum_{\beta\neq \alpha}\frac{1}{(\ft_\alpha-\ft_\beta)^{2k-1} } \Big) z^{2k-1}}\Bigr).
\end{equation*}

\end{proposition}

\begin{remark}\label{defnoftopology}
We see  that  
the topological CohFT ${\omega^{\ffp_\alp,\tw}}$ has the same vector space as 
{that of} the CohFT $\Omega^{\ffp_\alp}$,
but with different inner  product.   In fact if we define
$$\tilde \Delta^{\ffp_\alp}(z):=
 {\sqrt{5/\n}}\cdot { \ft_\alpha^{-(3+\n)/{2}}}\,   \Delta^{\ffp_\alp}(z),
$$
then we have the CohFT identity\footnote{ The $T$-action here is well-defined with suitable topology.
We skip the argument as we don't need it.
}
$$
\Omega^{\ffp_\alp,\tw}  = \tilde \Delta^{\ffp_\alp}. \Omega^{\ffp_\alp}.
$$
\end{remark}


\noindent 
{{\bf Convention.}} {\sl For simplicity, in the following we write $\npt$ as the disjoint union of $\pt_\alpha$, $1\le \alpha\le \n$.
Accordingly, $\Omega^{\npt,\tw}=\oplus_{\alpha=1}^{\n} \Omega^{\pt_\alpha,\tw}$, $\omega^{\npt,\tw}=\oplus_{\alpha=1}^{\n} \omega^{\pt_\alpha,\tw}$, etc.}

\subsubsection{CohFT  $\Omega^{Q,\tw}$ {of the} twisted GW theory of quintic threefold and the shifted CohFT $\Omega^{Q,\tw,\tau}$} 
Let $Q$ be a smooth quintic CY threefold. The CohFT $\Omega^{Q,\tw,\tau}$ consists of the triple
$$\sH_Q=H\sta (Q),\quad  (x,y)^{Q,\tw}=\int_Q\frac{xy}{\prod_{\alp=1}^\n (H+\ft_\alp)}=\int_Q \frac{xy}{-t^\n}, \quad 1_Q:=1 \in H\sta (Q),
$$
and the  map  
$$\Omega^{Q,\tw}_{g,n}(\tau_{\mathbf n})= 
\sum_{d=0}^\infty q^d \cdot
\ff_{g,n,d\ast}\Big( \frac{\ev_1\sta \tau_1 \cdots \ev_n\sta \tau_n}{\prod_{\alp=1}^\n e\big(  R\pi\lsta \ev_{n+1}\sta\sO(1)\cdot \ft_\alp \big)}  \cap
[\M_{g,n}(Q,d)]\virt\Bigr). 
$$

\begin{remark}\label{Q-to-tw} By dimension reason one calculates
$$
\Omega^{Q,\tw}_{g,n}(\tau_{\mathbf n})=  {(-\ft^\n)^{(g-1)} } \Omega^{Q}_{g,n}(\tau_{\mathbf n}) |_{q\mapsto q':=-q/\ft^\n}.
$$
\end{remark}
\smallskip

%

By the fundamental class axiom, if $\tau$ is a scalar multiple of the unit, $\Omega^\tau=\Omega$, for any CohFT $\Omega$.  In particular
$\Omega^\pt$ and $\Omega^{\pt,\tw}$ are not affected by any shift. Also, for $\tau \in   \sH_Q\otimes_F \aA  $, we denote by $\Omega^{Q,\tw,\tau}$ the $\tau$-shifted CohFT
of $\Omega^{Q,\tw}$. 

{\noindent}{\bf Convention}.
{\sl By abuse of  notations,  we denote
$\Omega^{Q,\tw}=\Omega^{Q,\tw,\tau_Q(q')}$ from now on.
Here $\tau_Q(q):=J_1(q) = I_1(q)/I_0(q)$ is the mirror map.} 

\subsubsection{CohFT  $\bar\Omega^{ Q,\tau}$ of ``normalized" shifted GW theory of  the  quintic threefold}  \label{cohftNQ}
We consider the following ``normalized" CohFT
$$
\bar\Omega^{ Q,\tau_Q}_{g,n}(-) :=  {(5Y)^{g-1}} \,  \Omega^{ Q,\tau_Q}_{g,n}(-) ,\qquad (v_1,v_2)^{ Q,\bar{} } =  {(5Y)^{-1}}(v_1,v_2)^Q.
$$
{We can see that} the graph sum defined in \S \ref{NMSPrule} is indeed an $R^{\bA,\cG}$-action on the normalized quintic CohFT  $\bar\Omega^{ Q,\tau}$.
In  \eqref{normalizedquintic}, the factor ${(5Y)^{g-1}}$ is from the above normalization factor, while $I_0^{-(2g-2+n)}$ is obtained by applying dilaton equations to the tail contributions (c.f. \eqref{normalizedR}, see also \cite[Sect.\,3.5]{NMSP2}). 

Further, with the change of variable $q\mapsto q':=-q/\ft^\n$ and by adding the normalized factor 
$(-5Y/\ft^\n)^{(1-g)}$, we can identify {these two CohFTs:}
\beq \label{twnormalized} 
\Omega^{Q,\tw}_{g,n}(-)=   \Big[ (-5Y/\ft^\n)^{(1-g)}  \bar\Omega^{Q,  \tau_Q  }_{g,n}(-) \Big] |_{q\mapsto q':=-q/\ft^\n}. 
\eeq




\subsubsection{CohFT $\Omega^{\aleph}$  }
The following CohFT is of fundamental importance to this paper:
 \begin{definition} \label{Omegaaleph}
 For $\aleph:= Q\cup \npt$,  we define the CohFT of the local theory to be
\begin{align*} 
 \Omega^{\aleph}:=  \Omega^{Q,\tw}  \oplus \omega^{\npt,\tw},
\end{align*}
where the triple is $\sH:=H\sta(\aleph)$, with the pairing and the unit   
$$(\cdot,\cdot)^{\tw}=(\cdot |_Q,\cdot|_Q)^{Q,\tw}+(\cdot |_{\ffp},\cdot |_{\ffp})^{\ffp,\tw},\qquad \mathbf 1=  \mathbf 1_Q+   \sum_\alp \mathbf 1_{\alp} .
$$
Here $\cdot|_{Q}: \sH\to\sH_0$ and $\cdot|_{\alp}: \sH\to\sH_{\ffp_\alp}:=H\sta(\ffp_\alp)$ are the projections.

\end{definition}

\noindent{{\bf Convention.}}  {\sl
Let $G=(\CC\sta)^{\n}$  act on   $\PP^{4+\n}$ via scaling the last $\n$ homogeneous coordinates of $\PP^{4+\n}$.
Let $p$ the equivariant-hyperplane class in $H^2_G(\PP^{4+\n})$.
In this paper, we will view $p^i$ as their images in
$\sH\uev:= H\uev(\aleph, \aA)\sub \sH$.}




\medskip

Now we recall some basic facts in the setup from \cite{NMSP2}. Considering the   natural decomposition $   \sH=\sH\uev\oplus H^3(Q) ,$
we pick a basis $\{\phi_i:=  {p^i} \}_{i=0}^{\n+3}$ of $\sH\uev$ with dual basis 
 $$
 \text{\small $\{\phi^0,\cdots,\phi^{\n+3}\}= \Bigl\{  \frac{p^3}{5}(p^\n -\ft^\n) ,    \frac{p^2}{5 }(p^\n-\ft^\n) ,  \frac{p}{5 }(p^\n-\ft^\n) ,   \frac{1}{5 }(p^\n- \ft^\n), \frac{p^{\n-1}}{5}, \frac{p^{\n-2}}{5}, \cdots,  \frac{p^{0}}{5}    \Bigr\}.$}
$$
By using the above basis, let $[\n]:=\{1,\cdots,\n\}$ we have
$$
\mathbf 1_\alp = \frac{p^4}{\ft_\alp^4}\prod_{\beta\neq \alp} \frac{\ft_\beta+p}{\ft_\beta-\ft_\alp}  \quad 
\text{for $\alpha\in[\n]$} ,\and
H^j = \frac{p^j}{\ft^\n}(\ft^\n -p^\n) \quad \text{for $j=0,1,2,3$}. 
$$
The Poincare dual of $\{1,H,H^2,H^3\} \cup \{1_\alp\}_{\alp=1,\cdots,\n}$ is 
     $$  \text{ \small $
    \{\frac{-\ft^\n}{5} H^3,\frac{-\ft^\n}{5} H^2, \frac{-\ft^\n}{5} H, \frac{-\ft^\n}{5} H^0 \}\cup\{  1^{\alpha}:=  \frac{\n \ft_\alp^3 \ft^\n}{(-5)}1_\alp     \}_{\alp=1,\cdots,\n} $}.$$

\begin{remark} In \cite{NMSP2} we use the notation
$$
[-]_{g,n}^{\bullet},\quad \text{
where $\bullet$ = ``$\lloc$", ``$Q,\tw$" or ``$\ffp_\alp,\tw$" , 
 }
$$
to define certain classes. They are closely related to the CohFT notation $\Omega^\bullet$ used here, with
a minor change:  in  
$\Omega^\bullet_{g,n}(-)$ descendents are not allowed, while in $[-]^\bullet_{g,n}$ they are.
\end{remark}

\vspace{1cm}

\section{Expressing $\nmsp$-$[0,1]$ theory via CohFTs}
{We first quote the results proved in \cite{NMSP2} in terms of the CohFT and $R$ matrix actions.}
The moduli of $\nmsp$ fields and their localizations are constructed in \cite{NMSP1}. 
In this paper we concentrate on the ``$\nmsp$-$[0,1]$ theory".
For $2g+n>2, \tt_i\in \sH[\![z]\!]$,  the $[0,1]$ theory is
\beq
\bigl[\tau_{\mathbf n} \bigr]^{[0,1]}_{g,n}
=  \sum_{d\ge 0}  (-1)^{d+1-g } q^d \!\!\!
 \sum_{\locg\in G_{g,n,d}^{[0,1]}}
\bl \pr_{g,n}\br\lsta \Bigl(\prod_{i=1}^n \ev_i\sta \tt_i \cdot \Cont_\locg \Bigr) \in H^*(\M_{g,n},\aA),
\quad
\eeq
where $\Cont_\locg$ are   contributions from those $\nmsp$ localization graphs supported on $[0,1]$(c.f. \cite[Def 0.1]{NMSP2}). 
Here $\pr_{g,n}:\cW_{g,n,\bd}\to\M_{g,n}$ is the projection.

We remark that in this paper we adopt the convention that 
$\bigl[\tau_{\mathbf n}\bigr]^{[0,1]}_{g,n}=\bigl[\tau_1,\cdots,\tau_n\bigr]^{[0,1]}_{g,n}.$
\begin{definition}\label{01theory-def}  We define $\displaystyle
\Omega^{[0,1]}_{g,n}(\tau_{\mathbf n})=\bigl[\tau_{\mathbf n}]^{[0,1]}_{g,n}$,   for $\tau_i\in \sH$.
 \end{definition}

In \cite[Thm~3 (Thm~3.1)]{NMSP2} we express $\Omega^{[0,1]}$ as a graph sum involving $R(z)$. 

\begin{theorem}[{\cite[Thm~3 and Thm~4]{NMSP2}}]\label{Omega01}
We define $R(z)\in \End \sH  \otimes \aA[\![z]\!]$ via the Birkhoff factorization
\begin{equation}\label{birkR}
\ S\msp(q,z) \begin{pmatrix} \diag \{\Delta^{\ffp_\alp}(z) \}_{\alpha=1}^\n & \\ &1  \end{pmatrix}
= R(z)    \begin{pmatrix} \diag \{S^{\ffp_\alpha}(z)  \}_{\alpha=1}^\n& \\ & S^Q(z)  \end{pmatrix} \Big|_{q\mapsto q'}
\end{equation} where $q'=-\frac{q}{t^\n}$, and $S\msp,S^Q$ are the $S$-matrices of $\nmsp$-$[0,1]$ theory and quintics
respectively\footnote{ See \cite[Sect.\,1]{NMSP2} for definitions of these $S$-matrices.}.
The $\nmsp$-$[0,1]$ theory    $\Omega^{[0,1]}$ forms a CohFT satisfying
$$ \Omega^{[0,1]}  \  =  \  R. \Omega^{\aleph}.  $$
Furthermore, for $2g-2+n>0$, the $\nmsp$-$[0,1]$ correlator
$$
\left<  \phi_{m_1}\bp^{k_1} ,\cdots, \phi_{m_n}\bp^{k_n} \right> ^{[0,1]}_{g,n} =  \int_{[\M_{g,n}]}  \psi_{1}^{k_1}\cdots\psi_n^{k_n} \cdot \Omega_{g,n}^{[0,1]}( \phi_{m_1},\cdots, \phi_{m_n})
$$
is a $q$-polynomial of degree $\le g-1+\frac{3g-3+ \sum_{i=1}^n m_i}{\n}$.
 \end{theorem}

A few remarks on Theorem \ref{Omega01} are in order.
{The whole argument  \cite[Sect.\,3.5]{NMSP2}  is  a composition of $R$-matrix actions on CohFTs}
$$
\Omega^\aleph=\Omega^{Q,\tw} \oplus   \omega^{\npt,\tw}  \xrightarrow{\quad \id\oplus (\oplus_\alp\Delta^{\pt_\alp })\quad}
  \Omega^{Q,\tw } \oplus   \Omega^{\npt,\tw}  \xrightarrow{\quad\, R\loc \,\quad } \Omega^{[0,1]} .
$$
Here  $R^\lloc$ is the $R$ matrix for torus localizations\footnote{ $R\loc$ is defined in \cite[Defn.\,1.7]{NMSP2}.},
 and $\Delta^{\ffp_\alp}$ is from Grothendieck-Riemann-Roch(GRR) formula at $\ffp_\alp$(c.f.\eqref{ptGRR}, see \cite{Mu83, FP, Giv2}). The  $q$-adic condition for the GRR's $R$ matrix holds since $ \Delta^{\ffp_\alp} = 1+O(z)$.  The  $q$-adic condition  for   $R\loc$ holds because its
 tail $T_{R\loc}$ lies in $q\aA\otimes V$ by  \cite[(3.10), Remark 3.4]{NMSP2}. 
  Thus Theorem \ref{PPZThm} implies $R.\Omega^\aleph=\Omega^{[0,1]}$, where $R$ is   the composition of these two actions
$$
R (z):=  R \loc(z)  \cdot (\id \oplus \Delta(z)^{\npt })\in \End( \sH) \otimes \aA[\![z]\!].
$$
It satisfies the {defining identity} \eqref{birkR} by \cite[Remark~3.6]{NMSP2}. 

 We define $R_i \in \End( \sH) \otimes \aA$ via
\begin{align*}
 R (z)= R_0 +R_1  z+R_2  z ^2 +\cdots. \end{align*}

\subsection{$\Omega^{[0,1]}$-theory in terms of stable bipartite   graphs of $\Omega^{[0]}$ and $\Omega^{[1]}$-theory}
In this section we decompose $\Omega^{[0,1]}$ into two subtheories. Such decomposition holds for $R$-matrix action on a general direct sum of CohFTs.

\begin{definition}[Restriction of $R$-matrix action on small blocks]
For $\tau_i \in \sH$  ($i=1,\cdots,n$) and $\star = 0$ or $1$ , we  define
$$
\qquad \Omega^{[\star]}  := \begin{cases}
   R^{[0]} . \Omega^{Q,\tw} & \star = 0 \\
   R^{[1]} .  \omega^{\npt,\tw}  & \star = 1
   \end{cases} \  ,
\qquad
$$
where the  $R^{[\star]}$-matrices are
\begin{align*} 
R^{[0]}(z)=R(z) |_{Q} &:=R(z)|_{\sH_Q} \in  \Hom(\sH_Q, \sH)[\![z]\!], \\
 R^{[1]}(z)=R(z) |_{\npt} &:=R(z)|_{\sH_{\npt}} \in  \Hom(\sH_{\npt}, \sH)[\![z]\!]. \nonumber
\end{align*}
 Notice that here \
 $$\sH_Q=\sspan\{ \phi^i \}_{i=0}^3 \oplus \sH_Q^{odd}  \and \sH_{\npt}=\sspan\{ \phi_j\}_{j=4}^{\n+3}$$
  have dimensions strictly less than that of $\sH$.
\end{definition}

\begin{remark}
By the definition, for $\star = 0$ or $1$
$$
\Omega^{[\star]}_{g,n}(\tau_1 ,\cdots,\tau_n) \in H^*(\M_{g,n})
$$
is equal to the summation of those stable graph contributions in
$$(R.\Omega^{\aleph})_{g,n}(\tau_1 ,\cdots,\tau_n)$$
whose vertices are all labeled by $\star$.   
\end{remark}

\begin{remark}  In this paper all operators from {$\sH_Q$ to $\sH_Q$ (resp. $\sH$ to $\sH$)}
are   identity on odd classes and send even classes to even classes. Hence we only describe their action on even classes.    \end{remark}

In this paper a stable graph is a graph whose vertices are decorated by genus, such that $2g_v-2+n_v>0$, where $n_v:=|E_v|+|L_v|$ is the valence of the vertex $v$.  A stable graph is called bipartite if each vertex is further decorated by  (level) $1$ or $0$, and each edge connects vertices of different levels.
 Let $\Xi_{g,n}^\01$ be the set of stable bipartite graphs, with total genus $g$ and $n$ many legs. 
For a stable bipartite graph $\Lam$, we use $V(\Lam)$, $E(\Lam)$ to denote the set of its vertices,  edges {respectively; use $V_{0}(\Lam)$ to denote its level $0$ vertices, etc..} 

\begin{theorem}\label{01bipartie}
We have the following stable bipartite graph formula
\begin{multline} \label{bip01}
\Omega^{[0,1]}_{g,n}(\tau_1,\cdots, \tau_n) = \sum_{\bipg \in \Xi_{g,n}^\01} (\xi_\bipg)_* \Big( \sbigotimes_{v\in V_0(\bipg)}\Omega^{[0]}_{g_{v},n_{v}}\Big) \bigotimes  \Big(  \sbigotimes_{v'\in V_{1}(\bipg)} \Omega^{[1]}_{g_{v'},n_{v'}}  \Big) \\
\bigg( \sbigotimes_{v\in V_0(\bipg),\atop l\in L_v}  \!\!\! \tau_l,\quad   \sbigotimes_{v'\in V_{1}(\bipg),\atop l'\in L_{v'}} \!\!\! \tau_{l'},  \quad\sbigotimes_{e=(v,v')\atop\in E(\bipg)} V^{01} (\psi, \psi')\bigg)\qquad
\end{multline}
where we define
\begin{align}\label{V01} \textstyle
 {V^{01}}(z,w) :=  \sum_{\alpha=1}^N \frac{R^{[1]}(z)  -R^{[1]}(-w)  }{z+w} \mathbf 1_\alpha \otimes  R^{[1]}(w)\mathbf 1^\alpha.
\end{align}
\end{theorem}
\begin{proof}
Just notice that for the graph sum formula of $[0,1]$-CohFT (via the $R$-matrix action on $\Omega^{\aleph}$), 
the contribution of  an edge that connects a $V_0$ vertex and a $V_1$ vertex is given by
\begin{align*} 
\Cont_{E_{01}}=  & \ \textstyle \sum_{i=0}^{\n+3} \frac{ -  R^{[0]}(-z)^*  \phi_i \otimes R^{[1]}(-w)^*  \phi^i }{z+w}  \\
 = &\textstyle \ \sum_{\alpha=1}^\n \frac{R^{[0]}(-z)^*\big(R(z)  -R(-w) \big)  \mathbf 1_\alpha \otimes  \mathbf 1^\alpha }{z+w}  =  \text{$\big( R^{[0]}(-z)^*\otimes R^{[1]}(-w)^*\big)  \cdot    V^{01}(z,w)   $}, \nonumber
\end{align*}
where we have used the symplectic condition,
$$R^{[0]}(-z)^* R^{[1]}(z)=0,\qquad    R^{[1]}(-z)^* R^{[1]}(z)=\id\in \End(\sH_{\npt}).$$
The graph sum formula then follows from the definition of the $R$-matrix action.
\end{proof}

The basis $\{ \mathbf 1_\alp  \}$ and $\{\mathbf 1^\alp\}$ in \eqref{V01} can be replaced by any basis of $\sH_{\npt}$ and   its dual.

\begin{example} The following is an example of a stable bipartite graph of total genus $9$ and two insertions $(\tau_1,\tau_2)$:
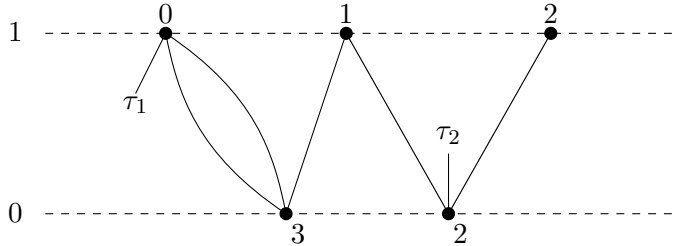
\vspace{-0.3cm}
\begin{figure}[H]
  \centering
  \begin{tikzpicture}[scale=0.8]
  \draw[dashed] (4,0)--(14.5,0) (3.5,-0.3) node[above] {$0$};
  \draw[dashed] (4,3)--(14.5,3)  (3.5,2.7) node[above] {$1$};
    \draw (5.5, 2) --  (6, 3) (5.5, 1.5)  node[above] {$\tau_1$} ;
     \draw (10.7, 1) --  (10.7, 0)  (10.7, 1)  node[above] {$\tau_2$};
    \draw (8, 0) .. controls (6.5, 1) and (6.2, 2)  .. (6, 3);
    \draw (8, 0) .. controls (7.8, 1) and  (7.5, 2)  .. (6, 3);
    \draw[fill]   (8, 0) circle(1mm) (8.2, 0) node[below] {$3$}  (6, 3) circle(1mm) node[above] {$0$}   ;
    \draw (8, 0) --  (9, 3);
    \draw[fill]  (10.7, 0) circle(1mm)   -- (12.4, 3) circle(1mm) node[above] {$2$} (10.9, 0) node[below] {$2$}  ;
    \draw[fill] (10.7, 0) -- (9, 3) circle(1mm) node[above] {$1$}   ;
  \end{tikzpicture}
\end{figure}
\vspace{-0.3cm}
\end{example}

\begin{convention}\label{conven} In the remainder of this and the next section, we
use $K  \in \{L,Y,X,I_k, A_i,B_i\}$ to mean the function $K|_{q\mapsto q'}$ of $q':= - {q}/{t^\n}$.  
For example,
  $L= (1+5^5 \frac{q}{\ft^\n})^{\frac{1}{\n}}$. 
\end{convention}

\begin{convention}From now on, we assume $\n$ is a prime.
\end{convention}

\subsection{Polynomiality of $[1]$-theory}\label{algorithmBCOV}
Let $R(z)= \sum_k R_k z^k $ and 
$$
\textstyle  V(z,w)=\sum_{k,l}  V_{kl} z^k w^l :=  \sum_j  \frac{1}{z+w} ({ \phi_j \otimes   \phi^j-  R(z)^{-1} \phi_j \otimes  R(w)^{-1} \phi^j}).
$$

 We start by a key Lemma in \cite{NMSP2}:

\begin{lemma}[{\cite[Lemma {C.1}]{NMSP2}}] \label{keylem}
Let $k,l\geq 0$, $a = 0,1,\cdots,\n\!+\!3$ and  $\alpha, \beta \in [\n]$. We consider the entries
$$ 
 (R_k)^\alpha _{a}:= \,    L^{\frac{\n+3}{2}} \cdot L_\alpha^{-a+k} (\mathbf 1^\alpha, R_k^*  \phi_a),\quad \text{with $L_\alpha := \zeta_\n^\alpha   \cdot t\cdot L $}  .
$$
\begin{itemize}
\item For the $R$-matrix, we have  that  $(R_k)_{a}^{ \alpha}$ does not depend on $\alpha$, and
\begin{align}\label{Rkform}
(R_k)_{a}:= (R_k)_{a}^{ \alpha} \in  \mathbb Q[X]_{k+\lfloor\frac{a}{\n}\rfloor}.
\end{align}
\item
 For the $V$-matrix, we have that
 $V_{kl} |_{\npt \times \npt}$ is of the following form
\begin{align*}
V_{kl} |_{\npt \times \npt}= L^{-3} t^\n
 \sum_{\alpha,\beta }  \sum_{s}  L_\alpha^{s-k} L_\beta^{2-s-l}   \cdot (  V_{kl})^{\alpha \beta; s} \,\mathbf 1_\alpha\otimes\mathbf 1_\beta, 
\end{align*}
such that the entries ${(V_{kl})}^{\alpha \beta;s}$  are independent of $\alpha,\beta$ and 
 \begin{align} \label{Vkl-poly}
  {(V_{kl})}^{\alpha \beta;s} 
\in   \mathbb Q[X]_{k+l+{1}} .
\end{align}
\end{itemize}
\end{lemma}

\begin{definition}
 Let $\star = [0], [1]$ or $[0,1]$,
we introduce the $\star$-potential  for  $\ba=(a_1,\cdots,a_n)$, $\bb=(b_1,\cdots,b_n)$
$$
{f^{[\star]}_{g,(\ba, \bb)}}:= \int_{\M_{g,n}}  \Big( \prod_{i=1}^n  \psi_i^{b_i}\Big) \cup\Omega^{[\star]}_{g,n}(\phi_{a_1},\cdots,\phi_{a_n}).
$$
Here  $r:=\frac{1}{\n} (|\ba|+|\bb|-n)$, and $|\ba|:=\sum a_i$.
\end{definition}

Our goal is to study the $[0]$ theory, using the $[1]$ and the $[0,1]$ theories. We first study the $[1]$ theory by  considering additional ``special" insertions:
\begin{multline*}  
\qquad {f^{[1]}_{g,(\ba, \bb), (\ba'\!,\bb')}} := L^{ \sum_{i=1}^m a_i'}\cdot \int_{\M_{g,n+m}}   \prod_{i=1}^n  \psi_i^{b_i}   \prod_{j=1}^m  \psi_{n+j}^{b_j'} \, \cdot
\\  \Omega^{[1]}_{g,n+m}\Big(\phi_{a_1},\cdots,\phi_{a_n}, R(\psi_{n+1})\bar\phi_{a_1'},\cdots, R(\psi_{n+m}) \bar\phi_{a_l'} \Big),\qquad
\end{multline*}
where
$\ba\in \{0,\cdots,\n\!+\!3\}^{\times n}$,  $\ba' \in $ {\small $[\n]^{\times m}$}, and
$\{\bar \phi_a:=L^{-{(\n+3)}/{2}}\sum_\alpha  (-\ft_\alp)^a   \mathbf 1_\alpha\}_{a=1}^\n$ is the ``normalized" basis\footnote{ Recall the flat basis of $\sH_{\npt}$ is given by $\{\phi_a:=p^a =\sum_\alpha  (-\ft_\alp)^a    \mathbf 1_\alpha \}_{a=4}^{\n+3}$, and notice that we can choose $\{\sum_\alpha  (-\ft_\alp)^a    \mathbf 1_\alpha\}_{a=k+1}^{k+\n}$ as a basis of $\sH_{\npt}$ for any $k$.}.
Let $\lfloor \frac{\ba}{\n} \rfloor := \sum_l \big\lfloor\frac{  a_l}{\n}\big\rfloor$

\begin{proposition}\label{polyofomega11}
Let $\n > 3g-3+n+m$ be sufficiently large.  
\begin{enumerate}
\item If {\small $r:= \frac{1}{\n}(|\ba|+|\ba'|+|\bb|+|\bb'| -n-m) \in \mathbb Z$}, then
$$
({Y}/{\ft^\n})^{   g-1+r   }   \cdot  f^{[1]}_{g,(\ba, \bb), (\ba'\!,\bb')}  \in \QQ  [X]
$$
 is a polynomial in $X$ of degree no more than
$$ 3g-3+n  +m-|\bb|- |\bb'| +  \lfloor \frac{\ba}{\n} \rfloor .$$
\item  Otherwise, $f^{[1]}_{g,(\ba, \bb), (\ba'\!,\bb')} = 0$.
\end{enumerate}
\end{proposition}

\begin{proof}
By definition of $\Omega^{[1]}:=R^{[1]}.\omega^{\npt,\tw}$, the $[1]$-potential is given by the sum of  the  stable graph contributions. For each graph $\Gamma$, the contribution is given  via applying the composition rule 
to the following placements: 
\begin{enumerate}	
\item at each leg with insertion $\phi_{a} \bp^b$ (one of the first $n$ legs), we put
$$
 R^*(-\psi_l) \phi_{a} \psi_l^b |_{\npt}= \sum_{\alpha,k} L^{-\frac{\n+3}{2}} L_\alpha^{a-k}   ( R_k)_a^{\alpha}  (-1)^k \psi_l^{k+b}  \mathbf 1_\alpha ;$$
 \item at each special leg with insertion $R(\psi_{l'})\bar\phi_{a_{l'}'}\psi^{b_{l'}'}$ (one of the last $m$  legs), we put
  $$
\quad  R(-\psi_{l'})^* R(\psi_{l'}) \bar \phi_{a_{l'}'}\psi_{l'}^{b_{l'}'}= \sum_\alpha  L^{-\frac{\n+3}{2}}  L^{a_{l'}'}  (-\ft_\alp)^{a_{l'}'} \mathbf 1_\alpha \psi_{l'}^{b_{l'}'} ;
$$

\item at each edge, we  put a bi-vector  
\begin{align} \label{Vzwdeg}
\qquad V(z, w)|_{\npt \times \npt} = &
 \sum_{\alpha,\beta }  \sum_{k,l,s} L^{-3} t^\n
L_\alpha^{1+s-k} L_\beta^{1-s-l}  \!\!\! \cdot (  V_{kl})^{\alpha \beta;s+1} \,\mathbf 1_\alpha\otimes\mathbf 1_\beta
 ;\end{align}

\item at each vertex of genus $g$ with $n$-legs,  we put a map

$$
(-) \mapsto
\sum_{\alp}  L^{\frac{\n+3}{2}(2g_v-2+n)}\frac{1}{s!} (\pr_{g,n,s})_*\, \omega_{g,n+s}^{\ffp_\alp, \tw} \Big(-, T_\alpha(\psi)^{\otimes s} \Big) , 
$$  
where
$
T_\alpha(z)= 
L^{\frac{\n+3}{2}}\sum_{k\geq 1}   (\mathbf 1^\alpha, R_k^* \mathbf 1)    (-z)^{k+1} \mathbf 1_\alpha = \sum_{k\geq 1} L_\alpha^{-k} \, (R_k)_0^\alpha \  (-z)^{k+1} \mathbf 1_\alpha.$
\end{enumerate}
Denote $L_v$ (resp. $L'_v$) the set of ordinary legs (resp. special legs respectively) over $v$.
We estimate the degree of the legs, edges, tails contributions at each vertex of level $\pt_\alpha$.  By using Lemma \ref{keylem},   we obtain that:
\begin{itemize}
\item 	
the factor involving $L_\alpha$, $L$ and $Y$ is 
\begin{multline*} 
\quad \quad  \  (-t_\alpha)^{(\n+3)(g-1)} L^{\frac{\n+3}{2}(2g_v-2)}  \prod_{l\in L_v} \!\!L_\alpha^{a_l-k_l}\cdot   \prod_{l'\in L'_v} \!\!L_\alpha^{a_{l'} }   \cdot \prod_{f=(e,v),\atop e\in E_v} \!\!  (t\cdot L)^{\frac{\n}{2}}L_\alpha^{1+s_f-k_{f}}  \cdot \prod_{t}   L_\alp^{-k_t}
\\
  \quad   =   L_\alpha^{\sum_{l\in L_v} \!(a_l+b_l-1)+\sum_{l'\in L'_v} \!(a_{l'}+b_{l'}-1)+\sum_{e\in E_v} \!s_{(e,v)}}L^{ \n(g_v-1+\frac{|E_v|}{2})}    t^{\n (g_v-1+\frac{|E_v|}{2})} ,
    \end{multline*}  
where we have used   $n_v=|L_v|+|L'_v|+|E_v|$, and
$$ \textstyle
\qquad \qquad  \sum_{t}  k_t+\sum_f k_f+\sum_l (k_l+b_l) +\sum_{l'}b'_{l'} =  3g_v-3+n_v ;
$$
\item
the total $X$-degree of the tail, edge and leg  contributions at the vertex   is   at most
\begin{align}\label{1-vX}
 \qquad\qquad &\textstyle \sum_t k_t +\sum_{f=(e,v), e\in E_v}(k_f+\frac{1}{2})+\sum_{l\in L_v} (k_l+\lfloor \frac{a_l}{\n} \rfloor) \\
&\textstyle =3g_v-3+n_v+\frac{|E_v|}{2} +\sum_{l\in L_v}  \lfloor \frac{a_l}{\n} \rfloor  -\sum_{l\in L_v}b_l -\sum_{l'\in L'_v}b'_{l'}. \nonumber \end{align}
\end{itemize}

  For each graph we may forget the hour decoration of each vertex to obtain  an  ``hour-free graph". 
  For each vertex $v$ in  an   ``hour-free" graph, we may sum up its all possible hours $\alpha=1,\cdots,\n$ 
and extract  a multiplicative factor  $L_\alpha^{r_v}$ with 
$$\textstyle r_v:=\frac{1}{\n}\big( {\small \text{$\sum_{l\in L_v} (a_l+b_l-1)+\sum_{l'\in L'_v} (a_{l'}+b_{l'}-1)+
\sum_{e\in E_v} \!\!s_{(e,v)}$}} \big).
$$  
By fixing a choice in each summand of (1)-(4) above, such extraction can be done for all vertices  at once.  Since  
$\sum_\alpha L_\alpha^k$ vanishes unless $\n | k$, we
see that if some $r_v\notin \ZZ$, the decomposition summand of $(1)-(4)$ contributed by  an   ``hour-free graph" vanishes.

  At each edge $e=(v_1,v_2)$, by the form of \eqref{Vzwdeg},  we see $s_{(e,v_1)}+s_{(e,v_2)}=0$.  This gives
  $$ \textstyle
r :=\sum_v r_v= \frac{1}{\n}(|\ba|+|\ba'|+|\bb|+|\bb'| -n-m) .
$$
   The argument above   proves the second statement.

Now we evaluate the contributions of all the vertices together.  After multiplying them over all vertices we have 
\begin{enumerate}
\item
the factor involving $L$ and $Y$ (using $L_\alp^\n=(t\, L)^\n  =  \ft^\n \cdot Y^{-1}$) becomes
 $$
(t\, L)^{\n\sum_v  r_v} (t\, L) ^{\sum_v {\n(g_v-1+|E_v|/2)}   } = (Y /  \ft^\n )^{-(g-1+r) } ;
$$
\item the total degree of $X$ of contributions of $\Gamma$  is the sum  of \eqref{1-vX} over all vertices $v$, which equals
$$3g-3+n+m+   \lfloor \frac{\ba}{\n} \rfloor -|\bb|-|\bb'|.
$$
\end{enumerate}
 Multiply (1) with (2), and sum over all graphs. The first statement is proved.
\end{proof}


  \subsection{Vanishing properties of $[0]$-theory} \label{Vanishing0theory}
Recall that we have computed 
\begin{align}  \label{R0z}
R^{[0]}(z)^* \mathbf 1  
   =     \varphi_0 +O(z^{\n-3}),  \and
R^{[0]}(z)^* p     =     z   B  \cdot  \varphi_0 +   \varphi_1 +O(z^{\n-2})  ,
\end{align}
 in \cite[Example 5.4]{NMSP2}.  Furthermore, 
 $R^{[0]}(z)^*$ satisfies the following ``QDE"\footnote{ We recall the explicit formulas for $A^Q\in  \End{\sH_Q}$ and $A^M \in \End{\sH}$ that were proved in \cite{NMSP2}
 \begin{align*} 
A^Q = \begin{pmatrix}0 &&& \\I_{11}&0&&\\&  {I_{22}} &0&\\&& I_{11} &0  \end{pmatrix} \quad \and \quad
 (A\msp)^i_{j} =  \begin{cases} 1  & \text{ if } {i=j+1}\\
  c_{j+1}  q - \delta_{i,4}t^\n  & \text{ if }  {i=j-\n+1}\\ 0 & \text{ otherwise }  \\  \end{cases}
\end{align*}
where $(c_{j})_{j=\n,\cdots,\n+4} =  (120,770,1345,770, 120)$. The same QDE matrix also appears in the proof of Lemma \ref{keylem}, see \cite[Sect.\,5 and Appendix A]{NMSP2} for more details. }:
\beq \label{QDEforR0}
  z\, D   R^{[0]}(z)^* = R^{[0]}(z)^* \cdot A\msp - A^Q \cdot R^{[0]}(z)^*.
\eeq
We have the following general property for $R^{[0]}(z)$:
\begin{lemma} \label{Rpreservedeg}
We introduce the   mod-$\n$ degrees by letting
$$
\deg \psi=1,\quad  \deg \varphi_j=j = \deg \phi_j.\quad   $$ 
Then, $R^{[0]}$-matrix preserves the mod-$\n$ degree.
 Furthermore,  let $\bar j:= j - \n  \lfloor \frac{j}{\n} \rfloor$ and $\varphi_j:=0$ for $j>3$, we have the following key property:
\beq  \label{degcontofR0}
R^{[0]}(z)^* \phi_j = c'_j q^{ \lfloor \frac{j}{\n} \rfloor}\cdot  \varphi_{\bar j}+  O(z^{\bar j -3}) \quad \text{ for  }  j=0,\cdots,\n+3, 
\eeq
where $(c'_{j})_{j=0,\cdots,\n+3} =  (1,\cdots,1,-120,-890,-2235,-3005)$.
\begin{proof}
Recall $
R^{[0]}(z)  = S\msp(z)  S^Q(q',z)$.
Since the  local and global S-matrices preserve mod-$\n$ degrees, the $R^{[0]}$-matrix preserves the mod-$\n$ degree as well.
 Furthermore,  because $\deg \varphi_i \leq 3$,  we obtain the $O{(z^{\bar j-3})}$ in  \eqref{degcontofR0}.  The leading term is from \eqref{QDEforR0}.
\end{proof}
\end{lemma}

 The shape of $R^{[0]}$ gives us  control  on $f^{[0]}_{g,(\ba,\bb)}$. The followings are the most direct ones.
\begin{lemma}\label{modNdeg}
If $r \notin \mathbb Z$ then
$f^{[0]}_{g,(\ba, \bb)} = 0$.
\end{lemma}
\begin{proof}  First by Lemma \ref{Rpreservedeg}, each edge(in the $R^{[0]}$ action on $\Omega^{Q,\tw}$) contributes the mod-$\n$ degree $2$.    Secondly,  observe that, for quintic CohFT, 
$ \int_{\M_{g,n}}  \Omega^{Q}_{g,n}( \otimes_{i=1}^n \phi_{a_i} \psi^{b_i}) =0$
unless $\sum_i (a_i +b_i) = n$.  The same statement holds when $ \Omega^{Q}_{g,n}$ is substituted by  $\Omega^{Q,\tw}_{g,n}$, by Remark  \ref{Q-to-tw}. The lemma follows by summing up the mod-$\n$ degrees over vertices and edges in arbitrary graph defining $\Omega^{[\star]}=R^{[0]} . \Omega^{Q,\tw}$.  
\end{proof}
We will assmue $r$ is an integer in the remainder of this paper.
\begin{lemma} \label{conditionforr}
Suppose $\n  >  3g-3+3n $. Let   $\ba:=(\bar a_1,\cdots,\bar a_n)$ with
 $$ \textstyle
\bar a_j := a_j - \n  \lfloor \frac{a_j}{\n} \rfloor, \and  r^{\sim}:=r- \lfloor \frac{\ba}{\n} \rfloor= \frac{|\bar \ba|+|\bb|-n}{\n}.
$$
We have $r^{\sim} \in\mathbb Z_{\geq 0}$; and if $r^{\sim} \neq 0$ then
$f^{[0]}_{g,(\ba, \bb)} = 0$. Namely,
$$|\bar \ba |+|\bb| \neq n \quad \Longrightarrow \quad  f^{[0]}_{g,(\ba, \bb)} = 0.$$
\end{lemma}
\begin{proof}
By $\n>3g-3+3n$ and the stability condition $3g-3+n\geq 0$, we have $-\n<-2n\leq -n$. Since $r^{\sim}=r- \lfloor \frac{\ba}{\n} \rfloor$ is an integer  we must
have $ r'=\frac{|\bar \ba| +|\bb| -n}{\n} \geq  \lfloor \frac{-n}{\n}  \rfloor = 0$.  
 This proves  the first statement.

Next we prove the vanishing result.
By definition,   if $r^{\sim}>0$
$$
 |\bar \ba| =r ^{\sim} \cdot\n  - (|\bb|-n) \geq  \n  - (|\bb|-n).
$$
By definition of $R$-matrix action, we write $f^{[0]}_{g,(\ba, \bb)}$ as a sum of stable graph contributions.
At each vertex $v$ the contribution is of the form
$$
\int_{\M_{g_v,n_v}} \Omega_{g_v,n_v}^Q\Big(\sbigotimes_{l\in L_v}R^{[0]}(-\psi_l)^* \phi_{a_l} \psi^{b_l}  \sbigotimes_{f=(e,v),e\in E_v} C_f(\psi_f) \Big),
$$
where $C_f$  is  from edge contributions.
By using \eqref{degcontofR0}, we see that,   if $r^{\sim}>0$,   the total degree of psi-classes of all vertices is at least
\beq  \label{degcontv}
 |\bar \ba| -3n + |\bb|  \geq \n -2n.
\eeq
On the other hand, the graph contribution vanishes if for any vertex $v$,
$$\textstyle
\sum_{l\in L_v}  (\bar a_l-3+ b_l) >3g_v-3+n_v.
$$
 Hence it vanishes if
$$ \textstyle
 |\bar \ba|-3n + |\bb|  > \sum_v (3g_v-3+n_v) = 3g-3+n -|E| .
 $$
 By the condition $\n>3g-3+3n$ and \eqref{degcontv} we finish the proof.
\end{proof}
\begin{corollary} \label{vanishcondition}
If $f^{[0]}_{g,(\ba, \bb)} $ is nonzero,  we have
\begin{align*}
r:=\frac{1}{\n}(|\ba|+|\bb|-n)=\big\lfloor \frac{\ba}{\n} \big\rfloor:=\sum_l \big\lfloor\frac{  a_l}{\n}\big\rfloor=\#\big\{ i :  a_i \ge \n\} .
\end{align*}
\end{corollary}
\begin{corollary}
If $f^{[0]}_{g,(\ba, \bb)} $ is nonzero,  we have
\beq  \label{01degbd}
g-1+r \leq 3g-3+r+n-|\bb|.
\eeq
\end{corollary}
\begin{proof}
If $g\geq 1$, \eqref{01degbd} follows from the non-vanishing condition $|\bb| \leq n$. If $g=0$, we have the non-vanishing condition
$|\bb| \leq 3g-3+n = n-3$. Hence
$
g-1+r = -1+r <r  -3+ n- |\bb|$. We finish the proof.
\end{proof}

  \subsection{Polynomiality of $[0]$-theory}
In the last subsection, we want to give the similar degree estimate for $[0]$-theory as what we have done for $[1]$-theory in Proposition \ref{polyofomega11}.

 {
 We introduce the $[0]$-potential with special insertions:
\begin{align*}
\quad {f^{[0]}_{g,(\ba, \bb),(\ba', \bb')}}  :=  \int_{\M_{g,n+m}}  \Omega^{[0]}_{g,n+m}
\Big(\sbigotimes_{l=1}^n \phi_{a_l}\psi_i^{b_l}, \sbigotimes_{l'=1}^m E_{a_{l'}',b_{l'}'}(\psi_{n+l'}) \Big), \quad
\end{align*}  
where  the indices
$\ba\in \{0,\cdots,\n\!+\!3\}^{\times n}$,\        $\ba' \in  [\n]^{\times m}$,    $\bb\in\ZZ_{\geq 0}^{\times n}$,  $\bb'\in \ZZ_{\geq 0}^{\times m}$  and
\beq \label{Eabdefn}
\textstyle E_{a'\!,b'}(\psi):=L^{-a'}\cdot  \Coeff_{z^{b'}} \frac{1}{\psi+z} \big(R(\psi)-R(-z)\big)\bar\phi^{a'} 
\eeq with 
  the dual basis  $\{\bar \phi^{a'}\!\!:=\! {\small\text{$L^{ \frac{(\n+3)}{2}}$}}\! \sum_\alpha  (-\ft_\alp)^{-a'}  \mathbf 1^\alpha\}_{a'=1}^\n\!$
 of the ``normalized"  basis $\{\bar\phi_i\}_{a=1}^{\n}$.

 Using
 $\textstyle
{V^{01}}(z,w) =\sum_{a=1}^{\n} \frac{R(z)  -R(-w)  }{z+w}  \bar\phi^a \otimes   R(w)\bar\phi_a
$,  
one has

 \beq \label{V01deg} \textstyle
  {V^{01}}(z,w) =  \sum_{a=1}^{\n} \sum_{b\geq 0}E_{a,b}(z)w^b \otimes   L^a\,R(w)\bar\phi_a.
\eeq

\begin{lemma}
  We have  $L^{-a+k}(\phi_a, R_k \bar\phi^b) \in \mathbb Q(\ft^\n)[X]_{k+\lfloor  \frac{a}{\n} \rfloor}$ and
\begin{align}\label{Rphimod}\textstyle
(\phi_a, R_k \bar\phi^b) =0\quad \text{if} \quad  a-k\neq b \mod \n.
\end{align}
\end{lemma}
\begin{proof} It follows from Lemma \ref{keylem} and
\begin{align*} 
(\phi_a, R_k \bar\phi^b) 
=&\textstyle \sum_{\alpha}  (-\ft_\alp)^{-b}   L_\alp^{a-k} (R_k)^\alp_a= \sum_{\alpha}   (\zeta_\n^\alp t)^{-b}   (\zeta_\n^\alp t\,L)^{a-k} (R_k)^\alp_a.
\end{align*}
Here we have used $\sum_{\alpha}   (\zeta_\n^\alp)^m =0$ unless $\n | m$  because $\n$ is a prime.
\end{proof}

\begin{lemma}
We have
$f^{[0]}_{g,(\ba, \bb),(\ba', \bb')} =0$,
unless
$$\textstyle
r:= \frac{1}{\n} {\small \text{$(|\ba|+|\bb|-|\ba'|-|\bb'|-n+s)$}} \in \mathbb Z.$$
\end{lemma}
\begin{proof}
 Recall  the mod-$\n$ degree introduced in Lemma \ref{Rpreservedeg}.
 Apply \eqref{Rphimod} to $R_k\bar\phi^b=\sum_{s=1}^{\n+3} (R_k\bar\phi^b,\phi_s)\phi^s$ one sees the mod-$\n$ degree of $ R_k\bar\phi^b$ is $3-(k+a)$. One then calculates
 the mod-$\n$ degree of $ E_{a,b}$ is  $2-a-b$. The same reasoning as proof of Lemma \ref{modNdeg} applies.
\end{proof}

 When $s=0$, we have $f^{[0]}_{g,(\ba, \bb),(\ba', \bb')} = f^{[0]}_{g,(\ba, \bb)}$ and $r= \frac{1}{\n} (|\ba|+|\bb| -n)$.

\begin{definition}
For any $(g,n)$, we introduce a statement 
\begin{align*} 
\fS_{g,n}  \  = \  {\small \text{ $ `` \  \forall \ba \in \{0,1,2,3\}^{\times n},  \forall \bb \in \ZZ_{\geq 0}^n  \quad
  (Y/\ft^\n)^{g-1} \!\cdot \! f^{[0]}_{g,(\ba, \bb)} \in  \mathbb Q[X]_{ 3g-3+n    -|\bb|} \  "$}}.
\end{align*}
We also introduce stronger statements
\begin{align}  \label{poly2}
  \fS_{g,n}'
  \ = \  
``&{\small \text{ $\forall k,s\geq 0$ with $\ell+s=n$,  $ \forall \ba \in  \{0,1,2,3,\n,\cdots,\n\!+\!3\}^{\times \ell} ,    \ba' \in  [\n]^{\times s}, (\bb,\bb)'\in  \ZZ_{\geq 0}^n $ }}\   \nonumber\\
 & \quad (Y/\ft^\n)^{g-1+r+s} \cdot  {f^{[0]}_{g,(\ba, \bb),(\ba', \bb')}}\in \mathbb Q[X]_{3g-3+\ell+ 2s  +\lfloor \frac{\ba}{\n} \rfloor -|\bb|+|\bb'|}  \  ".
\end{align}
 \end{definition}

One of the main result in the next section \footnote{ See \S \ref{polyof0} for the proof}, is the following lemma.

\begin{lemma}\label{bigstar}
 Suppose $(g,n)$ satisfies $2g-2+n>0$. Let $\n> {3g +n}$.   If    the statement $\fS_{h,m}$ holds for any $(h,m)<(g,n)$ and  $3h+m\leq 3g+n$. 
Then for any $(h,m)$ such that $(h,m)<(g,n)$ and  $3h+m\leq 3g+n$,  the statement  $\fS_{h,m}'  $ holds.
\end{lemma}


By using the above two lemmas, we prove

\begin{proposition} \label{Sgl}
Let $\n> 3g+n$. Then
  \beq  \label{poly1a}
 \  \  \forall \ba,\bb \in [\n+3]^{\times n}, \qquad
 (Y/\ft^\n)^{g-1+r} \!\cdot \! f^{[0]}_{g,(\ba, \bb)} \in  \mathbb Q[X]_{ 3g-3+n +\lfloor \frac{\ba}{\n} \rfloor   -|\bb|}  .
\eeq

\end{proposition}

\begin{proof}
By definition of $R^{[0]}$-matrix action, $f^{[0]}_{0,(\vec{a}, \vec{b})}$ is equal to a graph sum.  In case $(g,n)=(0,3)$, there is only one graph with a single vertex and with no psi-classes insertions. In this case $r=\sum_l \lfloor \frac{a_l}{\n}  \rfloor$.   
By the property of $\Coeff_{z^0} R^{[0]}(z)$ (c.f. \eqref{R0z} and \eqref{degcontofR0}), one calculates (for any $a_1,a_2,a_3$)
$$
(Y/\ft^\n)^{0-1+r} \!\cdot \! f^{[0]}_{0,(\ba, 0)} =    (Y/\ft^\n)^r \cdot  (Y/\ft^\n)^{-1} \cdot I_0^{2} I_{11}^2I_{22} \cdot C\,q^r  =C\cdot X^r
$$
for a $C\in\QQ$ (which is a product of $c_j$'s defined as in \ref{keypropR}).  This is a polynomial in $X$ of degree $0-3+3+r$. Here we have used $\left<H,H,H \right>^Q=I_{22}/I_{11}$ and $\left<1,H,H^2 \right>^Q=1$ (c.f. \cite[Appendix A]{NMSP2}).

\medskip

We now prove the proposition by     induction on $(g,n)$ under the lexicographical order. We will use  Proposition \ref{polyofomega11}.
Fix $g, n$ such that $2g-2+n>0$, and from induction hypothesis assume \eqref{poly1a} holds for any $(h,m)<(g,n)$. Then    $\fS_{h,m}$ holds  for any $(h,m)<(g,n)$ and $3g+m\leq 3g+n$.
 By Lemma \ref{bigstar}, $\fS_{h,m}'   $ holds for any $(h,m)<(g,n)$   and also $3h+m\leq3g+n$.  

 Now for any $\ba,\bb \in [\n+3]^{\times n}$
we consider the $[0,1]$-potential $f^{[0,1]}_{g,(\ba, \bb)}$. Suppose $f^{[0,1]}_{g,(\ba, \bb)}$ vanishes, \eqref{poly1a} holds.  
Suppose $f^{[0,1]}_{g,(\ba, \bb)}\neq 0$. By Corollary \ref{vanishcondition}, $ g-1+r\in \mathbb Z$.
Since $\n> 3g-3+n \geq \sum_i b_i\geq 0$, we have $\frac{3g-3+n-\sum b_i}{\n}<1$,  which implies
 $$
 \textstyle
  \lfloor g-1+\frac{3g-3+\sum_{i=1}^n a_i}{\n} \rfloor =g-1+r .$$
By Theorem \ref{Omega01}, $f^{[0,1]}_{g,(\ba, \bb)}\neq 0$ is a polynomial in $X$ of  degree 
$\deg_X f^{[0,1]}_{g,(\ba, \bb)} \leq g-1+r \leq 3g-3+n-|\bb|+r$ (by \eqref{01degbd}). 

On the other hand we apply \eqref{bip01} to this $[0,1]$-potential. There is a bipartite graph with only a single genus $g$ level $0$ vertex, which we call the leading graph.  It suffices to prove that for any non-leading graph $\Gamma$, 
the contribution 
\begin{align*} 
 \Cont_{\Gamma} \in  \mathbb Q[X]_{3g-3+n-|\bb|+r}.  \end{align*}
Indeed, every $[0]$ vertex of any non-leading graph is applicable for the statement $\fS'_{h,m}$.
 Apply \eqref{V01deg} first, and  Proposition \ref{polyofomega11} at $V_1$, and  \eqref{poly2} at $V_2$,
we obtain
\begin{itemize}
 \item   the degree of the total contribution of $\Gamma$ in $X$ is given by  (with $n_v:=|E_v|+|L_v|$) 
\begin{align*}
\qquad &\sum_{v \in V_0}  \Big(3g_v-3+n_v +|E_v| +\!\sum_{l\in L_v}\big(\lfloor \frac{a_l}{\n} \rfloor-b_l \big) + \!\sum_{e\in E_v} b'_{(e,v)}\Big)+
\\
&\quad+\!
\sum_{v \in V_1}  \Big(3g_v\!-3+n_v  + \!\!\sum_{l\in L_v} \big(\lfloor \frac{a_l}{\n} \rfloor-b_l \big) -\!\! \sum_{e\in E_v} b'_{(e,v)}   \Big) \leq  3g-3+n +\! \lfloor \frac{\ba}{\n} \rfloor \!- |\bb| ;
\end{align*}
\item
 the total factor involving $(Y/\ft^\n)$ is given by
$$\qquad\qquad
(Y/\ft^\n)^{g-1+r} \cdot \prod_{v \in V_0} (Y/\ft^\n)^{-(g_v-1+r_v+s)} \prod_{v \in V_1} (Y/\ft^\n)^{-(g_v-1+r_v)} = 1.
$$
\end{itemize}
 This finishes the induction.
\end{proof}

\vspace{1cm}

\section{From $\nmsp$ $[0]$-theory to the CohFT $\Omega^{\bA}$ via $R^X$-action} \label{NMSPruleproof}

We introduce the $\Omega^{\bA,\vec{0}}$-theory via the following $R$-matrix action
\begin{align*}
\Omega^{\bA,\vec{0}} := R^{\bA,\vec{0}}. \Omega^{Q,\tw}
,\qquad \text{ where} \ \ \ \ R^{\bA,\vec{0}}(z) := R^{\bA,\cG}|_{c_{1a}=c_{1b}=c_2=c_3=0}.
\end{align*} 
We also introduce the generating function:  for $a_i=0,1,2,3$ ($i=1,\cdots,n$)
\begin{align*}
f^{\bA,\vec{0}}_{h;\ba,\bb}:=
(-5Y/t^\n)^{h-1} \,\int_{\M_{h,n}}  \psi_{1}^{b_1}\cdots\psi_{n}^{b_n}\cup \Omega^{\bA,\vec{0}}_{h,n}(  \tp_{a_1},\cdots,  \tp_{a_n}).
\end{align*}
By the relation \eqref{twnormalized} (and the explanation in \S \ref{cohftNQ}), we see that
this is equivalent to the definition in \S \ref{NMSPrule} (with gauge  ${c_{1a}=c_{1b}=c_2=c_3=0}$ and with 
$q\mapsto q'$).

In this section  we will prove the polynomiality of  $\Omega^{\bA,\vec{0}}$-theory,
which is closely related to  $f^{\bA,\vec{0}}_{h,n}$ defined in the introduction, via the polynomiality of the $[0]$-theory. 
In the end, we will prove   Theorem \ref{thm2}.



To extract information from the $\nmsp$-$[0]$ theory, we consider the following matrix factorization which defines $R^X(z)$ :
\footnote{ Since $R^{\bA,\vec{0}}$ is invertable, such matrix $R^X(z)$ exists and can be calculated.} 
\begin{align}\label{0-X-A}
R^{[0]}(z) = R^X(z) \cdot R^{\bA,\vec{0}}(z),
\end{align}
where the matrix  is   under the following basis
$$
\{ \varphi_i \}_{i=0}^3  \xrightarrow{\quad R^{\bA,\vec{0}} \quad}  \{  \varphi_i \}_{i=0}^3  \xrightarrow{\quad R^X \quad} \{\phi_a\}_{a=0}^{\n+3}.
$$
(Recall $ \tp_i:=I_0 I_{11}\cdots I_{ii} H^i$ for $i=0,1,2,3$.)
By definition,
we see
\begin{align*}
\Omega^{[0]} = R^X.\Omega^{\bA,\vec{0}}.
\end{align*}

%

\subsection{Properties of $R^X$}  \label{formulaforRX}
  The advantage of the factorization \eqref{0-X-A} is: 
\begin{lemma} \label{RXprop}
The following properties hold for $k<\n-3$.
\begin{itemize}
\item[1.] if $j\neq k+a \mod \n$, then $(\phi_j, R_k^X \varphi^a) = 0$;
\item[2.] if $j < \n$, we have $(\phi_j, R_k^X \varphi^a) \in X \, \mathbb Q[X]_{k-1}$;
\item[3.] if $j  \geq  \n$, we have $ (\phi_j, R_k^X \varphi^a) \in q\, \mathbb Q[X]_{k}$;
\item[4.]  the $C^X(z) \in \End(\sH_Q, \sH)[z]$ defined below satisfies $R^X(-z)^* C^X(z) =\id_Q \in \End \sH_Q$,
\begin{align*} &C^X(z) :=
   \left( \!\!\!\!\!\!\!\!\!\!\!\arraycolsep=1.2pt\def\arraystretch{1}
   \begin{array} {*{3}{@{}C{\mycolwdd}}c}
   1&-z\cdot{\frac {24\,X}{625}}&{z}^
{2}\cdot{\frac {24\,X}{625}}&{z}^{3} \cdot\left( -{\frac {576\,{X}^{2}}{390625}}-{\frac {24
\,X}{625}} \right) \\ \noalign{\medskip}0&1&-{z}\cdot{\frac {202\,X}{625}}&{z
}^{2}\cdot \left( {\frac {4848\,{X}^{2}}{390625}}+{\frac {226\,X}{625}}
 \right) \\ \noalign{\medskip}0&0&1&-{z}\cdot{\frac {649\,X}{625}}\\ \noalign{\medskip}0&0&0&1
   \\
   \cdots &\cdots &\cdots &\cdots  \end{array} \ \right)
\end{align*}
where the dots represent zeros.
\end{itemize}
\end{lemma}
\begin{proof}
Recall the QDE \eqref{QDEforR0} for $R(z) |_Q$ is
$$
zD R^{[0]}(z)^*  =  R^{[0]}(z)^* \cdot A\msp -A^Q \cdot R^{[0]}(z)^*
$$
together with the definition $R^{[0]}(z)^* =R^{\bA,\vec{0}}(z)^* R^X(z)^*$, we obtain
$$
zD (R^{\bA,\vec{0}}(z)^* R^X(z)^*   ) =  (R^{\bA,\vec{0}}(z)^* R^X(z)^*   ) \cdot A\msp -A^Q \cdot (R^{\bA,\vec{0}}(z)^* R^X(z)^*   ).
$$
Then we have
\beq \label{QDEforRX}
 R^X(z)^*   \cdot A\msp=   zD \Big(  R^X(z)^* \Big)   +R^{\bA,\vec{0}}(-z)\Big[  (zD+A^Q)  R^{\bA,\vec{0}}(z)^* \Big] \cdot   R^X(z)^*.
\eeq
A direct computation shows that
\beq  \label{EQURA}
R^{\bA,\vec{0}}(-z)\Big[  (zD+A^Q)  R^{\bA,\vec{0}}(z)^* \Big]  = {\footnotesize  \left(\!\!\!\!\begin{array} {*{4}{@{}C{\mycolwdb}}} 0&0&0&-{\frac {24\,X}{625}{z}^{4}}
\\ \noalign{\medskip}1&0&- \frac{2\,X}{5}{z}^{2}&0\\ \noalign{\medskip}0&1&-Xz&0
\\ \noalign{\medskip}0&0&1&-Xz\end {array} \!\!\right)}.
\eeq
Hence we have an algorithm which recursively compute
$R^X(z)^* \phi_i$ ($i>0$)
 from $R^X(z)^* \phi_0 = 1+O(z^{\n-3})$ (which follows from the definition of $R^X(z)$ \eqref{0-X-A} and the formula \eqref{R0z}).  
 Furthermore,  since  the matrix in the algorithm always  increases   the mod $\n$ degree (see definition in Lemma \ref{Rpreservedeg}) by $1$ simultaneously, we see the first three statements hold. The last one is obtained by direct computation of the leading term of $R^X$ (see Appendix \ref{explicitformulaR}) and using the vanishing properties of $R^X$ in the first three statements.
\end{proof}

\begin{remark} 
By the first statement of the above lemma,
we see
$$R^X(-z)\sta 1=1_Q+O(z^{\n-3}) \and  T_{R^X}(z)=O(z^{\n-2}) .$$
\end{remark}

\subsection{Polynomiality of $\Omega^{\bA,\vec{0}}$-theory}
\begin{lemma}\label{fAvan}
We have $f^{\bA,\vec{0}}_{h;\ba,\bb} =0$ when $ \sum_i (a_i+b_i) \ne n$.
\end{lemma}
\begin{proof}
Just notice that the $R^{\bA,\vec{0}}$ action preserve degrees mod $\n$.
\end{proof}

\begin{proposition} \label{polyoffa} Assume $\fS_{h,m}\!$
holds for any $(h,m)<(g,n)$  and $3h+m\leq 3g+n$.  Then 
\begin{align} \label{polynomialityfA0}
 \  \forall  (h,m)<(g,n),  3h+m\leq 3g+n,     \forall \ba \in  \{0,1,2,3\}^{\times m},  f^{\bA,\vec{0}}_{h;\ba,\bb} \in    \mathbb Q[X]_{3h-3+m-\sum_i b_i }.
\end{align}
\end{proposition}
\begin{proof}
For $a_i\in\{0,1,2,3\}$  we define
\begin{align*} 
\tilde f^{[0]}_{h;\ba,\bb}  :=\,\int_{\M_{h,m}}   (R^X\!.\Omega^{\bA,\vec{0}})_{h,m}(  C^X(\bp_1) \tp_{a_1} \bp_{1}^{b_1},\cdots,   C^X(\bp_m) \tp_{a_m} \bp_{m}^{b_m}),
\end{align*}
where recall that $C^X(z):= \sum_{k=0}^3 C_k^X z^k$ is defined in \S \ref{formulaforRX} such that
$$R^X(-z)^* C^X(z) =\id_Q.
$$ 
 Furthermore,  $C_k^X$ satisfies the following property
$$
(\phi^j, C_k^X \varphi_a)     \  \text{vanishes,  if  } j>3 \text{ or  } j\neq k+a  \quad \and \quad 
(\phi^j, C_k^X \varphi_a)   \in  \mathbb Q[X]_k   .
$$
By the above property of $C^X_k$ and the condition $\fS_{h,m}$,  
we see
\beq  \label{f0ab}
 \tilde f^{[0]}_{h;\ba,\bb}  \in  (Y/t^\n)^{-(h-1)}  \mathbb Q[X]_{3h-3+m-\sum_i b_i}. 
\eeq  

We note that in the stable graph summation formula of $\tilde f^{[0]}_{h;\ba,\bb}$ via the $R^X$-matrix action on 
$\Omega^{\bA,\vec{0}}$, there is this ``leading" graph that is a single genus $h$ vertex with $m$-insertions $\tp_{a_1} \bp_1^{b_1},\cdots,  \tp_{a_m} \bp_1^{b_m}$.
This graph contributes to $f^{\bA,\vec{0}}_{h;\ba,\bb}$. The contribution of any other non-leading graph $\Gamma$ will be of the form
\beq\label{star}
\qquad \qquad \Big(  \sbigotimes_{v\in V(\Gamma)}  f^{\bA,\vec{0}}_{g_v,n_v} \Big)  \Big( \sbigotimes_{i=1}^m \tp_{a_i} \bp_i^{b_i}  \,\,  \sbigotimes \sbigotimes_{e\in E(\Gamma)} V_X(e)  \Big),
\eeq
where $ f^{\bA,\vec{0}}_{g_v,n_v}: H_\bA^{\otimes n_v} \rightarrow Q[\![q]\!]$ are the linear maps $f^{\bA,\vec{0}}_{g_v,n_v}(-) :=\left<- \right>^{\bA,\vec{0}}_{g_v,n_v}$ and \begin{align*}
V_X(e) := & \ \frac{\sum_{i=0}^3\tp_i\otimes \tp^i-\sum_{j=0}^{\n+3} R^X\!(-\bp_{v_1})^* \phi_j \otimes R^X\!(-\bp_{v_2})^* \phi^j}{\bp_{v_1}+\bp_{v_2}} 
\end{align*}
is a bivector\footnote{ We have used $(\tp_i, \tp_{3-i})= 5I_0^2I_{11}^2 I_{22}=5Y$.} in $H_\bA\otimes H_\bA$, for $v_1$ and $v_2$ 
incident to the edge $e$.
  Furthermore,  by Lemma \ref{RXprop} and  by using $\varphi^i = (5Y/t^\n )^{-1} \varphi_{3-i}$,    we have the following degree estimate:
\beq \label{degreeVX}
\forall k_1,k_2\qquad  (Y/t^\n)\cdot {\Coeff_{\psi_{v_1}^{k_1}\psi_{v_2}^{k_2}} V_X(e)} \quad  \in  \quad H_\bA^{\otimes 2} [X]_{k_1+k_2+1}.
\eeq


We now prove the polynomiality  by induction.
First we see for $(h,m)=(0,3)$, the ``leading" graph is the only graph. The theorem for this case follows directly.

Next we assume the polynomiality \eqref{polynomialityfA0} for all genus $h'$ with $m'$ insertions with $(h',m')<(h,m)$  and $3h'+m'\leq 3g+n$. 
Recall \eqref{f0ab} is equal to the graph sum of \eqref{star}. For a ``non-leading" graph $\Gamma$,
\begin{enumerate}
\item the factor of \eqref{star} associated to $\Gamma$  involving $Y$ is in total
$$
 \prod_v  (Y/\ft^\n)^{-(g_v-1)}  \prod  (Y/\ft^\n)^{-E} =   (Y/\ft^\n)^{-(h-1)};
$$
\item the $X$-degree of \eqref{star} associated to $\Gamma$ is in total 
\begin{align*} 
&\sum_v (3 g_v-3+n_v-\sum_{e\in E_v} b_{(e,v)} -\sum_{i\in L_v}b_i )      +\sum_{e=(v_1,v_2)} ( b_{(e,v_1)}+ b_{(e,v_2)}+1)\nonumber\\
= & (\sum_v 3 g_v)  -  3|V|+ 3 |E| +m  -\sum_i b_i = 3h-3+m -\sum_i b_i ;
\end{align*}
\end{enumerate}
as desired.
 This finishes the induction and proves $f^{\bA,\vec{0}}_{h;\ba,\bb} \in \mathbb Q[X]_{3h-3+m-\sum_i b_i }$.
\end{proof}


\subsection{Proof of Lemma \ref{bigstar}}  \label{polyof0}
We state the lemma we will use to prove Lemma \ref{bigstar}.

\begin{lemma} \label{key01}
We have the following degree estimate:  whenever $b''<\n-3$ 
\beq \label{RXEprop}
 (Y/\ft^\n)^{-r_E} \cdot   \Coeff_{z^{b''}} \Big(\varphi^{a''}, \!R^X\!(-z)^* E_{a'\!\!, b'}\!(z)\!  \Big) \in   \mathbb Q [X]_{b+b'+1},  
\eeq
where $r_E:={\frac1{\n}({a'\!+b'\!+a''+b''-\n-2})} $ and the LHS of \eqref{RXEprop} vanishes unless $r_E \in \mathbb Z$.
\end{lemma}

\begin{proof}[{Proof of Lemma \ref{bigstar}}]
Assume $\fS_{h,m}$ holds for all $(h,m)< (g,n)$  and $3h+m\leq 3g+n$.   By Proposition  \ref{polyoffa}
$$
\qquad  f^{\bA,\vec{0}}_{h;\ba,\bb} \in \mathbb Q[X]_{3h-3+m-\sum_i b_i }  \quad  \forall (h,m) \!<\! (g,n),  3h+m\leq 3g+n,  \ a_i=0,1,2,3. \
$$  We now look at the statement of $\fS'_{h,m}$, under the assumption $(h,m)<(g,n)$ and $3h+m\leq3g+n$. 
Let   $\ba \in \{0,1,2,3,\n,\cdots,\n+3\}^{\times \ell} $ and $m=\ell+s$.  By $\Omega^{[0]} = R^X.\Omega^{\bA,\vec{0}}$, we obtain 
\beq  \label{f0abab}
{f^{[0]}_{h,(\ba, \bb),(\ba', \bb')}}  =
   \int_{\M_{h,l+s}}
(R^X.\Omega^{\bA,\vec{0}})_{h,\ell+s}\Big(  \phi_{\ba} \psi^{\bb},
E_{\ba'\!,\bb'} (\psi)\Big).  \
\eeq
By applying the $R^X$-action, for each  stable graph $\Gamma \in G_{h,m}$, 
the contribution to \eqref{f0abab} consists of (using Lemma \ref{RXprop})\footnote{ 
We will denote the set of first $m$ (last $s$) legs by $L$ ($L'$ respectively).}
\begin{itemize}
\item at each leg $l\in L$, we have an insertion \footnote{\label{knomorethan3}
  Here we have used $k_l\leq k_l+b_l\leq 3h-3+m$, otherwise the contribution vanishes by dimension reason. Hence we have $k_l\leq   3g-3+n <\n-3$, and by this condition we  see that the only integer $a \in [0,3]$ making  $a_l - k_l\equiv a (\text{mod}\,\n)$
 is $a=\bar a_l -k_l$ (i.e. we must have $k_l\leq \bar a_l\leq 3$). 
 }
$$\qquad \quad
 \Big(  \varphi^{  \bar a_l -k_l   }, \psi_l^{b_l+k_l}  (-1)^{k_l}(R^X_{k_l})^*\phi_{a_l}  \Big) \in  \psi_l^{b_l+k_l}  (Y/\ft^\n)^{-\lfloor \frac{a_l}{\n} \rfloor} \mathbb Q[\![X]\!]_{k_l+\lfloor \frac{a_l}{\n} \rfloor} ;
$$  
\item at each leg $l'\in L'$, we have an insertion \footnote{  Here we have used $b''_{l'}\leq 3h-3+m<\n-3$ for the same reason as above.
}
$$\qquad \quad
  {\psi_{l'}}^{b''_{l'}}  \Coeff_{z^{b''_{l'}} } \Big(\varphi^{a''_{l'}}, \!R^X\!(-z)^* E_{a'_{l'} , b'_{l'} }(z) \Big) \in  {\psi_{l'}}^{b''_{l'}}  (Y/\ft^\n)^{r_{E_{l'}}} \cdot \mathbb Q[\![X]\!]_{b''_{l'}+b'_{l'}+1}  ;
$$ where $r_{E_{l'}}=\frac{1}{\n} (a'_{l'}+b'_{l'}+a''_{l'} + b''_{l'}  -\n -2)$;  \item at each edge $e=(v_1,v_2)$, we have a bivector $V_X(e)$.
\end{itemize}
Here we have used the degree estimate in Lemma \ref{key01}.  Further  we see 
\begin{enumerate}
\item 
the total factor of $(Y/\ft^\n)$ in the contribution of  graph $\Gamma$ to \eqref{f0abab}
is given by 
  $$
\qquad \prod_l \Big(\frac{Y}{\ft^\n}\Big)^{-\lfloor \frac{a_l}{\n} \rfloor}  \prod_{l'\in L'}  \Big(\frac{Y}{\ft^\n}\Big)^{r_{E_{l'}}}  \prod_v \Big(\frac{Y}{\ft^\n}\Big)^{ 1-h_v}\prod_e \Big(\frac{Y}{\ft^\n}\Big)^{ -1 } = \Big(\frac{Y}{\ft^\n}\Big)^{- (r +s+h-1)},
$$
where we have used  $\sum_{l'} (a''_{l'}+b''_{l'})+\sum_l (\bar a_l +b_l)  =  \ell + s$\footnote{ This identity follows from Lemma \ref{fAvan} and the fact that $V_X$  
has cohomology degree two  (by Lemma \ref{RXprop}, see also for \eqref{VXz} the explicit formula).}  and
\begin{align*}\textstyle
\sum_{l'\in L'}  r_{E_{l'}}  +r +s- \sum_l & \textstyle{\lfloor \frac{a_l}{\n} \rfloor} =\, {\frac1{\n}\sum_{l'} ({a_{l'}'\!+b_{l'}'\!+a''_{l'}+b''_{l'}-\n-2})}  \\
& \textstyle  + \frac1{\n}({|\bar\ba|+|\bb|-|\ba'|-|\bb'|-\ell+s})  +s = 0;
\end{align*}
\item
the total degree of the contribution of  graph $\Gamma$ to \eqref{f0abab} in $X$ is no more than
\begin{align*}
&\sum_v {\small \text{$\Big(3h_v-3+n_v   -  \sum_{l\in L_v}(b_l+k_l) -\sum_{l'\in L'_v}b''_{l'} -\sum_{e\in E_v} k_{(e,v)}  $}}\\
&\quad {\small \text{$+\sum_{l\in L_v}(k_l+ \lfloor \frac{a_l}{\n} \rfloor ) +\sum_{l'\in L'_v} (b''_{l'}+b'_{l'}+ 1  )\Big)  +\sum_{e}(k_{(e,v_1)}+k_{(e,v_2)}+1) $}}\\
&  =   3h-3+ \ell + 2s +\lfloor \frac{\ba}{\n} \rfloor   -|\bb|+|\bb'|.
\end{align*} 
Here we have used the degree estimate \eqref{degreeVX}.
\end{enumerate}
  To summarize we obtain
 $$
Y^{h-1+r+s} \cdot  {f^{[0]}_{h,(\ba, \bb),(\ba', \bb')}}\in \mathbb Q[X]_{3h-3+\ell+2s+\lfloor \frac{\ba}{\n} \rfloor -|\bb|+|\bb'|}.
 $$
  This finishes the proof of Lemma \ref{bigstar}, and therefore \eqref{poly1a} is correct
 by Proposition \ref{Sgl}.  
\end{proof}

\begin{lemma}\label{Eab} We have   $\big(\phi^a,  \Coeff_{z^b}E_{a',b'}(z)  \big)=0$ unless $a' +a+b'\!+b=2 \mod \n$.
\end{lemma}
\begin{proof}
By the definition \eqref{Eabdefn} of $E_{a',b'}$, 
\begin{align} \label{EabtoR} \textstyle
\big(\phi_a,  \Coeff_{z^b}E_{a',b'}(z)  \big) = L^{(\n+3)/2} \sum_\alpha L_\alpha^{-a'} (-1)^{b'} \big(\phi_a, R_{b'\!+b+1} \mathbf 1^\alpha\big) .
\end{align}
By  Lemma  \ref{keylem} it vanishes unless $a' = a-(b'\!+b+1) \mod \n$. 
\end{proof}

\begin{proof}[{Proof of Lemma \ref{key01}}]
The vanishing result  follows from Lemma \ref{RXprop} and \ref{Eab}.

 For the degree estimate,  we consider three cases:
 \begin{enumerate}
\item For $a=4,\cdots,\n-1$:  by  Lemma \ref{RXprop} for any $k'<\n-3$, (note $a''=0,1,2,3$ \textsuperscript{\ref{knomorethan3}}) 
\beq \label{RXpropa}
 \quad  (-1)^{k'}(R_{k'}^X)_a=  (\varphi^{a''},  \Coeff_{z^{k'}}R^X(-z)^* \phi_a) \neq 0 \quad \text{ only if }   a=a''+k' .
\eeq 
When it is nonzero, it is a degree $k'$ polynomial in $X\!$. This implies that \footnote{ 
By using \eqref{EabtoR}, $\phi^a\!=\!\phi_{\n+3-a}/5$  and the property \eqref{Rkform} of $(R_k)_a$, we have
\begin{align*} \textstyle
\Coeff_{z^b} \big(\phi^a,  E_{a',b'}(z) \big) =\frac{\n}{5}   (Y/t^\n)^{r_E} (-1)^{b'} 
 (R_{b+b'+1})_{\n+3-a} \in  (Y/t^\n)^{r_E}  \mathbb Q[X]_{b+b'+1}.
\end{align*}
Then each contribution to LHS of \eqref{Xbound1} has degree $\leq (b+b'+1)+k =b'+b''+1$ (here $b''=b+k'$).  
}
\begin{align}\label{Xbound1}
\qquad\qquad(Y/\ft^\n)^{-r_E}\cdot \Coeff_{z^{b''}}(\varphi^{a''}, R^X(-z)^* \phi_a)\big(\phi^a,  E_{a',b'}(z)  \big) \in \mathbb Q[X]_{b'+b''+1}.
\end{align}


%

\item For $a=0,\cdots,3$:   by Lemma \ref{RXprop} for any $k'<\n-3$ we still have \eqref{RXpropa}.    
Further,  when it is nonzero, it
is a degree $k'$ polynomial in $X$.  This implies that \footnote{ 
 By applying $\phi^a\!=\! (\phi_{\n+3-a}\!- {\ft^\n} \!\phi_{3-a})/5$ in \eqref{Xbound2}, the term $\phi_{\n+3-a}$ contributes   the same formula as \eqref{Xbound1}, except that by \eqref{Rkform} the $X$ degree bound is increased by $1= \lfloor \frac{\n+3-a}{\n} \rfloor $. The second term $\ft^\n\phi_{3-a}$ contributes
$ \Coeff_{z^b}  \big( t^\n  \phi_{3-a} ,   E_{a',b'}(z)  \big)   =   \!\sum_\alpha   (\frac{Y}{t^\n})^{r_E} (-1)^{b'} 
 (R_{b+b'+1})_{3-a} \! \cdot Y.$
 With  \eqref{RXpropa} we obtain \eqref{Xbound2}.
 }
\begin{align}\label{Xbound2}
\qquad\qquad (Y/\ft^\n)^{-r_E}\cdot \Coeff_{z^{b''}}(\varphi^{a''}, R^X(-z)^* \phi_a)\big(\phi^a,  E_{a',b'}(z)  \big) \in \mathbb Q[X]_{b'+b''+2}.
\end{align}
%

 

\item For $a=\n,\cdots,\n+3$:   by Lemma \ref{RXprop}, for any $k'<\n-3$, 
$$
\  \   (\varphi^{a''},  \Coeff_{z^{k'}}  R^X(-z)^* \phi_a) \neq 0 \quad \text{ only if }   a-\n=a''+k' .
$$
When it is nonzero, it is a degree $k'$ polynomial in $X$  multiplied by $q$. This implies (by argument similar to (2))
$$
\qquad\qquad (Y/\ft^\n)^{-r_E}\cdot \Coeff_{z^{b''}}(\varphi^{a''}, R^X(-z)^* \phi_a)\big(\phi^a,  E_{a',b'}(z)  \big) \in q\, \mathbb Q[X]_{b'+b''+1}.
$$
\end{enumerate}
 Sum up the process deducing (1),(2),(3). One obtains that, the LHS of \eqref{RXEprop} equals 
\begin{align*}
 \small{\text{$\Contr(1)+\!\!\!\!\!\!\!\! \sum_{ j+k=b'+b''+1,\atop  0\leq k\leq b'', 0\leq a\leq 3} \!\!\!\!\!\!   (-1)^{b'+k}  \frac{\n}{5}\Big(
    (R^X_k)_a             (R_j)_{\n+3-a}  -  Y (R^X_k)_a            (R_j)_{3-a} + ({Y}/{\ft^\n})  (R^X_k)_{\n+a}  (R_j)_{3-a}  \Big)
    $}}.
\end{align*}
where we denote by $(R^X_k)_i  :=  \big(\varphi^{a''}, (R^X_k)\sta  \phi_i \big)$, and $\Contr(1)$ is a sum of terms of form \eqref{Xbound1} in case (1) (thus lies in $Q[X]_{b'+b''+1})$.
By (2) and (3) the rest terms lie in $\mathbb Q[X]_{b'+b''\!+2}$.
Now we want to prove the top degree term indeed vanishes.
 The argument is similar to the one in the proof of  \cite[Appendix C]{NMSP2}. Recall \cite[(C.4)]{NMSP2}, for $a=\n,\cdots,\n+3$ we have  
\beq \label{keypropR}
 \Coeff_{X^{k+1}}  (R_k)_{a} = \frac{c_a'}{5^5}   \Coeff_{X^{k}} (R_k)_{a-\n},\quad
 (c'_a)_{a=\n}^{\n+3} =  (-120,-890,-2235,-3005).
\eeq  
Similar property holds for $R^X$ by using the explicit formula \eqref{RXz}  :  \footnote{ Note that $k<b''<\n-3$ implies $k\leq 3$, for the same reason as stated in footnote \ref{knomorethan3}.}
$$\textstyle \quad
\text{ for $a=\n,\cdots,\n+3$}\qquad  \Coeff_{X^{k}} \Big( q^{-1} \!\cdot (R^X_k)_{a} \Big) = {c_a'}  \cdot  \Coeff_{X^{k}}(R^X_k)_{a-\n}.
$$ 
Then we obtain for $a=0,1,2,3$ and for $j+k=b'+b''+1$ \begin{align*}
& \,\Coeff_{X^{j+k+1}} \Big( - Y\,(R^X_k)_{a}  (R_j)_{3-a} +(R^X_k)_{a}  (R_j)_{\n+3-a} +Y/\ft^\n (R^X_k)_{\n+a}  (R_j)_{3-a}  \Big)  \\
& \,=   \Big( 1+\frac{c'_{\n+a}}{5^5}+\frac{c'_{\n+3-a}}{5^5}\Big)\cdot \Coeff_{X^{b'+b''-1}} \Big(  (R^X_k)_{a}  (R_j)_{3-a}\Big)=0. \nonumber
\end{align*}
where we have used $Y=1-X$ and $5^5 Y q = t^\n X$. Hence the true degree in $X$ is decreased by $1$ and then we finish the proof.
\end{proof}

\subsection{Choice of gauge and finish the proof of Theorem \ref{thm2}} \label{gauge}
We consider the following symplectic transformation:
$$ \cG(z)^{-1} =I-
\begin{pmatrix}
0 & z\cdot c_{1a} &z^2\cdot c_2 &z^3\cdot c_3'\\
& 0& z\cdot c_{1b} & z^2\cdot c_2'\\
& & 0 & z\cdot c_{1a}\\
& & & 0\\
\end{pmatrix},
$$
where $c'_2 = -c_{1a}c_{1b} -c_2$ and $c'_3 = -c_{1a} c_2-c_3$. Then we are able to recover the family of $R$-matrices $R^{\bA,\cG}(z)$ defined in \eqref{RAG}  via 
\beq  \label{gaugeR}
R^{\bA,\cG}(z)^{-1}= R^{\bA,\vec{0}}(z)^{-1}   \cdot \cG(z)^{-1}
\eeq
where the family of propagators $E_{**}^\cG$ in  $R^{\bA,\cG}(z)$  is   related with the propagators  $E^{\vec 0}_{**}:=E_{**}^{\cG=\mathbf 0}$ in $R^{\bA,\vec{0}}(z)$ by the following
\begin{align*}
&E_{\bp}^\cG=  E^{\vec 0}_{\bp}+c_{1a},\qquad
E_{\tp\tp}^\cG =  E^{\vec 0}_{\tp\tp}+c_{1b},\qquad 
 E_{\tp\bp}^\cG =  E^{\vec 0}_{\tp\bp} - c_{1b} \,E^{\vec 0}_{\bp} -c_2  ,\\
&E_{\bp\bp}^\cG = E^{\vec 0}_{\bp\bp} + c_{1b}\,(E^{\vec 0}_{\bp})^2 -2 \, c_2  \,E_{\bp}   +c_3  .
\end{align*}


\begin{proof}[Proof of Theorem \ref{thm2}]
Recall we have proved (Proposition \ref{polyoffa})\footnote{ The assumption in the statement of Proposition \ref{polyoffa} is no longer needed after finishing the induction.}
\begin{align*} 
f^{\bA, \vec{0}}_{g;\ba,\bb} \in \mathbb Q[X]_{3g-3+n-\sum_i b_i }.
\end{align*}
Via \eqref{gaugeR}, we define the CohFT
$$
\Omega^{\bA,\cG}:=\cG.\Omega^{\bA,\vec{0}} = R^{\bA,\cG}.  \Omega^{Q, \tw} .
$$
Then we see the $A$-model master potential \eqref{AmodelMpotential} is indeed its generating function\footnote{ See more explanations at the end of this proof.}
\begin{align} \label{fAOmegaA}
f^{\bA,\cG}_{g;\ba,\bb}= (-5Y/t^\n)^{g-1} \,\int_{\M_{h,n}}  \psi_{1}^{b_1}\cdots\psi_{n}^{b_n}\cup \Omega_{g,n}^{\bA,\cG}(  \tp_{a_1},\cdots,  \tp_{a_n}).
\end{align}  
We claim that with the condition \eqref{conditionforcG}, the $\cG$-action will not change the polynomiality.
We can write down the graph sum formula for $\Omega^{\bA,\cG}:=\cG.\Omega^{\bA,\vec{0}}$ via the $\cG$-action. For each graph $\Gamma$, the contribution $\Cont_\Gamma$ to \eqref{fAOmegaA}  is given by the following construction
\begin{itemize}
\item at each leg $l$ with insertion $\tp_{a_l}\!\bp_l^{b_l}$, we put $
\sum_k{\cG}^*_k (-\psi_l)^k \tp_{a_l}\!\bp_l^{b_l}$;
\item at each edge $e=(v_1,v_2)$, we put
$$\textstyle
V^{\cG}(\psi_{v_1},\psi_{v_2}):=\frac{1}{\psi_{v_1}+\psi_{v_2}}(\id-{\cG}(\psi_{v_1})^{-1}{\cG}(-\psi_{v_2})) =Y^{-1}\cdot \sum_{k,l} V^{\cG}_{kl} \psi^k_{v_1}\psi^l_{v_2},
$$
where $\deg_X G^*_{k}=k$, $\deg_X V_{kl}=k+l+1$ and we have used $\varphi^i = (5Y/t^\n)^{-1} \varphi_{3-i}$ in the last equility.
\end{itemize}
The total factor involving $(-5Y/t^\n)$ in $\Cont_\Gamma$ is
$$ \textstyle
(-5Y/t^\n)^{-E}\prod_{v}(-5Y/t^\n)^{-(g_v-1)} = (-5Y/t^\n)^{-(g-1)},
$$
and the $X$-degree of total contribution of $\Cont_\Gamma$ is
\begin{align*}
& \textstyle \sum_v \big(3g_v-3+n_v-\sum_{l\in L_v} (k_l+b_l)-\sum_{e\in E_v} k_{(e,v)}\big)+\sum_{l\in L} k_l\\
&\qquad \qquad \textstyle +\sum_{e=(v_1,v_2)\in E} (k_{(e,v_1)}+k_{(e,v_2)}+1)= 3g-3+n -\sum_{l\in L} b_l.
\end{align*}
This proves
\begin{align}\label{BCOV-with-q'} f^{\bA,\cG}_{g;\ba,\bb} \in    \mathbb Q[X]_{3g-3+n-\sum_i b_i}. \end{align}
 
 
Pick $t$ such that $t^\n =-1$ and substitute it into \eqref{BCOV-with-q'}, then $q'=q$ and   
$\Omega^{Q,\tw,\tau_Q(q')}_{g,n} = \Omega^{Q,\tau_Q(q)}_{g,n}$ by Remark \ref{Q-to-tw}. By using the identification \eqref{twnormalized}, the definition \eqref{fAOmegaA} matches \eqref{AmodelMpotential}, and    
\eqref{BCOV-with-q'} becomes the statement of Theorem \ref{thm2}.
\end{proof}

\vspace{1cm}

\section{BCOV's Feynman graph sum via geometric quantization} \label{quantization}

In this section,  we view BCOV's Feynman graph sum as the  quantization of a symplectic transformation $R^\bB$, 
which is a restriction of our $A$-model propagator matrix $R^{\bA}$ in the smaller phase space.
 
\begin{convention}    \label{omitcG}  
In this and the next section, we will  omit the supscript $\cG$   in $\Omega^{\star,\cG}$, $R^{\star, \cG}$,$f^{\star,\cG}$, $E_{**}^{\cG}$, etc..
\end{convention}

\subsection{Quantization of the symplectic transformation in the small phase space } \label{RARB}
Let $\{ v_i \}_{i=0,1,2,3}  \!= \! \{ \tp_3 z^{-2} , -\tp_2 z^{-1} , \tp_1  , z \}$, with inner product given by
$$
 v_i \cdot  v_j :=  \frac{I_0^2}{5Y}\cdot\Res_{z=0} (v_i|_{z\mapsto -z} ,   v_j)   = {\footnotesize
 \left(\!\! \arraycolsep=1pt\def\arraystretch{0.4} \begin{array} {*{4}{@{}C{\mycolwda}}}
  &&& 1\\ && 1 \\&-1 \\ \  -1
  \end{array} \! \right) }  .
  $$
We consider the $4$-dimensional symplectic 
subspace
$$
H_{S}:= \sspan\{ v_i\}   \subset  \sH_Q[z,z^{-1}]\otimes \aA.
$$
By the explicit formula of the propagator matrix  $R^\bA$, we see
$$
R^\bA(z)  H_{S}  \subset H_{S}.
$$
Hence we can restrict the symplectic transformation $R^\bA(z)$ to subspace $H_S$, which is denoted by $R^\bB$. Under the symplectic basis $\{v_i \}_{i=0,1,2,3}$, we have
$$
R^\bB =  \begin{pmatrix} A & B \\C&D\end{pmatrix} := {\small  \left(\arraycolsep=1pt\def\arraystretch{0.8} \begin{array} {*{4}{@{}C{\mycolwdb}}}
1 &&  &\\
-E_{\bp} & 1&   & \\
-E_{\tp\bp}& -E_{\tp\tp}& 1 & \\
\ \  E_{1\bp^2} & \ \ E_{1,\tp\bp} &E_{\bp}  & 1\\
\end{array}  \!\!\!\!  \right) }.
$$

For a vector in $H_{S}$ under the symplectic basis $\{v_i\}$, we write it in the form
$$\vec v  = (\mathbf p, \mathbf x) = p_y v_0+ p_x v_1+x\, v_2+ y\, v_3 \in H_S$$
We define the quantization of the symplectic transformation $R^\bB$ as follows:
\begin{definition}
We introduce the following quadratic form over $H_S$:
\begin{align*}
\mathbf Q(\mathbf x,{\mathbf p}) = &\, (D^{-1} \mathbf x) \cdot {\mathbf p}'  - \frac{1}{2} (D^{-1}C {\mathbf p}') \cdot {\mathbf p}'
= \,   ({\mathbf p}')^t \begin{pmatrix} 1 & 0\\
 E_{\bp} &  1 \end{pmatrix}  \mathbf x   + \frac{1}{2} ({\mathbf p}')^t \begin{pmatrix}  E_{\tp\bp} & E_{\tp\tp}\\
 E_{\bp\bp} &  E_{\tp\bp} \end{pmatrix}  {\mathbf p}' .
\end{align*}
The quantization $\widehat R^\bB$ is defined via the following Feynman integral\footnote{ This is a finite dimensional Gaussian integral, hence it is well-defined.}
\begin{align}
(\widehat R^\bB F) (\hbar,\mathbf x) := \,&\ln \int_{\RR^2\times \RR^2} e^{ \frac{1}{{\hbar}} ( {\mathbf Q(\mathbf x,{\mathbf p}')  - \mathbf x' \cdot {\mathbf p}' }) +F(\hbar,\mathbf x') }  d\mathbf x'  d{\mathbf p}' .
\end{align}
 \end{definition}
The standard argument of Fourier transform  deduces  the following (we refer the reader to \cite[Sect.\,1.4]{CPS} for  detailed   discussion of the geometric quantization).
\begin{lemma}
We have the following operator form for $\widehat R^{\bB}$
\begin{align} \label{quantizRB}
(\widehat R^\bB F) (\hbar,\mathbf x)
= \,& \ln  \Big(  e^{{\hbar}\,V^B(\partial_{\mathbf x}, \partial_{\mathbf x}) } e^{F(\hbar, D^{-1} \mathbf x )}  \Big),
\end{align}
where the differential operator is defined by
\beq \label{VB} V^\bB(\partial_{\mathbf x}, \partial_{\mathbf x}) := -\frac{1}{2} \begin{pmatrix} \partial_x, \partial_{y}\end{pmatrix}\!\big( D^{-1} C  \big)\!\begin{pmatrix} \partial_x\\
 \partial_{y}\end{pmatrix} = \frac{1}{2} \begin{pmatrix} \partial_x , \partial_{y}\end{pmatrix} \begin{pmatrix}  E_{\tp\bp} &  E_{\bp\bp}\\
 E_{\tp\tp}&  E_{\tp\bp} \end{pmatrix}  \begin{pmatrix} \partial_x\\
 \partial_{y}\end{pmatrix}.
 \eeq
\end{lemma}

 Now our Theorem \ref{thm1} has an equivalent statement
 \begin{theorem}[BCOV's Feynman rule]  \label{quantizationR} 
Recall  that  $P^\bB(\hbar,x,y)$ is defined in \eqref{PBA}.
The quantization of $R^\bB$ acting on $P^{\bB}$ defines the $B$-model master potential function, which has the form
\beq  \label{fBeqRBPB}
f^\bB(\hbar,x,y) := \widehat R^\bB  P^\bB(\hbar,x,y)  =  \sum_{g,m,n} \hbar^{g-1} x^m y^n  \cdot f^\bB_{g,m,n},
\eeq
such that for each $(g,m,n)$, $f^\bB_{g,m,n}$ is a degree $3g-3+m$ polynomials in $X$.
\end{theorem}
\subsection{
Modified Feynman rule} \label{BmodelE}

We introduce the following modified $B$-model correlators
\begin{align}\label{tiPgmn}
\tilde P_{g,m,n}  =&\,  \big\langle (\tp -E_{\bp}\bp)^{\otimes m}, \bp^{\otimes n}\big\rangle^{Q,\bB}_{g,m+n}, 
\end{align}
and their generating function
\begin{align} \label{tiPB}
\tilde P^\bB(\hbar,x,y) :=&\, {\small \text{$ \sum_{g,m,n}$}} \hbar^{g-1} {\small \text{$\frac{ x^m y^n}{m!n!} $}}\cdot  \tilde P_{g,m,n}
 =\,   P^\bB(\hbar,x, y-E_{\bp}\,x).
\end{align}
It is not hard to see, if we replace $P_{g,m,n}$ in the BCOV's Feynman rule  by $\tilde P_{g,m,n}$, then the Feynman rule Theorem \ref{thm1} will still hold if we replace $\{E_{\tp\tp},E_{\tp\bp},E_{\bp\bp},E_\psi\}$ by \footnote{ To generalize Yamaguchi-Yau equations, similar modified propagators were defined in \cite{AL07} . }
\begin{align} \label{mpropagators}
&\tilde E_{\tp\tp} = \,  E_{\tp\tp},\quad
\tilde E_{\tp\bp} = \,  E_{\bp} E_{\tp\tp}+ E_{\tp\bp} , \\
&\tilde E_{\bp\bp} = \, E_{\bp} ^2 E_{\tp\tp} + 2  E_{\bp}  E_{\tp\bp} + E_{\bp\bp},\quad \tilde E_\psi = 0. \nonumber
\end{align}  
More precisely, the Feynman graph sum is given by the following quantization
\beq \label{modifiedR}
f^\bB(\hbar,x,y) = \widehat R^\bB|_{E_{**}\mapsto \tilde E_{**}}   \tilde P^\bB(\hbar,x,y).
\eeq
Indeed, the change of variables \eqref{tiPB} can be written as a quantization
 $$
 \tilde P^\bB(\hbar,x,y) = \widehat \E^\bB P^\bB(\hbar,x, y)
 $$
of the symplectic transformation $\E^\bB$ defined by\footnote{ 
We can see for this case $C=0$ and  by \eqref{VB} there is no edge contribution.}
\beq \label{EEB}
\E^\bB = {\small  \left(\arraycolsep=1pt\def\arraystretch{0.8} \begin{array} {*{4}{@{}C{\mycolwd}}}
1 &&  &\\
 -E_{\bp}& 1&   & \\
 & 0 & 1 & \\
\ \   & \ \   &E_{\bp}  & 1\\
\end{array}  \!\!\!\!  \right) }.
\eeq
Then the modified $B$-model propagator matrix
$R^\bB|_{E_{**}\mapsto \tilde E_{**}}$ is given by
$$
R^\bB|_{E_{**}\mapsto \tilde E_{**}} =
\tilde R^\bB:=  R^\bB \cdot (\E^\bB)^{-1},
$$
which matches \eqref{mpropagators}.

\vspace{1cm}

 \section{From $\nmsp$ Feynman rule to BCOV's Feynman rule}
  We have proved Theorem \ref{thm2} in \S \ref{NMSPruleproof} and established the $\nmsp$ Feynman rule.  In this section, we will prove the equivalence of $\nmsp$ Feynman rule and BCOV's Feynman rule (Theorem \ref{thm3}). This will finish the proof of the BCOV's Feynman rule.

Notice that the $A$-model state space $H_\bA$ has a higher dimension, with the $B$-model one $H_\bB$ as its subspace. In particular,   we have $3$ more extra propagators as edge contributions. We first deal with the edge that  contributes  a bivector $1\otimes \varphi_2$ (with propagator $E_{1\varphi_2}=E_\psi$).  The  idea is to consider the similar factorization of the symplectic transformation as in \S \ref{BmodelE}.

\subsection{Decomposition of $R^\bA$-matrix and modified quintic theory} \label{AmodelE}
We consider the following matrix factorization of $R^\bA$-matrix:
\beq \label{RAdecom}
R^{\bA}(z) =  \tilde R^{\bA}(z) \cdot  \E^{\bA}(z),
\eeq
where (recall $E_{1\tp_2} :=E_{\bp}$)
\beq \label{mOmegaQ}
 \E^{\bA}(z):= \id+    z
{\footnotesize  \left( \arraycolsep=1.4pt\def\arraystretch{0.7}\begin{array} {*{15}{@{}C{\mycolwdc}}}
 0&E_{1\tp_2} &  \\&0& 0 \\
 && 0& E_{1\tp_2}   \\&&&0
  \end{array}\right) } \in \End H_\bA.
\eeq  
The modified quintic CohFT is defined via
$$
\tilde \Omega^{Q}:= \E^{\bA}. \Omega^Q.
$$ Notice that here   $\tilde \Omega^Q$ theory depends on the choice of the gauge $\cG$. (Recall by Convention \ref{omitcG}, we always omit the supscript $\cG$ in this section.)

\begin{convention}
In this section, we will not distinguish the $\Omega^Q$ and the twisted theory $\Omega^{Q,\tw}$. We identify them by setting $t=1$ in this section.
\end{convention}

\begin{definition}
For  the following  coordinate
$$\bt =  x \,\tp_1	+y \, \tp_0\bp+a\, \tp_1\bp+b \,  \tp_0 \bp^2 +c \, \tp_0  \in H_\bA ,$$
we introduce modified normalized $A$-model  potential for  the  quintic $3$-fold
\begin{align}
& \tilde P^\bA(\hbar; \bt)=\tilde P^\bA(\hbar, x,y,a,b,c):= \sum_{g,n} \frac{\hbar^{g-1}}{n!} \frac{(5Y)^{g-1}}{I_0^{2g-2+n}} \int_{\M_{g,n}} \tilde\Omega^Q_{g,n}( \bt ^n). 
\end{align}
 In particular,   we define
$$
 \tilde P^\bA(\hbar; x,y):= \tilde P^\bA(\hbar;  x \,\tp_1	+y \, \tp_0\bp).
$$
\end{definition}


\begin{lemma} \label{SEDEfortQ}
String and dilaton equations hold for the theory $\tilde \Omega^Q$.
\end{lemma}
\begin{proof}
By the result of  \cite{Lee}, the $R$-matrix action preserve tautological equations. Hence the $\tilde \Omega^Q$ theory  satisfies   string and dilation equation as well.
\end{proof}

\begin{proposition}  \label{tPAB}
We have the following relation
\begin{align} \label{tPAeqtPB}
\tilde P^\bA(\hbar,x,y) 
 = \tilde P^\bB(\hbar,x,y)   -\ln(1-y).
\end{align}
\end{proposition}
\begin{proof}
By Lemma \ref{SEDEfortQ}, we can use dilaton equations to remove the $\varphi_0\psi$ insertions. 
Namely, both sides of \eqref{tPAeqtPB}  satisfy\footnote{ one can check that the $B$-model correlators satisfy dilaton equations directly.}
$$ \textstyle
\frac{\partial}{\partial y} f = \Big( 2 \hbar \frac{\partial}{\partial \hbar} +x\frac{\partial}{\partial x}+y\frac{\partial}{\partial y}\Big) f+\frac{\chi_Q}{24},\qquad  \chi_Q=-200.
$$
It suffices to prove
\beq  \label{tPAeqtPB0}
\tilde P^\bA(\hbar,x,0)
 = \tilde P^\bB(\hbar,x,0).
\eeq
Now we apply the graph sum formula to $\tilde \Omega^Q:= \E^\bA \Omega^Q$. Notice that when there is an insertion $\varphi_2={\small  \text{ $I_{0}I_{11}I_{22}H^2$}}$, the quintic correlators are zero unless $g=0$ (which is from degree $0$ contribution). It is not hard to see that in our case ({the leg insertions are all $\varphi_1$'s}),
the stable graph will contribute zero unless it is a loop with $l$-vertices:  at each vertex there is exactly one $\tp_1$ leg insertion and several $-E_{1\varphi_2} \varphi_0\psi$ insertions, at each edge the bivector is $E_{1\varphi_2} \varphi_0\otimes\varphi_2$. This only  contributes  to $g=1$ potential. Denoted by $P_1^E$ the generating function of such ``loop type" contribution, we have
$$
\tilde P^\bA(\hbar,x,0) = P^\bA(\hbar,x,-E_{1\varphi_2}x ) + P_1^E(x).
$$
By using the dilaton equation for each vertex\footnote{ Suppose there are $n_i$  $\psi$-insertions at the $i$-th vertices  ($i=1,\cdots,l$), by forgotting all  the $\psi$-insertions we get a factor $n_i!$.
}
of the ``loop type" graph, we obtain
 \beq 
 \begin{aligned}
P_1^E(x) = & \!\!\!\!\!\!\!\sum_{\Gamma  \ \text{is a loop with $l$  } \atop \text{  vertices and $n+l$ legs}}\!\!\!\!\! \frac{x^{l+n}}{l!n!} \frac{\Cont_\Gamma}{|\Aut \Gamma|} = \sum_{l>0}   \frac{(l-1)!}{l}( E_{1\varphi_2} x) ^l \prod_{i=1}^l \sum_{n_i\geq 0} (-E_{1\varphi_2} x)^{n_i}\\
 = & 
 - \ln\Big(1- \frac{E_{1\varphi_2} x}{1+E_{1\varphi_2} x}\Big) =  \ln(1+E_{1\varphi_2} x).
 \end{aligned}
\eeq
 In the second equality above we used that there are  $(l-1)!$ choices when we put $l$ different vertices in a loop.
Together with the following relations
\begin{align*}
\tilde P^\bB(\hbar,x,y)= & \  P^\bB(\hbar,x, y-E_{1\varphi_2}\,x), \quad \text{and}\quad
 P^\bA(\hbar,x,y)=  P^\bB(\hbar,x, y)-\ln(1-y),
\end{align*}
We obtain \eqref{tPAeqtPB0}, and hence finish the proof of this proposition.
\end{proof}

\begin{remark}
We can see the symplectic transformation \eqref{EEB} in \S \ref{BmodelE} is exactly the restriction of the $\E^\bA$-action to the B-model state space.
\end{remark}

Next, we will use  string equations proved in Lemma \ref{SEDEfortQ}, to write down any $\tilde \Omega^Q$-theory invariants in terms of $\tilde \Omega^Q$-theory invariants with only insertions $\tp$ and $\bp$. In this way, we deal with the  remaining   two ``extra" propagators.

\subsection{Modified propagators and  operator  formalism for the quantization action} \label{AmodelE}
By the definition of $\tilde R$-matrix and the $\tilde \Omega^Q$ (c.f. \eqref{RAdecom} and \eqref{mOmegaQ}), we see 
that the CohFT $\Omega^\bA$ is equal to the  $\tilde R^Q(z)$-action on the CohFT $\tilde \Omega^Q$:
\beq\label{tildeRactingonQ}
\Omega^A = \tilde R^Q . \tilde \Omega^Q.
\eeq
 Extending \S \ref{BmodelE},  \black for the edge contribution of $\tilde R^Q$-action,   we have the modified propagators
\begin{align} \label{mpropagatorsA}
& \tilde E_{\tp\tp} = \,  E_{\tp\tp},\qquad
\tilde E_{\tp\bp} = \,  E_{\bp} E_{\tp\tp}+ E_{\tp\bp} ,\qquad \nonumber\\
& \tilde E_{1,\tp\bp}  =  E_{1,\tp\bp} ,\qquad
\tilde E_{1\bp^2} = \, E_{1\bp^2} +   E_{\bp}  E_{1,\tp\bp},  \\
&  \tilde E_{\bp\bp} =  \, E_{\bp} ^2 E_{\tp\tp} + 2  E_{\bp}  E_{\tp\bp} + E_{\bp\bp} ,\qquad \tilde E_{\bp} = 0. \nonumber
\end{align}
  (Note $\ti E_{**}$'s are  $\ti E^{\cG}_{**}$'s defined via the same formulas.) \black
Using \eqref{mpropagatorsA}, we  write down the differential operator form of  $\nmsp$ $A$-model potential and BCOV's $B$-model potential.

\begin{proposition} \label{operatorform}
For $\star = \bA$ or $\bB$ and $u = x \tp_1+y\tp_0 \bp$, we have
\begin{align*}
\exp\big({ f^{\star}(\hbar,x,y)} \big)=\,& \exp \big({\hbar \cdot\tilde V^\star{(\partial_\bt,\partial_\bt)}} \big) \exp\big({ \tilde P^{\star}(\hbar; \E\bt)}\big) |_{\bt=R^\star(\bp)^{-1}u(\bp)}
\end{align*}
where the  $\tilde V$-operator is defined by
\begin{align*}{\small
\begin{aligned}
\quad \tilde V^\bB(\partial_\bt,\partial_\bt):=\, & \frac{1}{2}\tilde E_{\tp\tp} \frac{\partial^2}{\partial x^2}+{\tilde E_{\tp\bp}}\frac{\partial^2}{\partial x\partial y}+\frac{1}{2}\tilde E_{\psi\!\psi} \frac{\partial^2}{\partial y^2},\quad
\\
\tilde V^\bA(\partial_\bt,\partial_\bt) := \, &  \tilde V^\bB(\partial_\bt,\partial_\bt)+ \tilde V^E(\partial_\bt,\partial_\bt),
\quad  \,  \tilde V^E(\partial_\bt,\partial_\bt):={\tilde E_{1\!, \tp\bp}}\frac{\partial^2}{\partial a\partial c}+{\tilde E_{1 \psi^2}}\frac{\partial^2}{\partial b\partial c}.\quad
\end{aligned}}
\end{align*}
Here the operator $\tilde V^E$  corresponds   to edge contributions with extra propagators.
\end{proposition}
\begin{proof}
For the case $\star = \bA$, the formula follows from \eqref{fAOmegaA}, \eqref{tildeRactingonQ} and 
Givental's quantization formula \cite{Giv2}. For the case $\star=B$, the formula follows from the operator form of the $B$-model quantization formuma \eqref{quantizRB} and \eqref{fBeqRBPB}.
\end{proof}

\begin{lemma} \label{stringflow}
We have
$$
e^{\tilde P^\bA(\hbar, x,y,a,b,c)}  = e^{ \frac{c}{1-y} (a\frac{\partial }{\partial x}+b\frac{\partial }{\partial y}  )}  e^{\tilde P^\bA(\hbar, x,y) }.
$$
\end{lemma}
\begin{proof}
By string equations, we have\footnote{ Here since there is no $\varphi_2$-insertions, the unstable contribution does not appear in the equation.}
$$\textstyle
\frac{\partial }{\partial c} e^{\tilde P^\bA(\hbar, x,y,a,b,c)} =  \Big( a\frac{\partial }{\partial x}+b\frac{\partial }{\partial y} +y\frac{\partial }{\partial c} \Big) e^{\tilde P^\bA(\hbar, x,y,a,b,c) }.
$$
Then the  Lemma   follows from the initial condition
$$
\tilde P^\bA(\hbar, x,y,a,b,c) |_{c=0} =\tilde P^\bA(\hbar, x,y).
$$
This proves the  Lemma. 
\end{proof}
\subsection{Finish the proof of Theorem \ref{thm3}}\label{Fin-thm3} 
We first prove two identities.
\begin{lemma}
For any $f(x,y)$, the following identities hold:
\begin{align}
 \qquad e^{\hbar\, \tilde V^E(\partial_\bt,\partial_\bt) }  e^{\frac{c}{1-y}(a\frac{\partial}{\partial x}+b\frac{\partial}{\partial y})}   f(x,y)  \Big|_{a,b,c=0}
=\,& \  \sum_{k\geq 0} E(\partial_\bt)^k f(x,y),    \quad \label{identity1}
\\
e^{ - \hbar\,  \tilde V^\bB(\partial_\bt,\partial_\bt) }  (1-y)^{-1} e^{\hbar\,  \tilde V^\bB(\partial_\bt,\partial_\bt) }   (1-y) =\,& \  \sum_{k\geq 0} E(\partial_\bt)^k f(x,y),
\label{identity2}
\end{align}
where $E(\partial_\bt):={\small \begin{matrix} {\frac{ \hbar}{1-y}({\tilde E_{1\!, \tp\bp}} \frac{\partial}{\partial x}+{\tilde E_{1\psi^2}}\frac{\partial}{\partial y} )}  \end{matrix}}$.
\end{lemma}
\begin{proof}
For the first identity, we have that the LHS of \eqref{identity1}
\begin{align*}
&{\footnotesize
\begin{aligned}
 = &\, \sum_{n}  \frac{\hbar^n}{(n!)^2} \Big({\tilde E_{1\!, \tp\bp}}\frac{\partial^2}{\partial a\partial c}+{\tilde E_{1 \psi^2}}\frac{\partial^2}{\partial b\partial c}\Big)^n \Big(\frac{c}{1-y}(a\frac{\partial}{\partial x}+b\frac{\partial}{\partial y})\Big)^n  f(x,y)  \Big|_{a,b,c=0}
 \\
 = &\, \sum_{n}  \frac{\hbar^n}{ n! } \Big({\tilde E_{1\!, \tp\bp}}\frac{\partial}{\partial a }+{\tilde E_{1 \psi^2}}\frac{\partial}{\partial b }\Big)^n \Big(\frac{1}{1-y}(a\frac{\partial}{\partial x}+b\frac{\partial}{\partial y})\Big)^n  f(x,y)  \Big|_{a,b,c=0}
 \\
 = &\, \sum_{n}    \Big(\frac{\hbar^n}{1-y}(\tilde E_{1\!, \tp\bp}\frac{\partial}{\partial x}+\tilde E_{1 \psi^2}\frac{\partial}{\partial y})\Big)^n  f(x,y) .
 \end{aligned}}
\end{align*}
Here in the second equality we have used the following: when expanding the differential operators as power series, 
the contribution is non-zero only if $\tilde V^E(\partial_\bt,\partial_\bt)$ and ${\frac{c}{1-y}(a\frac{\partial}{\partial x}+b\frac{\partial}{\partial y})} $ appear in the form of the same powers.

For the second identity, by using $\tilde E_{\tp\bp}+\tilde E_{1\!, \tp\bp} = 0$,
${\tilde E_{\psi \psi}+\tilde E_{1 \psi^2}}= 0 $, and 
\begin{multline*}
e^{ -\hbar\, \tilde V^\bB(\partial_\bt,\partial_\bt) }  (1-y) e^{ \hbar\, \tilde V^\bB(\partial_\bt,\partial_\bt) }
=\, e^{\mathrm{ad}_{\hbar\, \tilde V^\bB(\partial_\bt,\partial_\bt)}} (1-y) \\
=\, (1-y) -[\hbar\, \tilde V^\bB(\partial_\bt,\partial_\bt),(1-y)] = (1-y) + \hbar\,{\small \begin{matrix} { ({\tilde E_{\tp\bp}} \frac{\partial}{\partial x}+{\tilde E_{\bp\bp}}\frac{\partial}{\partial y} )}  \end{matrix}},
\end{multline*}
we obtain $(1-y)^{-1} e^{ -\hbar\, \tilde V^\bB(\partial_\bt,\partial_\bt) }  (1-y) e^{ \hbar\, \tilde V^\bB(\partial_\bt,\partial_\bt) }  =\, (1- E(\partial_\bt)) f(x,y)$, which is equivalent to \eqref{identity2}.
\end{proof}

By the above two identities, we obtain the following key Lemma.
\begin{lemma}\label{keylemma1} For any $f(x,y)$ we have
\beq \label{keyidentity}
 (1-y) e^{ \hbar\, \tilde V^\bB(\partial_\bt,\partial_\bt)+\hbar\, \tilde V^E(\partial_\bt,\partial_\bt)}   e^{\frac{c}{1-y}(a\frac{\partial}{\partial x}+b\frac{\partial}{\partial y})}  {\small\text{$\frac{ f(x,y)}{1-y}$}}  \Big|_{a,b,c=0}\!\!
   =   e^{ \hbar\, \tilde V^\bB(\partial_\bt,\partial_\bt)} f(x,y).  \eeq
\end{lemma}
\begin{proof}
Since $\tilde V^\bB$ commutes with $\tilde V^E$, \eqref{identity1} and  \eqref{identity2} imply
\begin{align*}
\text{LHS} = \,&  (1-y)\, e^{\hbar\, \tilde V^\bB(\partial_\bt,\partial_\bt) } {\textstyle \sum_{k\geq 0}} E(\partial_\bt)^k   (1-y)^{-1} f(x,y)=\text{RHS}.
\end{align*}   
This proves the lemma.
\end{proof}

Now we finish the last step of the proof of Theorem \ref{thm3}.  By setting  $f(x,y) =  e^{\tilde P^\bB(\hbar, x,y)}$
in \eqref{keyidentity} and  by  using   Proposition \ref{tPAB},   we have
$$
 e^{\hbar\, \tilde V^\bB(\partial_\bt,\partial_\bt)+\hbar\, \tilde V^E(\partial_\bt,\partial_\bt)}   e^{\frac{c}{1-y}(a\frac{\partial}{\partial x}+b\frac{\partial}{\partial y})} e^{  \tilde P^\bA(\hbar,x,y)}  \Big|_{a,b,c=0}
   =   (1-y)^{-1} \,  e^{ \hbar\, \tilde V^\bB(\partial_\bt,\partial_\bt)} e^{\tilde P^\bB(\hbar,x,y)}.
$$
Then by Lemma \ref{stringflow}, the identity becomes
$$
 e^{\hbar\, \tilde V^\bB(\partial_\bt,\partial_\bt)+\hbar\, \tilde V^E(\partial_\bt,\partial_\bt)}   e^{\tilde P^\bA(\hbar, x,y,a,b,c) }  \big|_{a,b,c=0}
   =   (1-y)^{-1} \,  e^{ \hbar\, \tilde V^\bB(\partial_\bt,\partial_\bt)} e^{\tilde P^\bB(\hbar,x,y)}.
$$
Together with Proposition \ref{operatorform} we complete the proof.


\section{Reduction of generators, Yamaguchi-Yau's equations and examples}

 The modified propagators \eqref{mpropagators} and \eqref{mpropagatorsA} 
were introduced to remove the $(1,\varphi_2)$ edges in the $\nmsp$ rule  in order to prove Theorem \ref{thm3}. As a by-product,  we find that four specific modified propagators  give exactly Yamaguchi-Yau's generators, which   generate a subring  containing the normalized quintic potentials $P_{g>1}$.  

\begin{theorem}\label{YYeqn2}
We consider the following modified propagators as generators   \footnote{
Our generator $\VV_k$  is   related with the $v_i$ defined in  \cite{YY} as follows: 
$
v_1 = 
 - \VV_1, \ v_2 = 
- \VV_2,\ 
 v_3 =
 \VV_3 - X \,\VV_2.
$
In a sense, we give a geometric explanation
for Yamaguchi-Yau's generators $v_i$: they are edge contributions (propagators) of the modified Feynmann rule introduced in \S \ref{BmodelE}.
}
\beq \label{V123} \small
\begin{aligned}
& \VV_1 := \tilde E^{\vec 0}_{\tp\tp}=A+2B,\quad \VV_2 =  
 \tilde E^{\vec 0}_{\tp\bp}= -B_2+ B ( A+2 B) ,\quad\\
& \VV_3 =
  \tilde E^{\vec 0}_{\bp\bp}  = -B_3-(B+X)\, B_2  +(A+2B) \,B^2-\frac{2}{5} X\, B ,
\end{aligned} \eeq
and we introduce the subring which  is closed under the differential operator $D$:
$$
\tilde  \sR:= \mathbb Q[\VV_1,\VV_2,\VV_3,X] \subset \sR. 
$$
Then for $2g-2+m+n>0$ , the $\tilde P_{g,m,n} $ defined in \eqref{tiPgmn} lie in
$\sR$.
In particular, we have the reduction of generators which was originally conjectured in \cite{YY}: 
\beq  \label{YYHAE2a}
P_{g}  \in \tilde \sR  \quad  \text{ for  }  \  g>1
\eeq
 \end{theorem}

\begin{remark}
Notice that
\begin{align*}
& \tilde E_{\tp\tp}^\cG = \VV_1+ c_{1},\quad \tilde E_{\tp\bp}^\cG = \VV_2+	c_2,\quad
 \tilde E_{\bp\bp}^\cG =\VV_3+	c_3 .
\end{align*}
Hence the subring $\tilde \sR$ is also independent of the choice of gauge.
\end{remark}

\begin{proof}[Proof of Theorem \ref{YYeqn2}]
First, we prove $\tilde\sR$ is closed under $D$ \footnote{ This follows from a direct computation by using the relations \eqref{YYrelation}, which is proved in \cite{YY}.
See also \eqref{EQURA} which gives equivalent relations.}
$$
\text{ \small $D \VV_1 =-X\, \big(\VV_1-\frac{2}{5}\big)-\VV_1^2+2 \VV_2,\quad
D \VV_2 = -X \, \VV_2-\VV_1 \VV_2+\VV_3,\quad
D \VV_3 = \frac{24}{625} X-X \,\VV_3-\VV_2^2 $}.
$$
Next by using the dilaton equation,   $\tilde P_{g,m,n} = (2g-3+m+n)\tilde P_{g,m,n-1} $, we see
\beq\label{mdilaton}
\tilde P_{g,m} \in \tilde \sR\quad  \Rightarrow  \quad \tilde P_{g,m,n} \in \tilde \sR.\eeq
Now we prove $\tilde P_{g,m} \in \tilde  \sR$ by induction. Initially we have 
$$ \textstyle
\tilde P_{1,0,1} =  {\chi}/{24} -1  \and
\tilde P_{0,3} = 1 \quad \in  \quad \tilde \sR .
$$Assume $\tilde P_{h,l} \in \tilde  \sR$ for $(h,l)<(g,m)$. By using the modified Feynman rule (see \S \ref{BmodelE}), for $2g-2+m>0$,   we have $f_{g,m}^{B,\vec{0}} \in \mathbb Q[X]_{3g-3+m}$ is equal to the sum over contributions of  stable graphs  $\Gamma \in G_{g,m}$.

Except for the ``leading graph" (which has a single genus $g$ vertex with $m$-legs), the vertices in the other graphs all satisfy $(g_v,n_v)<(g,m)$. By induction assumptions and \eqref{mdilaton}, these vertices contributions $\tilde P_{g_v,m_v,n_v} $ all lie in the ring $\tilde\sR$. Together with that the edge contributions $\VV_k \in \tilde \sR$ for $k=1,2,3$, we deduce  $\tilde P_{g,m} \in \sR$ and  finish the induction.
\end{proof}
 
\begin{theorem}\label{YYeqn12}
The Yamaguchi-Yau equations hold:
\begin{align}  \label{YYHAE1}
&\qquad\qquad\qquad-\partial_{A} P_{g} = {\small \text{$  \frac{1}{2}$}} P_{g-1,2}  + {\small \text{$ \frac{1}{2} \sum_{g_1+g_2=g}$}} P_{g_1,1}P_{g_2,1}  ,\\
& \Big(-2 \partial_{A}+\partial_{B} + (A+2B)\partial_{B_2} + \big( (B-X)(A+2B) -B_2-{\textstyle\frac{2}{5}}X\big)\partial_{B_3}  \Big) P_{g} = 0    \label{YYHAE2}.
\end{align}
 Indeed, 
the second equation \eqref{YYHAE2}
 is equivalent to the reduction of generators \eqref{YYHAE2a}. 
\end{theorem}
\begin{proof}
In the end, we prove \eqref{YYHAE1}.
By using Theorem \ref{quantizationR}, the definition of $V^\bB$ \eqref{VB} and the definition of  quantization action \eqref{quantizRB}  we have
\beq \label{Ronf}
\exp {  P^\bB(\hbar,x,y)}= e^{-{\hbar}\,  V^\bB(\partial_{\mathbf x}, \partial_{\mathbf x}) }  \exp f^\bB(\hbar,x,y) .
\eeq
Note both sides lie in the ring $ \sR[[\hbar, \hbar^{-1}, x,y]]$. By applying  the partial derivative $\partial \in \sspan\{ \partial_A, \partial_B, \partial_{B_2}, \partial_{B_3}\}$ on both sides of \eqref{Ronf}, we see
\beq  \label{PDEforM}
- \partial  P^\bB(\hbar,x,y) \exp { P^\bB(\hbar,x,y)}={\hbar}\,\partial   V^B(\partial_{\mathbf x}, \partial_{\mathbf x})  e^{-{\hbar}\, V^B(\partial_{\mathbf x}, \partial_{\mathbf x}) }  \exp f^\bB(\hbar,x,y) ,
\eeq
where we have used $[\partial  V^B, V^B]=0$,  $\partial f^\bB=0$, and we recall
$$ \textstyle
V^\bB(\partial_\bt,\partial_\bt):=\,  \frac{1}{2} E_{\tp\tp} \frac{\partial^2}{\partial x^2}+{ E_{\tp\bp}}\frac{\partial^2}{\partial x\partial y}+\frac{1}{2}  E_{\psi\!\psi} \frac{\partial^2}{\partial y^2}$$
with $E_{**}$ defined in \eqref{bcovgens}.
We claim  \eqref{PDEforM}  will give us PDEs for $P_{g,m,m}$:  Let $\partial=\partial_A$ we have $\partial_{A}   V^B = \frac{1}{2} \partial_x^2$. Then \eqref{PDEforM} becomes the following PDE
\beq
-\partial_{A} P^\bB(\hbar,x,y)  = \frac{1}{2} \partial_x^2 P^\bB(\hbar,x,y)+ \frac{1}{2}  \Big(\partial_xP^\bB(\hbar,x,y) \Big)^2 .
\eeq
In particular by setting $x=y=0$, for $g>2$ the coefficient of $\hbar^{g-1}$ gives exactly \eqref{YYHAE1}. Let $\partial$ be the differential operator on the LHS of \eqref{YYHAE2}, we see it kills $V^\bB$. By using similar argument, 
we deduce \eqref{YYHAE2}. 
\end{proof}

The proof of Theorem \ref{YYeqn2} indeed gives another algorithm which computes the genus $g$ potential $P_g$ recursively from the lower genus potentials, by using the modified Feynman rule \eqref{modifiedR}.  The advantage of this algorithm is that only four generators/propagators (instead of five) are involved, expressing  $P_g$ in simpler terms. 

For any $g>1$, suppose the  master potential is given by 
$$ \textstyle
f^{\bA,\vec 0}_g =f^{\bB,\vec 0}_g =f_g(X) := \sum_{k=0}^{3g-3} f_{g,k} X^k,
$$
then one can solve the genus $g$ ``normalized" GW potential $P_g$ from the low genus by using (either $\nmsp$ or BCOV's, modified or original) graph sum formulae. 
\begin{example}
In terms of the generators \eqref{V123}, a maple program gives
\begin{align}
\ \text{\small $P_2$}= & \ \   \text{\footnotesize  ${\frac {
350\,{\VV_3}}{9}}+{\frac {25\,\VV_1\,\VV_2}{6}}+{\frac {5\,\VV_1^{3}}{24}}+{\frac {625\,\VV_2}{36}}+{\frac {25\,\VV_1^{2}}{24}}
+{\frac {25\,X\,
\VV_2}{36}}+\frac{X\,\VV_1^{2}}{6}+{\frac {13\,{X}^{2}\VV_1}{
288}}+{\frac {167\,X\,\VV_1}{720}}+{\frac {625\,\VV_1}{
288}}$}  \nonumber
\\
&\text{\footnotesize $
\quad +f_2(X), $} \quad \text{  with   \footnotesize  $f_2(X)=-\frac{1}{240}X^3-\frac{41}{3600}X^2+\frac{5759}{3600}X-\frac{25}{144}$;} \label{P2formula}
\\
\ \text{\small $P_3$}= & \ \  \text{\footnotesize $ {\frac {8225\VV_3^{2}}{27}}+{\frac {275\VV_1 \!\VV_2{
\VV_3}}{3}}+{\frac {29375\VV_2\VV_3}{108}}+{\frac {185 
\VV_1^{3}\VV_3}{24}}+{\frac {575\VV_1^{2}\VV_3}{24}}+{
\frac {29375\VV_1\VV_3}{864}}-{\frac {10450\VV_2^{3}
}{81}}-{\frac {3595\VV_1^{2}\VV_2^{2}}{72}}$} \nonumber\\
& \text{\footnotesize $ -{\frac {3575
\VV_1\VV_2^{2}}{54}}
+{\frac {14375\VV_2^{2}}{288}}-{
\frac {35\VV_1^{4}\VV_2}{3}}-{\frac {4075\VV_1^{3}{
\VV_2}}{144}}-{\frac {8125\VV_1^{2}\VV_2}{432}}+{\frac {
15625\VV_1 \!\VV_2}{1728}}-\frac{5 \VV_1^{6}}{4}-{\frac {25 
\VV_1^{5}}{6}}-{\frac {3125\VV_1^{4}}{576}} $}\nonumber\\
&\text{\footnotesize $-{\frac {15625
\VV_1^{3}}{5184}}+X \!\!\cdot \!\!\Big( {\frac {1175\VV_2\!{\VV_3}}{108}}+{\frac {39\VV_1^{2}{
\VV_3}}{8}}+{\frac {7849\VV_1\!{\VV_3}}{2160}}-{\frac {1397{
\VV_1}\!\VV_2^{2}}{54}}+{\frac {2773\VV_2^{2}}{2160}}-{
\frac {1687\VV_1^{3}\VV_2}{144}}-{\frac {16163\VV_1^{
2}\VV_2}{1080}} $}\nonumber\\
&\text{\footnotesize $-{\frac {21433\VV_1\VV_2}{8640}}-{\frac {
23\VV_1^{5}}{12}}-{\frac {3107\VV_1^{4}}{720}}-{\frac {
5893\VV_1^{3}}{1728}}-{\frac {82091\VV_1^{2}}{86400}} \Big)+
X^2 \!\! \cdot \! \Big(
{\frac {611\VV_1{\VV_3}}{864}}-{\frac {1603\VV_2^{2}}{
864}}-{\frac {1897\VV_1^{2}\VV_2}{432}} $}\nonumber\\
&\text{\footnotesize $-{\frac {4363\VV_1\VV_2}{2880}}-{\frac {731\VV_1^{4}}{576}}-{\frac {14609
\VV_1^{3}}{8640}}-{\frac {51473\VV_1^{2}}{86400}}\Big)-X^3 \!\! \cdot  \!\Big({\frac {325\VV_1{\VV_2}}{576}}+{\frac {2305{\VV_1}^{3}}{
5184}}+{\frac {4337{\VV_1}^{2}}{17280}}
\Big)
-D^2 P_2 \! \cdot  \! \frac{\VV_1}{2} $}\nonumber\\
&\text{\footnotesize $+ DP_2 \! \cdot \! \Big({\frac {19{\VV_2}}{3}}+\frac{\VV_1^{2}}{2}+{\frac {25\VV_1}{12
}}-{\frac {11X\,\VV_1}{12}} \Big) +
{\frac {47{\VV_3}}{3}}+\VV_1{\VV_2}+{\frac {25{\VV_2}}{6}}
+\frac{{X}^{2}\VV_1}{12}+X\! \left(\! \frac{13 {\VV_2}}{2}+\frac{\VV_1^{2}}{2}+{
\frac {19{\VV_1}}{12}} \!\right) $}\nonumber\\
&\text{\footnotesize $\quad+f_3(X), $} \quad \text{  with   \footnotesize  $f_3(X)={\textstyle \sum_{i=1}^6} f_{{3,i}}{X}^{i} +\frac{125}{36288} $}.  \label{P3formula}
\end{align}
Here  $f_{g,0} = (5)^{g-1} N_{g,0}$ is computed by using \eqref{fg0} and the ambiguity polynomial $f_2(X)$ is deduced from the lower degree GW invariants computed in Appendix \ref{lowdegreeGW}.

These formulae \eqref{P2formula}, \eqref{P3formula} 
match  the physicists' predictions  \cite{BCOV,YY} for the potential functions of the quintic $3$-folds up to the ``ambiguity" $\{f_{g=3,k}\}$.
\end{example}

\vspace{1cm}
\begin{appendix}

\addtocontents{toc}{\protect\setcounter{tocdepth}{1}}

\section{Low degree GW-invariants}  \label{lowdegreeGW}
 
Recall $N_{g,d}$ are the genus $g$ and degree $d$ GW-invariants of quintic threefolds. 
The degree zero invariants are computed in \cite{FP} :
\beq  \label{fg0}
  N_{g,0} = \frac{(-1)^{g} \cdot \chi \cdot |B_{2g}|\cdot |B_{2g-2}|}{ 2\cdot2g\cdot (2g-2)\cdot (2g-2)!}
\eeq

In this appendix, we will show

\begin{proposition} \label{GW}The low degree genus two GW-invariants are given by
$$N_{1,1}=\frac{2875}{12},\quad N_{2,1}=\frac{575}{48},\quad N_{2,2}=\frac{5125}{2},\and N_{2,3}=\frac{7930375}{6}.
$$
\end{proposition}
 
{We let $Q\sub \PP^4$ be a general} quintic threefold. For a smooth curve $E\sub Q$, we denote by $N_{E/Q}$
the normal bundle of $E$ in $Q$, and call $E$ rigid if $h^0(N_{E/Q})=0$. 

We let $f: C\to Q$ be a stable map from a genus
2 curve $C$ to $Q$ of degree $d\le 3$. We let $E=f(C)$ be the image curve.

\begin{lemma}
Let the notation be {as stated}. Then $E$ either is a smooth rigid rational curve or a smooth rigid elliptic curve.
\end{lemma}

\begin{proof}
Because $\deg f\le 3$, the image curve $E$ has degree at most three. In case $E$ is a union of rational curves, by
\cite{Ka,JK}, $E$ is irreducible, smooth and rigid. 

Now suppose $E$ contains an elliptic curve. As elliptic curves in $\PP^4$ have
degree at least 3, {$E$ is irreducible and has degree three}. Thus $E$ must be an irreducible component of $Q\cap L$, the intersection of $Q$
with a plane $L\sub\PP^4$. This way, $Q\cap L=E\cup E'$, where $E'$ is rational and of degree 2. 
By the rigidity proved in
\cite{Ka,JK}, {there is no infinitesimal deformation of $E'$ in $Q$. As $E'$ determines $L$, 
there is no infinitesimal deformation of $E$ in $Q$, thus $E\sub Q$ is rigid.}
\end{proof}

We recall the following results from \cite{FP, Pand3}. We let $C_0(h,d)$ be the contribution to $N_{h,de}$ from a rigid degree $e$
smooth rational curve $E\sub Q$. Then for any $d\ge 1$,
$$
\sum_{h=0}^\infty  C_0(h,1)\, t^{2h}=\Bigl(\frac{\sin(t/2)}{t/2}\Bigr)^{-2}
\and C_0(h,d)=d^{2h-3}C_0(h,1) .
$$
In particular, for any $d\geq 1$,
$
C_0(1,1)=\frac{1}{12 d}$ and $C_0(2,d)=\frac{d}{240} $.

We let $C_1(h,1)$ be the contribution to {$N_{1+h,e}$} from a rigid degree $e$
smooth elliptic curve $E\sub Q$. Then 
$$C_1(h,1)=0.
$$

\begin{proof}[Proof of Proposition \ref{GW}] By multiple cover formula of $N_{0,d}$, and the
known $N_{0,d\le 3}$, we see that the general quintic $Q$ has exactly $n_1=2,875$, $n_2=609,250$
and $n_3=317,206,375$ many degree one, two and three rational curves, all rigid, smooth, and mutually
disjoint. Applying the proceeding arguments, we get
$$N_{2,1}=n_1C_0(2,1),\quad N_{2,2}=n_1C_0(2,2)+n_2 C_0(2,1),\quad
N_{2,3}=n_1 C_0(2,3)+n_3 C_0(2,1).
$$
 Plugging the numbers,
we get $N_{2,1}=\frac{575}{48}$,
$N_{2,2}=\frac{5125}{2}$, and $N_{2,3}=\frac{7930375}{6}$.
We obtain $N_{1,1}= \frac{2875}{12}$ for the same reason.
\end{proof}

\vspace{0.5cm}

\section{Original forms of Feynman rules in the paper of BCOV}
The original form of Feynman graphs in \cite{BCOV} took a slightly different shape of edges, with certain freedom of gauges. We present BCOV's original form, and the generalization with insertions in the original style  in this section, for the readers who are more familiar with the B-model theory. We also give $g=1$
and $2$ examples in the original forms. 

\subsection{Original statement of BCOV's Feynman rule}  \label{origBCOV}
In \cite{BCOV}  the authors considered all $g$ B-model topological partition function
$\mathcal F^Z_g(q, \bar q)$ 
for an arbitrary compact Calabi-Yau threefold $Z$.
Its definition uses path integral, and it is a non-holomorphic extension of the GW potential 
$F^{Y}_g(q)$ of the mirror Calabi-Yau threefold $Y$ of $Z$:
\begin{align}\label{limitBA} \lim_{\bar q \rightarrow 0} \mathcal F^Z_g(q, \bar q) =F^{Y}_g(q).
\end{align}

One of the primary result in \cite{BCOV} is that $\mathcal F^Z_g$ satisfies ``holomorphic anomaly equation"(HAE). For the case  of  the quintic threefold $Z$,  it is
\begin{align*} 
\partial_{\bar q}\mathcal F^Z_g(q, \bar q) = \frac{1}{2} C^{qq}_{\bar q}\Big( D_{q}^2 \mathcal F^Z_{g-1}(q, \bar q) +\sum_{g_1+g_2=g}D_{q}\mathcal F^Z_{g_1}(q, \bar q) D_{q}\mathcal F^Z_{g_2}(q, \bar q)  \Big),
\end{align*}
where $D_q$ is certain covariant derivative and $C^{qq}_{\bar q}$ is certain three point function 
{ (Yukawa coupling)} that can be calcuated by B side special geometry. 
 Using   integrations by parts, \cite[Sect.\,6]{BCOV} solves HAE and express its solutions $\cal F_g^Z$ via  
 Feynman rules. We state here the BCOV's Feynman rules for $F^{Y}_g$ in \eqref{limitBA}, with $Y$ being the quintic $3$-fold.

\medskip
\noindent {\bf BCOV's Feynman graph:} 
{\sl For any $g>1$, we consider the set $G_{g}^{\BCOV}$ of genus $g$ stable graphs with three types of edges: solid lines; half dotted half solid lines,
and dotted lines. For each graph $\Gamma$, we do the following:

{\bf Edge}: at each edge drawn as solid lines, half dotted  lines and dotted lines, we place one of the  opagators $(T^{\tp\tp},T^{\tp},T)$ defined in \eqref{BCOVgauge}  respectively;

{\bf Vertex}: at each vertex of genus $g$, with $m$ solid  half edges and $n$ dotted half edges,
we place $P_{g,m,n}$ (defined in \eqref{Pgmn}).

\smallskip
We define $\Cont_\Gamma$ to be the product of the edge and the vertex placements; and define
$$
f_g^{\BCOV}:= \sum_{\Gamma \in G_{g}^{\BCOV}}    \frac{1}{|\Aut (\Gamma)|}   \Cont_\Gamma.
$$

 \begin{conjecture}\label{bcov1}
For $g>1$, $f_g^\BCOV$ is a degree $3g-3$ polynomial in $X$.
\end{conjecture}}

 This original BCOV's rule can be generalized  to allow legs: 
\smallskip

{\sl \noindent {\bf   BCOV's   Feynman graph with legs:} We consider the set $G_{g,n}^{\BCOV}$ of genus $g$, $n$-leg stable graphs  with three types of edges (as above) and two types of legs:  solid half lines and dotted half lines. Besides what we do for edges and vertices as above,  furthermore

{\bf Leg}: at each leg, we place one of the following $2$-types of propogators
\begin{align}  \label{BCOVpropagators2}
\begin{aligned}
\qquad \quad & {    E_{\tp} :=1 ,\quad \text{and}\quad 
 -E_{\bp}^{c_{1a}}:=- B - c_{1a}}\end{aligned} \quad
 \end{align}
according to the types of the edge: $\varphi$ goes with solid half line and $\psi$ goes with half dotted line. Here $c_{1a}$
can be any polynomial of $X$ with degree no more than $1$.

\smallskip
We define $\Cont_\Gamma$ to be the product of the legs, edges and vertices placements, and define
$$ f_{g,n}^{\BCOV}:= \sum_{\Gamma \in G^\BCOV_{g,n}}  \frac{1}{|\Aut (\Gamma)|} \Cont_\Gamma. 
$$
\begin{conjecture}\label{bcov2}
For $2g-2+n>0$, $f_{g,n}^\BCOV$ is a degree $3g-3+n$ polynomial in $X$. 
\end{conjecture}}

By setting $m=0$ and picking
the gauge $(c_{1b} , c_2, c_3) = (\frac{3}{5},-\frac{2}{25}, - \frac{4}{125})$ 
in Theorem \ref{thm1},  we recover the statement in Conjecture \ref{bcov2};  furthermore  by setting $n=0$ we recover the statement in Conjecture \ref{bcov1}.


\subsection{Example of BCOV's  original Feynman rule} \label{examples}
 We illustrate how   BCOV's Feynman rules  compute  genus $g$ GW potential from lowers genus GW potentials. 
\begin{example}[$g=1,n=1$] \label{g1n1} In this case, the BCOV's Feynman graphs are
$$
 \xy
   (11,0)*+{\displaystyle{{}_{E_\varphi}}};
  (15,0); (20,0), **@{-}; (19.8,-0.2)*+{\bullet};
  (21,-2)*+{\displaystyle{{}_{g=1}}};
\endxy \quad;    \qquad\qquad
 \xy
   (10,0)*+{\displaystyle{{}_{-E_\psi}}};
  (15,0); (20,0), **@{.}; (19.8,-0.2)*+{\bullet};
  (21,-2)*+{\displaystyle{{}_{g=1}}};
\endxy \quad;   \qquad  \qquad
 \xy
   (10,0)*+{\displaystyle{{}_{E_\varphi}}};
 (14,0); (19,0), **@{-}; (19.2,-0.2)*+{\bullet};
  (24,-2)*+{\displaystyle{{}_{g=0}}};
(23.6,0)*++++[o][F-]{}  ;
\endxy\quad .
$$
BCOV's rule gives us (note by definition \eqref{Pgmn},  
$P_{1,0,1} = \frac{\chi}{24}-1$.)
\beq
E_\varphi \cdot P_{1,1}+ (-E_\psi)\cdot P_{1,0,1} + \frac{1}{2} E_\varphi \cdot T^{\tp\tp}\cdot P_{0,3} = f_{1,1}^\BCOV(X) \  \in \  \mathbb Q[X]_1.
\eeq
By using the initial data $N_{1,0}, N_{1,1}$ (see Appendix \ref{lowdegreeGW}), and setting $c_{1a}=0$ we obtain
$$ \textstyle
f_{1,1}^\BCOV(X) = -\frac{1}{12}X-\frac{107}{60} = P_{1,1} - \frac{28}{3}\cdot (-B)+\frac{1}{2} \Big( A+ 2B+\frac{3}{5} \Big).
$$
Hence we solve $P_{1,1}$ that matches Zinger's formula \cite{Zi}  (also c.f. \cite{KLh}, \cite{CGLZ}) 
\beq \textstyle \label{P11formula}
P_{1,1} = -\frac{1}{2}A-\frac{31}{3}B-\frac{1}{12}X-\frac{25}{12}.
\eeq
\end{example}

\begin{example}[$g=2,n=0$]\label{g2n2}
In this case, the BCOV's Feynman rule becomes  {\small
\begin{align}
 &P_2+\frac{1}{2}\, T^{\tp\tp} P_{1,1}^2+\frac{1}{2}\, T^{\tp\tp} P_{1,2} +\frac{1}{2}\, (T^{\tp\tp})^2 P_{1,1}+\frac{1}{8}\, (T^{\tp\tp})^2 P_{0,4}+\frac{1}{8}\, (T^{\tp\tp})^3+\frac{1}{12}\, (T^{\tp\tp})^3  \quad
 \nonumber \\
&\qquad  \qquad+\frac{\chi  }{24} \,T^{\tp} P_{1,1} + \frac{1}{2}\,\frac{\chi}{24}\,  T^{\tp} T^{\tp\tp}+\frac{1}{2}\,\frac{ \chi }{24}\,  (\frac{\chi}{24} -1) T    \,= \, f^\BCOV_2(X) \ \in \  \mathbb Q[X]_3  , \label{FeynmanG2} 
\end{align} }
accroding to the BCOV's Feynman graphs  listed below: 
\vskip-20pt
{\footnotesize
\begin{align*}\\
  \xy
 (19.8,-0.2)*+{\bullet};
 (21,-2)*+{\displaystyle{{}_{g=2}}} \endxy
 &   \qquad F_2 \,,  \\[3mm]
  \xy
 (9.5,0); (20,0), **@{-}; (20,-0.2)*+{\bullet};
(11,-2)*+{\displaystyle{{}_{g=1}}}; (21,-2)*+{\displaystyle{{}_{g=1}}};
(10.2,-0.2)*+{\bullet};
\endxy&
  \qquad  \frac{1}{2}\, F_{1,1}^2\cdot T^{\tp\tp} \,,
  & \xy
  (21,-0.2)*+{\bullet};
(16,-2)*+{\displaystyle{{}_{g=1}}};
(24.6,0)*++++[o][F-]{}
\endxy&
  \qquad \frac{1}{2}\,F_{1,2}\cdot T^{\tp\tp} \,,
\\[3.5mm]
 \xy
 (9.5,0); (15,0), **@{-};  (15,0); (20,0), **@{.}; (19.8,-0.2)*+{\bullet};
(11,-2)*+{\displaystyle{{}_{g=1}}}; (21,-2)*+{\displaystyle{{}_{g=1}}};
(10.2,-0.2)*+{\bullet};
\endxy   & \qquad F_{1,1}\cdot T^{\tp} \cdot  \Big(\frac{\chi}{24}-1\Big)\,,
& \xy
 (9.5,0); (20,0), **@{.}; (19.8,-0.2)*+{\bullet};
(11,-2)*+{\displaystyle{{}_{g=1}}}; (21,-2)*+{\displaystyle{{}_{g=1}}};
(10.2,-0.2)*+{\bullet};
\endxy&
\qquad  \frac{1}{2}\, \Big(\frac{\chi}{24}-1\Big)^2  \cdot T\,,
\\[3.5mm]
\,\,\quad
 \xy
  (20.2,-0.2)*+{\bullet};
   (20,-0.2)  ;   (27.5,-0.2)  ,{\ellipse_{}},{\ellipse^{.}};
(16,-2)*+{\displaystyle{{}_{g=1}}}
\endxy &  \qquad F_{1,1}\cdot T^{\tp} \,,\qquad\qquad
& \xy
  (20.8,-0.2)*+{\bullet};
(16,-2)*+{\displaystyle{{}_{g=1}}};
(24.6,0)*++++[o][F.]{}
\endxy&
\qquad   \frac{1}{2}\,\Big(\frac{\chi}{24}-1\Big)  \cdot T\,,
\\[3.5mm]
 \xy
 (9.5,0); (20,0), **@{-}; (20,-0.2)*+{\bullet};
(11,-2)*+{\displaystyle{{}_{g=1}}}; (24,-2)*+{\displaystyle{{}_{g=0}}};
(23.6,0)*++++[o][F-]{}  ; (10.2,-0.2)*+{\bullet};
\endxy   & \qquad \frac{1}{2}\,F_{1,1}\cdot (T^{\tp\tp})^2\cdot F_{0,3}\,,
 \quad
& \xy
 (9.5,0); (14.5,0), **@{.};  (14.5,0); (19,0), **@{-};  (19.8,-0.2)*+{\bullet};
(11,-2)*+{\displaystyle{{}_{g=1}}}; (24,-2)*+{\displaystyle{{}_{g=0}}};
(23.5,0)*++++[o][F-]{}  ; (10.2,-0.2)*+{\bullet};
\endxy &  \qquad\frac{1}{2}\, \Big(\frac{\chi}{24}-1\Big)\cdot T^{\tp}\cdot F_{0,3}\cdot T^{\tp\tp}\,,
\\[3.5mm]
  \xy
  (19.5,-0.2)*+{\bullet};
(16,-2)*+{\displaystyle{{}_{g=0}}};
(23,0)*++++[o][F-]{}  ;(15.8,0)*++++[o][F-]{};
\endxy &\qquad
\frac{1}{8}\,F_{0,4}\cdot (T^{\tp\tp})^2\,,
& \xy
(16.5,-2)*+{\displaystyle{{}_{g=0}}};
 (20,-0.2)  ;   (27.5,-0.2)  ,{\ellipse_{}},{\ellipse^{.}};
  (20,-0.2)*+{\bullet};
  (16.3,0)*++++[o][F-]{};
\endxy &\qquad
\frac{1}{2}\,F_{0,3}\cdot \! T^{\tp\tp}\cdot T^{\tp}\,,
\\[3.5mm]
 \xy
 (9.5,0); (14,0), **@{-}; (14.2,-0.2)*+{\bullet};
(18,-2)*+{\displaystyle{{}_{g=0}}};
(7,-2)*+{\displaystyle{{}_{g=0}}};
(17.8,0)*++++[o][F-]{}  ;(6.4,0)*++++[o][F-]{};(10,-0.2)*+{\bullet};
\endxy &
\qquad  \frac{1}{8}\, F_{0,3}^2 \cdot (T^{\tp\tp})^3
& \xy
 (10,0); (18,0), **@{-}; (17.4,-0.2)*+{\bullet};
(20,-2)*+{\displaystyle{{}_{g=0}}};(8.2,-2)*+{\displaystyle{{}_{g=0}}};
(14,0)*++++[o][F-]{}  ; (10.6,-0.2)*+{\bullet};
\endxy& \qquad \frac{1}{12}\, F_{0,3}^2 \cdot (T^{\tp\tp})^3
\end{align*}} 
\vspace{0.1cm}
\begin{center} The list of stable $g=2$ decorated graphs, thirteen of them.
\end{center}
\medskip

By using the genus $1$ formula \eqref{P11formula}, the divisor equation $P_{1,2}= (D -A) P_{1,1}$,  together with the initial data $N_{2,0},N_{2,1},N_{2,2},N_{2,3}$\footnote{ These are originally conjectured by physicists by using some ``boundary" behavior of $F_g$. A mathematical computation of them is put in Appendix \ref{lowdegreeGW}.}, 
one  obtains  
$$ \textstyle
f_2^\BCOV(X) =-{\frac {1}{240}\,{X}^{3}}+{\frac {113}{
7200}\,{X}^{2}}+ {\frac {487}{300}\,X}-{\frac{11771}{7200}}.
$$ 
Hence one  solves  from \eqref{FeynmanG2}
 {\footnotesize
\begin{multline}\nonumber
-P_2 =-{\frac {350\,{B_3}}{9}}-
 \Big( {\frac {25\,A}{6}}+{\frac {425\,B}{9}}+{\frac{625}{36}}
 \Big) {B_2}
 +{\frac {5\,{A}^{3}}{24}}+{\frac {65\,{A}^{2}B}{12}}+{\frac {1045\,A{B}^{2}}{18}}+{\frac {865\,{B}^{3}}{9}}\\
+{\frac{25}{144}}+\Big(\frac{{A}^{2}}{6}+{\frac {49\,AB}{36}}+{\frac {167\,A}{720}}+{\frac {37\,{B
}^{2}}{18}}-{\frac {1811\,B}{120}}-{\frac {475\,{B_2}}{12}}-{\frac{
5759}{3600}} \Big) X
 \\
 +{\frac {25\,{A}^{
2}}{24}}+{\frac {775\,AB}{36}}+{\frac {350\,{B}^{2}}{9}}+{
\frac {625}{288}\,(A+2B)}+ \left( {\frac {13\,A}{288}}+{\frac {13\,B}{144}}+{\frac{41}{3600}}
 \right) {X}^{2}+{\frac {{X}^{3}}{240}}.
\end{multline}}
This is exactly the formula in Theorem~\ref{genus2formula}. Here we just rewrite the propagators in terms of Yamaguchi-Yau's generators via \eqref{BCOVgauge}.
\end{example}

\vspace{0.5cm}

\section{Remarks on the $R$-matrix actions on CohFTs}
\subsection{Unit axiom}
 We prove that the $R$-matrix action  preserves  the unit axiom if it is invertible, as stated in Theorem \ref{PPZThm}.  
\begin{lemma}\label{unit-lem} Let $\Omega$ be a CohFT with the triple $(V,\eta,\mathbf 1)$. We consider another triple $(V',\eta',\mathbf 1')$ with $\dim_F V'=\dim_F V$, and a symplectic transformation $R(z)\in \End (V,V')\otimes \aA[\![z]\!]$ acting on $\Omega$. We have
$$
R. \Omega_{0,3}(\mathbf 1', \alpha,\beta) = \eta'(\alpha,  \beta).
$$
\end{lemma}
\begin{proof}
By definition of the $R$-matrix action, we have
\begin{align*}
R. \Omega_{0,3}(\mathbf 1', \alpha,\beta)  
= &\, \sum_{k\geq 0}  \frac{1}{k!}  (\pr_k)_* \Omega_{0,3+k}\Big(  R_0^{-1}\mathbf 1',   R_0^{-1}\alpha,   R_0^{-1}\beta,  (T_0 \bp )^k \Big), 
\end{align*}
where $T(z):=z\, \mathbf 1 -R^{-1}(z)\mathbf 1' = T_0 z+O(z^2)$, and hence the second equality holds simply for dimensional reason.

Let $\omega$  be the topological part of $\Omega$ (i.e. the part of degree zero classes, c.f. \cite{PPZ}).   By   the axiom of CohFT,   it  is uniquely determined by the quantum product.  In particular,   
$$
\omega_{0,n +2}(\tau_{\mathbf n}, \beta_1,\beta_2) =  {\textstyle  \sum_{\alpha} } \omega_{0,n+1}(\tau_{\mathbf n}, e_\alpha)  \cdot\omega_{0,3}(e^\alpha, \beta_1,\beta_2) =    \omega_{0,n+1}(\tau_{\mathbf n}, \beta_1*\beta_2), 
$$
where we have used the spliting axiom in the first equality and the definition of the quantum product in the second equality. Hence
\begin{align*}
R. \Omega_{0,3}(\mathbf 1', \alpha,\beta)  = &\,  { \textstyle \sum_{k\geq 0}  \frac{1}{k!} }  \omega_{0,3+k}\Big(\, R_0^{-1}\mathbf 1', \, R_0^{-1}\alpha, \, R_0^{-1}\beta, (\, T_0   )^k \Big)  (\pr_k)_*  (\psi_{4}\cdots \psi_{k+3}) \\
 = &\,   \omega_{0,3}\Big( { \textstyle \sum_k} \, R_0^{-1}\mathbf 1' * (\, T_0)^{*k}, R_0^{-1}\alpha, R_0^{-1}\beta \Big) \\
  =   &\,   \omega_{0,3}\Big(  \mathbf 1  ,   R_0^{-1}\alpha,  R_0^{-1}\beta \Big)   = \Big( R_0^{-1}\alpha,  R_0^{-1}\beta \Big), 
\end{align*}
where we have used $\int_{\M_{0,3+k}}{ \psi_1\cdots \psi_k} =k!$ and
$
{\textstyle \sum_{k\geq 0} } R_0^{-1}\mathbf 1' * (  T_0)^{*k} = (\mathbf 1 -   T_0)* {\textstyle \sum_{k\geq 0} }  ( T_0)^{*k} = \mathbf 1
$
in the third equality; and the fundamental class axiom in the last equality.
 Furthermore,  since $R_0$ is invertible, $R_0^{-1}$ is symplectic as well, hence we finish the proof.
\end{proof}

\subsection{Dilaton flow} \label{normalizedR}

 Let $\Omega$ be an arbitrary CohFT with triple $(V,\eta ,\mathbf 1)$ and the coefficient adic ring  $\aA=F[\![q]\!]$.  Let $R (z) \in  \End(V,V')\otimes \aA[\![z]\!]  $ be symplectic.
 
We consider an arbitrary nonzero ``scaling constant" $c \in 1+q\aA$,  and we let 
\begin{align}\label{tiRT} 
\tilde R^{-1}(z)=c^{-1} R^{-1}(z) \and \tilde {T}(z)=z( \mathbf 1-  {\tilde  R^{-1} (z)}  \mathbf 1' ). 
\end{align}

 For any $2g-2+n>0$, using $\pr_{1\ast} \psi_{n+1} = 2g-2+n$, we have 
 \begin{align}
 T_R \Omega_{g,n}(-)  & \textstyle =  \sum_{k\geq 0}\frac{1}{k!}  \pr_{k\ast} \Omega_{g,n+k}(-, T(\bp)^k) \nonumber \\
  & \textstyle   =  \sum_{\ell,m\geq 0}\frac{1}{\ell!m!}  \pr_{\ell+m\ast} \Omega_{g,n+\ell+m}(-,   [(1-c)\bp \mathbf 1]^\ell, [c \ti T(\bp)]^m)   \nonumber \\
&= \textstyle \sum_{m\geq 0 }\frac{1}{m!} c^{-(2g-2+n) }\pr_{m\ast} \Omega_{g,n+m}\Big(   (-),  \tilde T(\bp) ^m \Big). \label{TailD} \end{align}   
We see that, usually if \eqref{Tail-action} converges , then \eqref{TailD} converges as well.
For example, if \eqref{qT} holds , then $\tilde {T}(z)$ also lies in $z^2\aA[\![z]\!]\otimes V+q\, z\, \aA[\![z]\!]\otimes V$.
Then $c-1\in q\aA$ makes the infinite sum converges  in the  $q$-adic topology. 

\medskip

In the end,
we give an example that how the Dilaton flow relates the $R$-matrix actions with general $R_0$ to the one defined in \cite{PPZ} for the semi-simple cases.

\begin{example} \label{examplePPZ}
For a semi-simple CohFT $\Omega_{g,n}$, we can state Givental-Teleman's reconstruction theorem in a slightly different form: there exists an $R$-matrix such that
$$
\Omega  = R. (\Omega_{\pt}^{\oplus n})   ,\qquad  R=R_0+R_1z+\cdots \in \End F^{n} \otimes \aA[\![z]\!]
$$
where the state space of $\Omega$ is still $F^{n}$ as a linear space; the unit of $\Omega$ is also the same one:
$$ \textstyle
\mathbf 1:= \sum_{\alpha=1}^n  e_\alpha ,\qquad  e_\alpha  \text{ is the unit of each copy of } I_{\pt};
$$
and the pairing is different in general, which we will denote by $(\cdot , \cdot )^\tw$.

Indeed, since we require $R$ to be symplectic, the pairing of $\Omega$ is indeed determined by $R_0$ (note the pairing in $\Omega_{\pt}^{\oplus n}$ is the standard pairing of $F^n$) .
Let $c_\alpha :=   (e_{\alpha}, R_0^{-1}  \mathbf 1)$,  and 
$$\Psi :=\diag(\{c_\alpha^{-1}\}_{\alpha=1}^n),  
\qquad \bar e_\alpha : = \Psi  e_\alpha =  c_\alpha^{-1}  e_\alpha .$$ 
Then the inner product of $\Omega$ is given by
 $$( e_\alpha, e_\beta)^{\tw}:= \delta_{\alpha\beta} c_\alpha^{2}\qquad \text{  or } \qquad
 ( \bar e_\alpha, \bar e_\beta)^{\tw}:= \delta_{\alpha\beta}.
  $$
We define the normalized $R$-matrix  via  $$\tilde R(z)=  R(z) \Psi^{-1} =   \id +O(z),$$
 which is indeed the $R$-matrix defined in \cite{PPZ}. By using Dilaton flow, one checks 
 $$
 \Omega \  =  \  R.  \big( \Psi. (I_{\pt}^{\oplus n})\big)  \  = \  R. \omega, 
 $$
where the $\Psi$-matrix transforms the trivial CohFT $I_{\pt}^{\oplus n}$ with standard pairing, to the topological part $\omega$ of $\Omega$ with the twisted pairing $(,)^{\tw}$  
$$
\omega_{g,n}(e_{\alpha_1},\cdots, e_{\alpha_n}) =  \delta_{\alpha_1,\cdots,\alpha_n}  c_\alpha^{ -(2g-2)}.
$$
Here $\delta_{\alpha_1,\cdots,\alpha_n}=1$ if $\alpha_1=\cdots=\alpha_n$, otherwise it is zero.
\end{example}

\vspace{0.5cm}

\section{Explicit formulae for $R$-matrices} \label{explicitformulaR}
First we give the explicit formulae for the leading terms of $R^{[0]}$. We hope it will make the arguments in the \S3 and \S4 more clear, though we do not really use it  in our proof.
 \begin{lemma}\label{explicitR0}
We have the following explicit formula (with understanding that $R^{[0]}(z)$ is identity operator on odd classes \footnote{  in this paper all operators from $\sH_Q$ to $\sH$ or conversely $\sH$ to $\sH_Q$ are assumed to be identity on odd classes. Thus we only describe their action on even classes.  })
 \begin{align}  \label{explicitR0f}
   R^{[0]}(z)^* =
   &  \tiny\left( \arraycolsep=1.4pt\def\arraystretch{1.2} \begin{array} {*{15}{@{}C{\mycolwda}}}
 I_0& & \\&  I_0 I_{11}\\&&  I_0 I_{11} I_{22}&  \\&&&  I_0 I_{11}^2 I_{22}&
  \end{array} \  \  \right)
\! \cdot  \!  \left[  \left(  \arraycolsep=1.4pt\def\arraystretch{1.2}\begin{array} {*{4}{@{}C{\mycolwda}}*{5}{@{}C{\mycolwda}}cc}
1& & &&&   &-120q \\
 &  1& &&& &&  -890q \\
 &&  1& & &&& &  -2235q & \\
 &&&  1&  &\cdots    &  & &&  -3005q  &
  \end{array}\right)\right.  \nonumber\\
 & +z
 {\tiny \left(\!\!\!\! \arraycolsep=1.4pt\def\arraystretch{1.3}\begin{array} {*{10}{@{}C{\mycolwd }}cccc}
 0&B& &&&&& -q(890B\atop+120) \\&0& A+2B&&&&&& -2235q(A+2B)\atop-1010q\\
 && 0& B-X &&&&& & -3005q(B-X)\atop -120q \\&&&0&-2X&\qquad \cdots & \cdots&&&& &&&
  \end{array}\,\,\right) } \nonumber\\
 & +z^2
 {\tiny \left(\!\!\!\!\arraycolsep=1.4pt\def\arraystretch{1.4} \begin{array} {*{10}{@{}C{\mycolwd}}*{4}{@{}C{\mycolwda}}}
 0&0 & B_2&&&&&&-5q \,(447 B_2+\atop 202 B+24)  \\&0&0& (A+2B)(B-X)\atop-B_2-\frac{2}{5}X&&&&&&3005q\,\big(B_2+\frac{2}{5}X-\frac{226}{601} \atop -(A+2 B) (B+\frac{649}{601}-X)\big)&&&\\&& 0&0 &-(2B-2X\atop+\frac{7}{5})X\\&&&0 & 0&  {\tiny \text{$\!\!\!-6X+\frac{17}{5}$}} &\qquad \cdots&\cdots
  \end{array}\!\!\right)}\qquad\qquad \nonumber\\
&  +z^3  {\tiny \left.
\left(\!\!\!\! \arraycolsep=1.4pt\def\arraystretch{1.4} \begin{array} {*{9}{@{}C{\mycolwd}}*{3}{@{}C{\mycolwdb}}}
 0& 0 &0 & B_3&&&&&&&-5 q \, ( 601 B_3+649 B_2\atop +226 B+24)&  \\&0& 0 &0 &  2(A+2 B) (B-X+\frac{7}{10})\atop-2B_2-\frac{4}{5}X+\frac{2}{5}  \\&& 0& 0 &0 & -6(B-X+\frac{7}{10}) X \atop+\frac{17}{5} (B- X)+\frac{9}{5} \\&&&0 &0 &0 & \frac{2}{5}(60X^2\atop-66X+13)  &\qquad\cdots&\cdots
  \end{array}\!\!\!\right)\!\right]}\quad \nonumber\\
&  \quad +  O(z^{4}),
\end{align}
  under the basis $\{\phi_j\}_{j=0}^{\n+3}$ and $\{H^i\}_{i=0}^3$,
where $\cdots$ are all zeros.
 \end{lemma}
 \begin{proof}
By using the QDE \eqref{QDEforR0} of $R^{[0]}(z)$ and the initial data $R^{[0]}(z)^*\mathbf 1$ in
 \eqref{R0z}, $R^{[0]}(z)$ can be computed recursively. A direct computation shows this lemma.
 \end{proof}
 
 Next, we give the explicit formulae for the leading terms of $R^X(z)$, as defined in \eqref{0-X-A}.
 \begin{lemma}
 We have the following explicit form of $R^X(z)\sta: \sH\to \sH_Q[\![z]\!]\otimes A$ in terms of basis
$\{\phi_j\}$ for $\sH$ and $\{ \varphi_i\}_{i=0}^3$ respectively:
 \begin{align} \label{RXz}
\ R^X(z)\sta   =&
   \left( {\tiny\arraycolsep=1.4pt\def\arraystretch{1.2}\begin{array} {*{15}{@{}C{\mycolwda}}}
1& & &&& &  &-120q \\
 &  1& &&& &&&  -890q \\
 &&  1& & &&& &&  -2235q & \\
 &&&  1&  &\cdots &\cdots  &  & &&  -3005q  &
  \end{array} }\right) \\
 &+z\cdot
 \left(\!\!\!\!\!\!\!  {\tiny \arraycolsep=1.4pt\def\arraystretch{1.2}\begin{array} {*{8}{@{}C{\mycolwd
}}*{18}{@{}C{\mycolwdd
}}}
 0&0& &&& \cdots&-120q\\&0&0 &&& \cdots& & -1010q\\
 && 0& - X && \cdots&& & 5q(601X-649)  \\&&&0&-2X&\cdots &&&
  \end{array}\ }\right)  \nonumber\\
 &+z^2\!\! \cdot
 \left(\!\!\!\!\!\! {\tiny\arraycolsep=1.4pt\def\arraystretch{1.3} \begin{array} {*{7}{@{}C{\mycolwd}}*{1}{@{}C{\mycolwda}}cccc}
 0&0 & 0 &&&&\qquad\cdots&& { -120q }
  \\&0&0& -\frac{2}{5}X  &&&\qquad\cdots&&&{ \!\!\!-2q(601X-565)}\\
  && 0&0 & (10X-7) \atop \cdot  X/5  &&\qquad\cdots \\&&&0 & 0&   (30X-17) \atop \cdot  X/5&\qquad\cdots & &
  \end{array} } \right)\qquad\qquad \nonumber \\
& +z^3\!\! \cdot
\left(\!\!\!\!\!\! {\tiny \arraycolsep=1.4pt\def\arraystretch{1.3} \begin{array} {*{12}{@{}C{\mycolwd}}}
 0& 0 &0 & 0&&&&\,\,\quad\cdots& \qquad -120q
  \\&0& 0 &0 &{(2X-1)\atop\cdot 2X/5}&&&\,\,\quad\cdots
  \\&& 0& 0 &0 &  {-(30X^2-38Y\atop +9) \cdot X/5} &&\,\,\quad\cdots\\&&&0 &0 &0 & \!\!\!\!\!\!\!\!  -2 X (60X^2-\atop 66X+13)  /5  &\,\,\quad\cdots
  \end{array} \quad \ \  } \right)
 +O(z^4) , \nonumber
\end{align}
where $\cdots$ are all zeros.  
\end{lemma}
\begin{proof}
The $R^X$-matrix can be computed by using the algorithm introduced in the proof of Lemma \ref{RXprop}, which starts from $R^X(z)^* \phi_0 = 1+O(z^{\n-3})$ and computes $R^X(z)^* \phi_j (j>0)$ recursively by using the equation \eqref{QDEforRX}.
\end{proof}

\begin{corollary}\label{VXformula}
	Recall $
	V_X(z,w) :=   \frac{\sum_{i=0}^3\tp_i\otimes \tp^i-\sum_{j=0}^{\n+3} R^X\!(-z)^* \phi_j \otimes R^X\!(-w)^* \phi^j}{z +w} $,
	we have
		\begin{align} \label{VXz}
		& Y\cdot  V_X(z,w)=
		\\
		&{\footnotesize\left(\arraycolsep=1.2pt\def\arraystretch{1.1}
		\begin{array} {*{3}{@{}C{\mycolwddd}}c} -{\frac { \left( 24\,{w}^{2}-24\,zw+24\,{
					z}^{2} \right) X}{625}}&-{\frac { \left( 24\,w-24\,z \right) X}{625}}&
		\!\!\!\!\!\!\!\!\!\!\!\!\!-{\frac {24\,X}{625}} &0\\ \noalign{\medskip}{\frac { \left( 24\,w-24\,
				z \right) X}{625}}&-{\frac {202\,X}{625}}&\!\!\!\!\!\!\!\!\!\!\!\!\!0&0\\ \noalign{\medskip}-{
			\frac {24\,X}{625}}&0&\!\!\!\!\!\!\!\!\!\!\!\!\!0&0\\ \noalign{\medskip}0&0&\!\!\!\!\!\!\!\!\!\!\!\!\!0&0
		\end {array}
		\ \right)}+\sum_iO(z^{i})O(w^{\n-i}). \nonumber
		\end{align}
\end{corollary}
\begin{proof}
	By Lemma \ref{RXprop},  the coefficients of $z^iw^j$ for $i+j<\n$ is non-zero only when $i+j\leq 3$. Then the matrix can be computed directly by \eqref{RXz}.
\end{proof}

\vspace{1cm}

\section{List of symbols}
{
\begin{table}[H] \centering 
\begin{center}
 \noindent\begin{tabular}{cp{0.8\textwidth}}
 $\n$ & a prime that will be taken large \\
 $t_\alpha$  & $t_{\alpha}=- \zeta_{\n}^{\alpha} t$ for $\alpha \in [\n]$, where $\zeta_{\n}$ is the primitive $\n$-th root of unity \\
  $p$ & the equivariant  hyperplane class $c_1(\sO_{\PP^{4+\n}}(1))$    
 \\
 $H$ & hyperplane class of  the quintic $3$-fold $Q$ \\
  $F$ & the base field $F=\QQ(\ft)$ for all CohFTs\\
  $\aA$ & coefficient ring  $\aA=\QQ(\ft)[\![q]\!]$ of all cohomologies and CohFTs\\
  $\sH$ & the extended $\nmsp$ 
  state space with twisted inner product $(\, , \,)^\tw$\\
  $\{\phi_i\}$ & the basis $\{\phi_i:=p^i\}_{i=0}^4$ of $\sH$ with dual basis $\{\phi^{i,\tw}\}_{i=0}^4$.  \\
  $\{\varphi_i\}$ & the normalized basis $\{\varphi_i:=I_0\cdots I_{ii}H^i\}$ of $\sH_Q$ with dual basis $\{\varphi^i\}_{i=0}^3$
  \\
  $\Omega^X$ & the CohFT defined by the Gromov-Witten class of $X$\\
   $R(z)$ & the $R$-matrix action from $\Omega^{\aleph}$ (Defn.\,\ref{Omegaaleph}) to $\nmsp$-$[0,1]$-theory $\Omega^{[0,1]}$\\
    $R^{[0]}, R^{[1]}$ & the restriction of the $R$-matrix to $\sH_Q$ or $\sH_\npt$\\
   $\Omega^{[0]}, \Omega^{[1]}$ & the theory defined via $R^{[0]}/R^{[1]}$-action on quintic CohFT $\Omega^Q$\\
 $f^{[0]} , f^{[1]}$ &  the generating function for the $[0]/[1]$-theory \\
 $A_k,\!B_k$ & generators defined via quintic $I$-function, in particular $A:=A_1$, $B:=B_1$\\
  $H_\bA,H_\bB$& the state space for the $\bA/\bB$-model quintic  theory \\
    $R^{\bA}, R^{\bB}$ & the symplectic transformation from quintic to the $\bA/\bB$ master theory
    \\
  $P^\bA, P^\bB$  & the generating function for the $\bA/\bB$-model quintic  theory\\
   $f^\bA, f^\bB$ &  the generating function for the $\bA/\bB$-model master theory\\
\end{tabular}
\end{center}
\end{table}
}

\end{appendix}

\vspace{1cm}

\end{document}